\documentclass[11pt
]{article}
\usepackage{
mathrsfs,amsmath,amsthm,amssymb,amscd,longtable,array,graphicx,color,epic,eepic,mathrsfs, cite
}
\newcommand{\nc}{\newcommand}
\nc{\rnc}{\renewcommand}
\nc{\nn}{\nonumber}
\nc{\der}{{\partial}}
\rnc{\Im}{{\rm{Im}\,}}
\rnc{\Re}{{\rm{Re}\,}}
\nc{\db}{\displaybreak[0]\\}
\nc{\bra}{\langle}
\nc{\ket}{\rangle}
\nc{\bs}{\boldsymbol}

\DeclareMathOperator{\End}{End}

\newlength{\minitwocolumn}
\newcommand{\Z}{{\Bbb Z}} 
\newcommand{\C}{{\Bbb C}} 
\newcommand{\N}{{\Bbb N}} 
\newcommand{\FF}{{\Bbb F}} 

\newcommand{\ta}{\tilde{a}}

\newcommand{\cD}{{\cal D}}

\newcommand{\cA}{{\cal A}}
\newcommand{\cU}{{\cal U}}

\newcommand{\cI}{{\cal I}}
\newcommand{\cH}{{\cal H}}
\newcommand{\cN}{{\cal N}}
\newcommand{\cR}{{\cal R}}
\newcommand{\cP}{{\cal P}}
\newcommand{\cE}{{\cal E}}

\newcommand{\cV}{{\cal V}}

\newcommand{\cX}{{\cal X}}

\newcommand{\bR}{{\overline{R}}}
\newcommand{\hd}{\widehat{d}}
\newcommand{\hL}{\widehat{L}}

\newcommand{\tR}{{\widetilde{R}}}

\newcommand{\cK}{{\cal K}}

\newcommand{\la}{\lambda}
\newcommand{\La}{\Lambda}

\newcommand{\al}{\alpha}
\newcommand{\bet}{\beta}
\newcommand{\ga}{\gamma}

\newcommand{\ep}{\epsilon}
\newcommand{\vep}{\varepsilon}

\newcommand{\tS}{\widetilde{S}}

\newcommand{\txi}{\widetilde{\xi}}
\newcommand{\tW}{\widetilde{W}}
\newcommand{\tX}{\widetilde{X}}

\newcommand{\bA}{\overline{A}}
\newcommand{\bB}{\overline{B}}
\newcommand{\bC}{\overline{C}}

\newcommand{\bep}{\bar{\epsilon}}

\newcommand{\hell}{{\ell}}

\newcommand{\hf}{\widehat{f}}

\newcommand{\bL}{{\bar{L}}}

\newcommand{\bt}{{\bf t}}

\newcommand{\bfw}{{\bf w}}

\newcommand{\bfP}{{\bf P}}
\newcommand{\bfQ}{{\bf Q}}
\newcommand{\bfK}{{\bf K}}
\newcommand{\bfL}{{\bf L}}

\newcommand{\te}{{\Theta_p}}
\newcommand{\tes}{{\Theta}_{p^*}}

\newcommand{\bal}{{\boldsymbol{\al}}}
\newcommand{\bbe}{{\boldsymbol{\beta}}}

\newcommand{\bla}{{\boldsymbol{\la}}}
\newcommand{\bmu}{{\boldsymbol{\mu}}}
\newcommand{\bnu}{{\boldsymbol{\nu}}}

\newcommand{\bea}{\begin{eqnarray}}
\newcommand{\ena}{\end{eqnarray}}
\newcommand{\beit}{\begin{itemize}}
\newcommand{\enit}{\end{itemize}}

\newcommand{\be}{\begin{eqnarray*}}
\newcommand{\en}{\end{eqnarray*}}
\newcommand{\lb}[1]{\label{#1}}

\newcommand{\ds}[1]{{\displaystyle #1 }}

\newcommand{\id}{{\rm id}}
\newcommand{\Ad}{{\rm Ad}}

\newcommand{\wt}{{\rm wt}}

\newcommand{\ev}{{\rm ev}}

\renewcommand{\Im}{\mbox{Im}\,}

\newcommand{\qdet}{{q}\mbox{\rm -det}}

\def\infq4p#1{{(#1;q^4,p)_\infty}}

\newcommand{\tot}{\widetilde{\otimes}}

\newcommand{\mmatrix}[1]{\begin{matrix} #1 \end{matrix}}

\font\teneufm=eufm10
\font\seveneufm=eufm7
\font\fiveeufm=eufm5
\newfam\eufmfam
\textfont\eufmfam=\teneufm
\scriptfont\eufmfam=\seveneufm
\scriptscriptfont\eufmfam=\fiveeufm

\let\goth\mathfrak
\newcommand{\slth}{\widehat{\goth{sl}}_2}

\newcommand{\gsl}{\goth{sl}}
\newcommand{\slnh}{\widehat{\goth{sl}}_N}

\newcommand{\g}{\goth{g}}

\newcommand{\gH}{\goth{H}}
\newcommand{\bgH}{\overline{\goth{H}}}

\newcommand{\slnhbig}{\widehat{\mbox{\fourteeneufm sl}}_N}  
\newcommand{\glnhbig}{\widehat{\mbox{\fourteeneufm gl}}_N}
\newcommand{\gl}{{\goth{gl}}}
\newcommand{\gln}{{\goth{gl}_N}}
\newcommand{\glnh}{\widehat{\goth{gl}}_N}
\newcommand{\glth}{\widehat{\goth{gl}}_2}
\newcommand{\glt}{{\goth{gl}}_2}

\newcommand{\h}{\goth{h}}
\newcommand{\gh}{\widehat{\goth{g}}}
\newcommand{\hh}{\goth{h}}

\newcommand{\gS}{\goth{S}}

\font\fourteeneufm=eufm10 scaled\magstep2    
   
\makeatletter
\@addtoreset{equation}{section}
\makeatother

\newtheorem{theorem}{Theorem}[section]
\newtheorem{lemma}[theorem]{Lemma}

\newtheorem{proposition}[theorem]{Proposition}

\newtheorem{definition}[theorem]{Definition}

\newtheorem{thm}{Theorem}[section]
\newtheorem{prop}[thm]{Proposition}
\newtheorem{lem}[thm]{Lemma}
\newtheorem{cor}[thm]{Corollary}

\newtheorem{df}[thm]{Definition}
\newtheorem{dfn}[thm]{Definition}

\numberwithin{equation}{section}

\numberwithin{equation}{section}

\begin{document}%
%
\title{Gelfand-Tsetlin Bases for Elliptic Quantum Groups}

\author{
Hitoshi Konno\thanks{E-mail: hkonno0@kaiyodai.ac.jp} \
and \
Kohei Motegi\thanks{E-mail: kmoteg0@kaiyodai.ac.jp}
\\\\
{\it 
Tokyo University of Marine Science and Technology,}\\
 {\it Etchujima 2-1-6, Koto-Ku, Tokyo, 135-8533, Japan} \\
\\\\
\\
}

\date{\today}

\maketitle

\begin{abstract}
We study the level-0 representations of  the elliptic quantum group $U_{q,p}(\widehat{\gl}_N)$. 
We give a classification theorem of the finite-dimensional irreducible representations of  
$U_{q,p}(\widehat{\gl}_N)$ in terms of the theta function analogue of the Drinfeld polynomial for the 
quantum affine algebra $U_q(\glnh)$.  
We also construct the Gelfand-Tsetlin bases for the level-0 $U_{q,p}(\widehat{\gl}_N)$-modules 
following the work by Nazarov-Tarasov for the Yangian $Y(\gl_N)$-modules.  
This is a construction in terms of the Drinfeld generators. 
For the case of tensor product of the vector representations, we give another construction of 
the Gelfand-Tsetlin bases in terms of the $L$-operators and make a connection between the two constructions. We also compare them with those obtained by the first author by using the $\gS_n$-action realized by the elliptic dynamical $R$-matrix on the standard bases. 
As a byproduct, we obtain an explicit formula for the partition functions of the corresponding 2-dimensional 
square lattice model in terms of the elliptic weight functions of type $A_{N-1}$.

\end{abstract}

\section{Introduction}
The classification theorem of the finite-dimensional irreducible representations of the Lie algebra $\gln=\gl(N,\C)$ in terms of the dominant integral weights and the construction of the Gelfand-Tsetlin bases are 
 the most important results in the classical representation theory. An extension of the former to the quantum groups, the Yangian $Y(\gln)$ and the quantum affine algebra $U_q(\glnh)$, was initiated by Tarasov\cite{Tarasov}, stated by Drinfeld \cite{Drinfeld} and established by Chari and Pressley \cite{CPYangian,CP,CP1994,CPBook}.  There the classification is given in terms of the polynomials 
called the Drinfeld polynomials, which can be regarded as a generating functions of a set of  complex numbers  
specifying the highest weight. To prove the theorem, for example for $U_q(\glnh)$, one needs  careful studies of the embedding structure $U_q(\gln)\ \hookrightarrow\ U_q(\glnh)$, the evaluation homomorphism $ev_a: U_q(\glnh)\ \mapsto\ U_q(\gln)$ as well as the isomorphism \cite{Drinfeld,Beck} between the two realizations of $U_q(\glnh)$, i.e. 
the Drinfeld-Jimbo realization\cite{Drinfeld86,Jimbo85} and the Drinfeld realization\cite{Drinfeld}. 
For elliptic quantum groups, only a partial result for $U_{q,p}(\slth)$ was obtained by using the evaluation homomorphism $U_{q,p}(\slth)\ \mapsto \ U_{q}(\slth)$\cite{KonnoJGP}. One of the purpose of this paper is to give a complete exposition of the necessary structures of the elliptic quantum groups $U_{q,p}(\glnh)$ and $U_{q,p}(\slnh)$, and establish an elliptic version of the classification theorem.  

To construct the Gelfand-Tsetlin bases for the elliptic quantum groups is the other purpose of this paper. 
The Gelfand-Tsetlin bases\cite{GT} of finite-dimensional irreducible representations of quantum groups were constructed for the Yangian $Y(\gl_N)$\cite{NT94,NT,Molev} and  the quantum group $U_q(\gln)$\cite{Jimbo85GT,UTS}. 
In particular, the  constructions in terms of  the quantum minor determinants of the $L$-operators, or equivalently the Drinfeld generators of $Y(\gl_N)$, developed by 
Nazarov-Tarasov \cite{NT} and Molev \cite{Molev} are important for us, because they can be extended to the elliptic case straightforwardly.  For the elliptic quantum groups $U_{q,p}(\glnh)$ and $E_{q,p}(\glnh)$,  the quantum minor determinants of the $L$-operators and their various properties were studied in \cite{KonnoASPM}. 

On the other hand, recently, the Gelfand-Tsetlin bases  of the tensor product of the vector representations
have been attracting a lot of attention in geometric representation theory of quantum groups. 
There the Gelfand-Tsetlin bases are identified with the fixed point bases (classes) of the equivariant cohomology, $\mathrm{K}$-theory and elliptic cohomology of quiver varieties corresponding to the quantum groups.  
See for example \cite{KonnoJinttwo}.  A key to this  is the identification of 
the weight functions in the representation theory of quantum groups \cite{GRTV,RTV,KonnoJinttwo,RTV19} with the stable envelopes in the geometry of quiver varieties\cite{MO,Okounkov,AO}.    
In fact, it has been shown  for the case of the elliptic quantum group $U_{q,p}(\widehat{\gl}_N)$ that the Gelfand-Tsetlin bases are obtained by transforming  the standard bases with the change of basis matrix 
given by the elliptic weight functions\cite{KonnoJinttwo}
(see also \cite{RTV} for the affine quantum group case). 
 
In this paper, we extend the construction of the Gelfand-Tsetlin bases to arbitrary level-0 $U_{q,p}(\widehat{\gl}_N)$-modules whose weights are labelled by the Gelfand-Tsetlin patterns. We give a general construction in terms of the Drinfeld generators following the results in  \cite{NT}. For the tensor product of the vector representations, we also give another construction in terms of the $L$-operators and make a connection to the general construction. We also make a connection to the one obtained in \cite{KonnoJinttwo}, where the Gelfand-Tsetlin bases were constructed by using the action of the permutation group $\gS_n$ realized by the elliptic dynamical $R$-matrix on the tensor product space. 
As a byproduct, we obtain an explicit evaluation of the partition functions of the 2-dimensional square lattice statistical model defined by the  $R$-matrix in terms of the elliptic weight functions. 

This paper is organized as follows. 
In Section 2, we expose some properties of the elliptic quantum groups $U_{q,p}(\glnh)$,  $U_{q,p}(\slnh)$ and $E_{q,p}(\glnh)$.  In Section 3, we give a proof of  the classification theorem of level-0 finite-dimensional irreducible representations of the elliptic quantum groups. In Section 4, we present an explicit and complete example of the classification of all finite-dimensional irreducible representations of $U_{q,p}(\glth)$ and their Gelfand-Tsetlin bases. In Section 4, we give a general construction of the Gelfand-Tsetlin bases for $U_{q,p}(\glnh)$-modules in terms of the Drinfeld generators.  In Section 5, we give the second construction of the Gelfand-Tsetlin bases for the case of the tensor product of the vector representations of $U_{q,p}(\glnh)$ in terms of the $L$-operators. We make a connection of them to the general construction as well as the one obtained in \cite{KonnoJinttwo}. By considering the change of basis matrices, we obtain an explicit evaluation of the partition function of the statistical model in terms of the elliptic weight functions.

\section{Elliptic Quantum Groups $U_{q,p}(\widehat{\gl}_N)$ and $E_{q,p}(\glnh)$}

In this section, we review the
elliptic quantum groups $U_{q,p}(\widehat{\gl}_N)$ and $E_{q,p}(\slnh)$ 
following \cite{KonnoASPM}. In particular we summarize embeddings of $U_q(\gl_N)$ to 
$U_q(\glnh)$ and further to $U_{q,p}(\widehat{\gl}_N)$ as well as evaluation homomorphisms from $U_{q,p}(\widehat{\gl}_N)$ to $U_q(\glnh)$ and further to $U_q(\gl_N)$.  
We also summarize an isomorphism between $U_{q,p}(\widehat{\gl}_N)$ and $E_{q,p}(\widehat{\gl}_N)$ and list some formulas on quantum minor determinants.

\subsection{Preliminaries}
Let $A=(a_{ij})$ $(0 \le i,j \le N-1)$ be the generalized Cartan matrix of the affine Lie algebra $\widehat{\gsl}_N=\widehat{\gsl}(N,\mathbb{C})$.
Let $\mathfrak{h}=\widetilde{\mathfrak{h}} \oplus \mathbb{C} d$,
$\widetilde{\mathfrak{h}} = \overline{\mathfrak{h}} \oplus \mathbb{C} c$,
$\overline{\mathfrak{h}}=\oplus_{i=1}^{N-1} \mathbb{C} h_i$ be the Cartan subalgebras of $\widehat{\gsl}_N$,
and $\mathfrak{h}^*=\widetilde{\mathfrak{h}}^* \oplus \mathbb{C} \delta$,
$\widetilde{\mathfrak{h}}^*=\overline{\mathfrak{h}}^* \oplus \mathbb{C} \Lambda_0$,
$\overline{\mathfrak{h}}^*=\oplus_{i=1}^{N-1} \mathbb{C} \overline{\Lambda}_i$ their duals.
Let
$\mathcal{Q}=\oplus_{i=1}^{N-1} \mathbb{Z} \alpha_i$ be the root lattice
and $\mathcal{P}=\oplus_{i=1}^{N-1} \mathbb{Z} \overline{\Lambda}_i$ the weight lattice.
The pairings between $\delta, \Lambda_0$, $\alpha_i, \overline{\Lambda}_i$ $(1 \le i \le N-1) \in \mathfrak{h}^*$ 
are given by
\begin{align}
\langle \alpha_i, h_j \rangle=a_{ji}, \ \langle \delta, d \rangle=\langle \Lambda_0, c \rangle=1, \ \langle \overline{\Lambda}_i, h_j \rangle=\delta_{i,j},
\end{align}
and the other pairings are 0. 
An element $\la$ in $\cP^+:=\oplus_{i=1}^{N-1} \mathbb{Z}_{\geq 0} \overline{\Lambda}_i$ is called a dominant integral weight. 
Let $\{ \epsilon_j \ (1 \le j \le N) \}$ be an orthonormal basis in $\mathbb{R}^N$
with the inner product $(\epsilon_j,\epsilon_k)=\delta_{j,k}$.
We set $\overline{\epsilon}_j=\epsilon_j-\sum_{k=1}^ N \epsilon_k/N$,
and realize the simple roots by $\alpha_j=\overline{\epsilon}_j-\overline{\epsilon}_{j+1}$ $(1 \le j \le N-1)$
and the fundamental weights by $\Lambda_j=\overline{\epsilon}_1+\dots+\overline{\epsilon}_j$ $(1 \le j \le N-1)$. In order to discuss $\glnh$, we consider $\overline{\gH}^*=\oplus_{l=1}^N\C\ep_l$ and  define $h_{\ep_l}$ $(1\leq l\leq N)$ by $h_i=h_{\ep_i}-h_{\ep_{i+1}}$. We set $\overline{\gH}=\oplus_{l=1}^N\C h_{\ep_l}$.  
We regard $\overline{\mathfrak{H}} \oplus \overline{\mathfrak{H}}^*$ as the Heisenberg algebra by
\begin{align}
[h_{\alpha},\beta]=(\alpha,\beta), \ \ \ [h_{\alpha},h_{\beta}]=[\alpha,\beta]=0, \ \ \ \alpha,\beta \in \overline{\mathfrak{H}}^*.
\end{align}
Let $\{P_\alpha, Q_\beta \}$ $(\alpha,\beta \in \overline{\mathfrak{H}}^*)$
be the Heisenberg algebra defined by
\begin{align}
[P_{\alpha},Q_{\beta}]=(\alpha,\beta), \ \ \ [P_{\alpha},P_{\beta}]=[Q_{\alpha},Q_{\beta}]=0.
\end{align}
We also set $\cR_Q=\oplus_{l=1}^N\C Q_{\ep_l}$ and denote by $\C[\cR_Q]$ the group ring of $\cR_Q$. 
Note that for $e^{Q_\al}, e^{Q_\beta}\in \C[\cR_Q]$ one has $e^{Q_\al}e^{Q_\beta}=e^{Q_\al+Q_\beta}\in \C[\cR_Q]$. 

We define the commutative algebra $H$ as $H=\widetilde{\h}\oplus\left(\oplus_{l=1}^N\C P_{\ep_l}\right)$, and its dual space $H^*=\widetilde{\mathfrak{h}}^* \oplus \Big( \oplus_{j=1}^N \mathbb{C} Q_{{\epsilon}_j} \Big)$. Then define $\mathbb{F}=\mathcal{M}_{H^*}$ to be the field of meromorphic functions on $H^*$.

For $x\in \C$, we set  ${[x]_q=\frac{q^x-q^{-x}}{q-q^{-1}}}$.  
Let $q_1,q_2,\cdots,q_k\in \C^\times$ satisfying $|q_1|,|q_2|,\cdots,|q_k|<1$. 
We introduce the $q$-infinite product
\bea
&&(x;q_1,q_2,\cdots,q_k)_\infty=\prod_{n_1,n_2,\cdots,n_k=0}^\infty(1-xq_1^{n_1}q_2^{n_2}\cdots q_k^{n_k}). 
\ena
 In particular, the $k=1$ case gives
\begin{align}
(x;q)_\infty=\prod_{n=0}^\infty (1-xq^n).
\end{align}
Let $p$ be a generic complex number satisfying $|p|<1$, 
and introduce the Jacobi's odd theta functions as
\begin{align}
 \Theta_p(z)=(z;p)_\infty (p/z;p)_\infty (p;p)_\infty.
\end{align}
It is also convenient to use
\bea
&&\theta(z)=-z^{-1/2}\Theta_p(z). 
\ena

 We also use the elliptic Gamma function defined by
 \bea
 &&\Gamma(z;p,q)=\frac{(pq/z;p,q)_\infty}{(z;p,q)_\infty}\qquad |p|, |q|<1.\lb{ellGamma}
 \ena
This satisfies
\bea
&&\Gamma(pz;p,q)=\frac{\Theta_q(z)}{(q;q)_\infty}\Gamma(z;p,q),\qquad 
\Gamma(qz;p,q)=\frac{\Theta_p(z)}{(p;p)_\infty}\Gamma(z;p,q).
\ena

\subsection{The elliptic dynamical $R$-matrix of the $\widehat{\gl}_N$ type}\lb{edR}

Let $\displaystyle \widehat{V}=\bigoplus_{j=1}^N {\mathbb F} v_j$
be the $N$-dimensional vector space over ${\mathbb F}$ and set 
$\widehat{V}_z:=\widehat{V}[z,z^{-1}]$. 
We assume $e^{Q_\al}\cdot v_j=v_j$. 
We consider the elliptic dynamical $R$-matrix
${R}^+(z,\Pi) \in \mathrm{End}(\widehat{V}_{z_1} \otimes
\widehat{V}_{z_2})$ given by
\begin{align}
\displaystyle {R}^+(z,\Pi)
=&\rho^+(z)\overline{R}(z,\Pi),\\
\displaystyle \overline{R}(z,\Pi)
=&\sum_{j=1}^N E_{jj} \otimes E_{jj}
+\sum_{1 \le j_1 < j_2 \le N}
\Biggl(
b(z,\Pi_{j_1,j_2}) E_{j_1,j_1} \otimes E_{j_2,j_2}
+\overline{b}(z) E_{j_2,j_2} \otimes E_{j_1,j_1}\nonumber \\
&+c(z,\Pi_{j_1,j_2}) E_{j_1,j_2} \otimes E_{j_2,j_1}
+\overline{c}(z,\Pi_{j_1,j_2}) E_{j_2,j_1} \otimes E_{j_1,j_2}
\Biggr), \label{ellipticrmatrix}
\end{align}
where 
$E_{i,j}v_\mu=\delta_{j,\mu}v_i$, and we set 
$\Pi_{j,k}=q^{2(P+h)_{j,k}}$,
$(P+h)_{j,k}:=(P+h)_{{\epsilon}_j}-(P+h)_{{\epsilon}_k}$,
and
\begin{align}
\displaystyle 
\rho^+(z)&=q^{-(N-1)/N}\frac{\Gamma(z;q^{2N},p)\Gamma(q^{2N}z;q^{2N},p)}{\Gamma(q^2z;q^{2N},p)\Gamma(q^{2N}q^{-2}z;q^{2N},p)},\\
b(z,\Pi)&=\frac{\theta(q^2\Pi)\theta(q^{-2}\Pi)\theta(z)}{\theta(\Pi)^2\theta(q^2z)},
\ \ \ \overline{b}(z)=\frac{\theta(z)}{\theta(q^2z)}, \\
c(z,\Pi)&=\frac{\theta(q^2)\theta(z\Pi)}{\theta(\Pi)\theta(q^2z)},
\ \ \ \overline{c}(z,\Pi)=\frac{\theta(q^2)\theta(z\Pi^{-1})}{\theta(\Pi^{-1})\theta(q^2z)}.
\end{align}
The $R$-matrix
$\overline{R}^+(z,q^{2s})$ satisfies the dynamical Yang-Baxter equation
\begin{align}
&\overline{R}^{+(12)}(z_1/z_2,q^{2(s+h^{(3)})})
\overline{R}^{+(13)}(z_1/z_3,q^{2s})
\overline{R}^{+(23)}(z_2/z_3,q^{2(s+h^{(1)})}) \nonumber \\
&=\overline{R}^{+(23)}(z_2/z_3,q^{2s})
\overline{R}^{+(13)}(z_1/z_3,q^{2(s+h^{(2)})})
\overline{R}^{+(12)}(z_1/z_2,q^{2s}),\lb{DYBE}
\end{align}
where $q^{2h_{j,k}^{(\ell)}}$ acts on the $\ell$-th tensor space $\widehat{V}_{z_l}$ by
$q^{2h_{j,k}^{(\ell)}} v_\mu=q^{2 \langle \overline{\epsilon}_\mu, h_{j,k} \rangle} v_\mu$.

\subsection{The elliptic algebra  $U_{q,p}(\widehat{\gl}_N)$
 }

\begin{dfn}\lb{defUqpgl}
The elliptic algebra $U_{q,p}(\glnh)$ is a topological algebra over $\FF[[p]]$ generated by 
$e_{j,m}, f_{j,m}, k_{l,m}$,  
$(1\leq j\leq N-1, 1\leq l\leq N, m\in \Z)$
and the central elements $q^{\pm c/2}$. 
The  defining relations are given in terms of the following generating functions called the elliptic currents. 
\bea
&&e_j(z)=\sum_{m\in \Z} e_{j,m}z^{-m}
,\quad f_j(z)=\sum_{m\in \Z} f_{j,m}z^{-m},\lb{defenfngl}\\
&&k_l^+(z)=\sum_{m\in \Z_{\geq 0}} k_{l,-m}z^{m}+\sum_{m\in \Z_{> 0}} k_{l,m}p^mz^{-m},\\
&& k^-_l(z)=q^{2h_{\ep_l}}k^+_l(zpq^{-c}).
\lb{defkngl}
\ena 
The relations are given in  Appendix \ref{defrelUqp}\footnote{Through this paper we do not consider the derivation operator ${\hd}$.}. 
\end{dfn}

In the following, we set $p^*=pq^{-2c}$. 
Let $\cE^{l}_m\ (1\leq l\leq N, m\in \Z_{\not=0})$ be operators satisfying\footnote{
Our $\cE^{l}_m$ is $(q^m-q^{-m})^2\cE^{+l}_m$ in \cite{FKO,KonnoASPM}. }
\bea
&&[\cE^{ l}_m,\cE^{ l}_n]
=\delta_{m+n,0}\frac{[cm]_q[(N-1) m]_q(q^m-q^{-m})^2 }{m[m]_q [N m]_q}\frac{1-p^m}{1-p^{*m}}q^{-cm},\lb{comcElcEl} \\
&&[\cE^{k}_m,\cE^{ l}_n]
=-\delta_{m+n,0}q^{-{\rm sgn}(k-l)Nm +(k-l)m}
\frac{[cm]_q(q^m-q^{-m})^2 }{m[Nm]_q}\frac{1-p^m}{1-p^{*m}}q^{-cm},\lb{comcEkcEl}\\
&&\sum_{l=1}^Nq^{(l-1)m}\cE^l_m=0.\lb{sumcE0}
\ena
Then one can realize $k^+_l(z)$ as
\bea
k^+_l(z)&=&K^+_{\ep_l} \exp\left\{-\sum_{m>0}\frac{\cE^{l}_{-m}}{1-p^m} (q^lz)^{m}\right\}
\exp\left\{\sum_{m>0}\frac{p^m\cE^{l}_m}{1-p^m} (q^lz)^{-m}\right\}.\lb{kpl}
\ena
 Here, expanding $k_{l,m}\in U_{q,p}(\glnh)$ in $p$  
\bea
&&k_{l,m}=\sum_{r\in \Z_{\geq 0}}k^{(r)}_{l,m}p^r,\qquad 1\leq l\leq N,\ m\in \Z,
\ena
we have
\bea
&&K^+_{\ep_l}=k^{(0)}_{l,0}.
\ena

Let us set
\be
&&K(z)=k^+_1(z)k^+_2(q^{-2}z)\cdots k^+_N(q^{-2(N-1)}z).
\en
Then $K(z)$ belongs to the center of $U_{q,p}(\glnh)$ (Proposition 3.2 in \cite{KonnoASPM}).  
This can also be verified  from \eqref{sumcE0}, \eqref{kpl} and the fact that $\prod_{l=1}^Nq^{-h_{\ep_l}}$ belongs to the center.  

Let $U^{D}_q(\glnh)$ be the quantum affine algebra over $\C$ 
generated by  the Drinfeld generators $X^\pm_{j,m}, K^\pm_{l,m}, q^{\pm c/2}$. The defining relations can be seen in Appendix A  in \cite{KonnoASPM}.  
 We set 
 \be
 &&X^\pm_j(z)=\sum_{m\in\Z}X^\pm_{j,m}z^{-m},\qquad 
 K^\pm_{l}(z)=\sum_{m\in\Z_{\geq 0}}K^\pm_{l,\mp m}z^{\pm m}.
 \en
 Note the relation $K^+_{l,0}K^-_{l,0}=1=K^-_{l,0}K^+_{l,0}$. 

Let us introduce the currents $u^+_{\ep_l}(z,p)\in U_{q,p}(\glnh)[[p]][[z]]$, $u^-_{\ep_l}(z,p)\in U_{q,p}(\glnh)[[p]][[z^{-1}]]$ $(1\leq l\leq N)$ by 
\bea
&&u^+_{\ep_l}(z,p)=\exp\left\{\sum_{m>0}\frac{(pq^{-c})^m\cE^l_{-m}}{1-p^m}\; z^m\right\},\\
&&u^-_{\ep_l}(z,p)=\exp\left\{-\sum_{m>0}\frac{p^m \cE^l_{m}}{1-p^m}\; z^{-m}\right\},
\ena
and set 
\bea
&&u_j^\pm(z,p)=u^\pm_{\ep_j}(z,p)u^\pm_{\ep_{j+1}}(qz,p)^{-1}\qquad (1\leq j\leq N-1).
\ena
These are well defined elements in $(U_{q,p}(\glnh)[[p]])[[z,z^{-1}]]$ in the $p$-adic 
topology.

\begin{thm}\lb{Uq2Uqp}
The map  $\pi_p:  (\FF\otimes_\C U^D_{q}(\glnh))\sharp\C[\cR_Q] \to U_{q,p}(\glnh)$ defined by 
\be
&& \pi_p(X^+_j(z) e^{-Q_{j}}) = u_j^+(z,p)e_j^+(z), \\
&& \pi_p(X^-_j(z)) =  f_j^-(z)u_j^-(z,p),\\
&&\pi_p(K^+_l(z)e^{-Q_{\ep_l}}) = u_{\ep_l}^+(q^{-c+l}z,p)k^+_{l}(z)u_{\ep_l}^-(q^{l}z,p),\\ 
&&\pi_p(K^-_l(z)e^{-Q_{\ep_l}})= u_{\ep_l}^+(q^{l} z,p)k^-_{l}(z)u_{\ep_l}^-(q^{-c+l}z,p).
\en
is an algebra homomorphism. 
 Here the smash product $\sharp$ is defined as follows.
\be
&&g(P,P+h)a\otimes e^{Q_\al} \cdot f(P,P+h)b\otimes e^{Q_\beta}\\
&&\quad= g(P,P+h)f(P-\bra Q_\al,P\ket  ,P+h-\bra Q_\al+\wt( a),P+h\ket  )ab\otimes e^{Q_\al+Q_\beta}
\en
where $\wt({a})\in \bar{\hh}^*$ s.t. $q^h a q^{-h}=q^{\bra \wt({a}),h\ket  }a$ for $a, b\in U_q(\gh), f(P), g(P)\in \FF, e^{Q_\al}, e^{Q_\beta}\in \C[\cR_Q]$. 
\end{thm}

In particular, one obtains the following identifications.
\begin{cor}\lb{UqbyUqp}
\bea
&&K^+_{l,0}e^{-Q_{\ep_l}}=q^{-2h_{\ep_l}}K^-_{l,0}e^{-Q_{\ep_l}}=K^+_{\ep_l}=q^{-h_{\ep_l}}e^{-Q_{\ep_l}},\\
&&\pi_p(K^+_l(z))=q^{-h_{\ep_l}}\exp\left\{-\sum_{m>0}\frac{1-p^{*m}}{1-p^m}\cE^l_{-m}(q^lz)^m\right\}\ \in U_{q,p}(\glnh)[[z]],\\
&&\pi_p(K^-_l(z))=q^{h_{\ep_l}}\exp\left\{\sum_{m>0}\cE^l_{m}(q^{l-c} z)^{-m}\right\}\ \in U_{q,p}(\glnh)[[z^{-1}]].
\ena
\end{cor}
One also has the inverse of $\pi_p$ ( Sec.5.2 in \cite{KonnoASPM}). Hence one obtains the isomorphism
\\ ${U_{q,p}(\glnh) \cong (\FF\otimes_{\C}U^D_q(\glnh))\sharp \C[\cR_Q]}$. 
 In the below we set $K_{\ep_l}=K^+_{l,0}=(K^{-}_{l,0})^{-1}=q^{-h_{\ep_l}}$.

The quantum affine algebra associated with $\glnh$ has the other realization $U^{DJ}(\glnh)$ 
called  the Drinfeld-Jimbo realization\cite{Drinfeld86,Jimbo85}.  
This is generated by the Chevalley type generators 
$\cX^\pm_r, \cK_{\ep_l}, q^{\pm c/2}\ (0\leq r\leq N-1,\ 1\leq l\leq N)$ subject to
\be
&&\cK_{\ep_l}\cK_{\ep_m}=\cK_{\ep_m}\cK_{\ep_l},\qquad \cK_{\ep_l}\cK_{\ep_l}^{-1}=\cK_{\ep_l}^{-1}\cK_{\ep_l}=1\qquad (1\leq l,m \leq N),\\ 
&&\cK_{\ep_l}\cX^\pm_r\cK_{\ep_l}^{-1}=q^{\mp\bra\al_r,h_{\ep_l}\ket}\cX^\pm_r,\\
&&[\cX^+_r,\cX^-_s]=\frac{\delta_{r,s}}{q-q^{-1}}(\cK_{\ep_r}^{-1}\cK_{\ep_{r+1}}-\cK_{\ep_r}\cK_{\ep_{r+1}}^{-1}),
\en
where we set $\al_0=\delta-\theta $ and $\cK_{\ep_0}\equiv q^c \cK_{\ep_N}$.

The isomorphism between  $U^D_q(\glnh)$ and $U^{DJ}_q(\glnh)$ is given as follows\cite{Drinfeld,
Koyama,
Beck}. 
Let us set in $U_q^D(\glnh)$
\bea
&&K_j:=K_{\ep_j}K_{\ep_{j+1}}^{-1},\qquad (1\leq j\leq N-1)\\
&&K_\theta:=K_1K_2\cdots K_{N-1}=K_{\ep_1}K_{\ep_{N}}^{-1}.
\ena
Then one has an isomorphism of $\C$-algebras $f: U^{DJ}_{q}(\glnh)\to U^D_q(\glnh)$  
given by
\bea
&&f(\cK_0)=q^cK_\theta,\qquad f(\cK_{\ep_l})=K_{\ep_l},\qquad f(\cX^\pm_j)=X^\pm_{j,0},\\
&&f(\cX^+_0)=[X^-_{N-1,0},[X^-_{N-2,0},\cdots[X^-_{2,0},X^-_{1,1}]_{q^{-1}}\cdots]_{q^{-1}}]_{q^{-1}}
K_\theta^{-1},\\
&&f(\cX^-_0)=K_\theta [[\cdots[X^+_{1,-1},X^+_{2,0}]_q\cdots,X^+_{N-2,0}]_q,X^+_{N-1,0}]_q,
\ena
where $[A,B]_{q^{\pm1}}=AB-q^{\pm1}BA$. One also has the opposite algebra homomorphism $F: U^{D}_{q}(\glnh)\to U^{DJ}_q(\glnh)$ in the same way as in \cite{Beck} for $U_q(\slnh)$.

It is also obvious that $U^{DJ}_q(\glnh)$ has a subalgebra $U_{q}(\gl_N)$ generated 
by $\cX^\pm_j, \cK_{\ep_l}\ (1\leq j\leq N-1,\ 1\leq l\leq N)$ 
 subject to
\be
&&\cK_{\ep_l}\cX^\pm_j\cK_{\ep_l}^{-1}=q^{\mp\bra\al_j,h_{\ep_l}\ket}\cX^\pm_j,\\
&&[\cX^+_i,\cX^-_j]=\frac{\delta_{i,j}}{q-q^{-1}}(\cK_{\ep_j}^{-1}\cK_{\ep_{j+1}}-\cK_{\ep_j}\cK_{\ep_{j+1}}^{-1}).
\en

The opposite algebra homomorphism $\ev_a: U^{DJ}_q(\glnh)\ 
\ {\longrightarrow}\ U_{q}(\gl_N)$ was obtained by Jimbo \cite{Jimbo86}. 
Let us define $\widehat{E}_{ij}\in U_q(\gln)\ (1\leq i\not=j \leq N)$ inductively by $\widehat{E}_{jj+1}=\cX^+_j$, $\widehat{E}_{j+1j}=\cX^-_j$ and
\be
&&\widehat{E}_{ij}=\widehat{E}_{ik}\widehat{E}_{kj}-q^{\pm 1} \widehat{E}_{kj}\widehat{E}_{ik}
 \qquad (i\gtrless k\gtrless j).
\en
Then  the following map gives an algebra homomorphism $\ev_a: U^{DJ}_q(\glnh)\ \ {\longrightarrow}\ U_{q}(\gl_N)[a,a^{-1}]$ with $a\in \C^\times$.
\bea
&&\ev_a(\cK_0)=\cK_{\ep_1}\cK_{\ep_N}^{-1},\qquad \\
&&\ev_a(\cX^+_0)=aq^{-1}\cK_{\ep_1}\cK_{\ep_N}\widehat{E}_{N1},\\
&&\ev_a(\cX^+_0)=a^{-1}q\cK_{\ep_1}^{-1}\cK_{\ep_N}^{-1}\widehat{E}_{1N},\\
&&\ev_a(\cK_{\ep_l})=\cK_{\ep_l},\qquad \ev_a(\cX^\pm_j)=\cX^\pm_j\qquad (1\leq l\leq N, 1\leq j\leq N-1). 
\ena

In summary we have the following embeddings
\bea
&&U_q(\gl_N) \ \hookrightarrow \ U^{DJ}_q(\glnh)\ \hookrightarrow \ 
(\FF\otimes_\C U^{D}_{q}(\glnh))\sharp\C[\cR_Q]
\ \stackrel{\pi_p}{\hookrightarrow}\ U_{q,p}(\glnh),\lb{embeddings}
\ena
where in the middle we used $U^{DJ}_q(\glnh)\stackrel{\cong}{\longrightarrow} U^{D}_q(\glnh)$ by Beck \cite{Beck}. In particular, we have the identification 
\bea
&&\cK_{\ep_l}\ \mapsto \ K_{\ep_l},\quad  K_{\ep_l}e^{-Q_{\ep_l}} \mapsto\ K^{+}_{\ep_l}.\lb{cKlkl0}
\ena
On the other hand, we have the algebra homomorphisms of the opposite direction.  
\bea
&&\hspace{-1cm}U_{q,p}(\glnh) \ \stackrel{\pi_p^{-1}}{\longrightarrow} \ (\FF\otimes_\C U^D_{q}(\glnh))\sharp\C[\cR_Q] 
\ \stackrel{\ev_a}{\longrightarrow}\ (\FF\otimes_\C U_{q}(\gl_N)[a,a^{-1}])\sharp\C[\cR_Q].\lb{evaluationhom}
\ena

\subsection{The elliptic algebra $U_{q,p}(\slnh)$}
The elliptic algebra $U_{q,p}(\slnh)$ is defined as the  quotient algebra 
\bea
&&U_{q,p}(\glnh)/<K(z)-1>.
\ena
Set
\bea
&&K^+_j:=K^+_{\ep_j}K^{+-1}_{\ep_{j+1}},\\
&&\al_{j,m}:=\frac{1}{q-q^{-1}}\left(\cE^j_m-q^{-m}\cE^{j+1}_m\right)\qquad 1\leq j\leq N-1,\ m\in \Z.\lb{alcEmcE}
\ena
One has
\bea
&&[\al_{i,m},\al_{j,n}]=\delta_{m+n,0}\frac{[a_{ij}m]_q
[cm]_q
}{m}
\frac{1-p^m}{1-p^{*m}}
q^{-cm}.\lb{ellboson}
\ena
Solving \eqref{alcEmcE} and \eqref{sumcE0}, one  obtains
\bea
&&\cE^{j}_m=(q-q^{-1})\frac{q^{jm}}{[Nm]_q}\left(-q^{-Nm}\sum_{k=1}^{j-1}[km]_q\al_{k,m}+\sum_{k=j}^{N-1}[(N-k)m]_q\al_{k,m}\right)\quad (1\leq j\leq N-1),\nn\\
&&\cE^{N}_m=-\frac{(q-q^{-1})}{[Nm]_q}\sum_{k=1}^{N-1}[km]_q\al_{k,m}.\lb{cENal}
\ena
Let us define $\psi^\pm_j(z)\ (1\leq j\leq N-1)$ by\footnote{Our $\psi^\pm_j(z)$ are $\psi^\mp_j(z)$ in \cite{JKOS,KojimaKonno}.}
\bea
&&\psi^+_j(q^{-c/2}q^jz):=\kappa k^+_j(z)k^+_{j+1}(z)^{-1},\\
&&\psi^-_j(q^{-c/2}q^jz):=\kappa k^-_j(z)k^-_{j+1}(z)^{-1}\lb{psijkjkjp1}
\ena
with $\kappa$ given  in \eqref{def:kappa}. 
Then one has
\bea
&&{\psi}_j^+(
q^{-\frac{c}{2}}
z)=K^+_{j}\exp\left(-(q-q^{-1})\sum_{n>0}\frac{\al_{j,-n}}{1-p^n}z^n\right)
\exp\left((q-q^{-1})\sum_{n>0}\frac{p^n\al_{j,n}}{1-p^n}z^{-n}\right),\nn\\
&&\psi^-_j(z)=q^{2h_j}\psi^+_j(zpq^{-c}).\lb{def:psijp}
\ena

\begin{dfn}\lb{defUqp}
The elliptic algebra $U_{q,p}(\slnh)$ is a topological algebra over $\FF[[p]]$ generated by $\al_{j,n}, e_{j,m}, f_{j,m}\ (1\leq j\leq N-1, n\in \Z\backslash\{0\}, m\in \Z)$ and the central elements $q^{\pm c/2}$ satisfying the relations in Appendix \ref{defrelUqpslnh}. 
\end{dfn}

The quantum affine algebra $U_q(\slnh)$ is an associative algebra generated by 
$X^\pm_{j,m}\ (1\leq j\leq N-1, m\in \Z)$ and 
$a_{j,n}$ $(n\in \Z\backslash\{0\})$ defined by
\bea
&&a_{j,n}=\left\{\mmatrix{\al_{j,n}&n>0\cr
\frac{1-p^{*n}}{1-p^{n}}q^{2cn}\al_{j,n}&n<0\cr}\right..
\ena
One finds
\be
&&[a_{i,m},a_{j,n}]=\frac{[a_{ij}m]_q[cm]_q}{m}q^{-c|m|}\delta_{m+n,0}.
\en
Furthermore define 
\bea
&&\Phi^\pm_j(q^jz):=K^\pm_j(z)K^{\pm}_{j+1}(z)^{-1}.\lb{PhijKjKjp1}
\ena
Then one gets the generating functions of $a_{j,m}$ as 
\bea
&&\Phi^\pm_j(z)=q^{\mp h_j}\exp\left\{\mp(q-q^{-1})\sum_{m>0}a_{j,\mp m}z^{\pm m}\right\}.\lb{def:Phipm}
\ena

\subsection{The elliptic algebra $E_{q,p}(\widehat{\gl}_N)$}
The elliptic algebra $E_{q,p}(\widehat{\gl}_N)$ is generated by the $L$-operators. 
Let $\bar{L}_{ij, n}\ (n\in \Z, 1\leq i,j\leq N) $ be abstract symbols.   
We define the $L$-operator $L^+(z)=\sum_{1\leq i,j\leq N}E_{ij}\otimes L^+_{ij}(z) $ by  
\bea
&&L^+_{ij}(z)= \sum_{n\in \Z} L_{ij, n} z^{-n},\qquad L_{ij, n}=p^{{\rm max}( n,0)} 
\bar{L}_{ij, n}.\lb{genL}
\ena
\begin{dfn}\lb{defEqp}
Let $R^{+}(z,\Pi)$ be the same $R$ matrix as in Sec.\ref{edR}. 
The elliptic algebra  $\cE=E_{q,p}(\glnh)$  is a topological algebra over $\FF[[p]]$ generated by $\bar{L}_{ij, n}$ 
 and the central element $q^{\pm c/2}$ satisfying the following relations. 
\bea
&&R^{+(12)}(z_1/z_2,\Pi)L^{+(1)}(z_1)L^{+(2)}(z_2)=L^{+(2)}(z_2)L^{+(1)}(z_1)R^{+*(12)}(z_1/z_2,\Pi^*),\lb{RLL}\\
&&g({P+h})\bar{L}_{ij,n}=\bar{L}_{ij,n}\; g(P+h-\bra Q_{\bep_i},P+h\ket  ),\lb{lgr}\\
&&g({P})\bar{L}_{ij,n}=\bar{L}_{ij,n}\; g(P-\bra Q_{\bep_j},P\ket  ),\lb{rgr}
\ena
where $g({P+h}), g(P)\in \FF$ and 
\be
L^{+(1)}(z)=L^{+}(z)\otimes \id, \qquad L^{+(2)}(z)=\id \otimes L^{+}(z). 
\en
The $R$-matrix $R^{+*}(z,\Pi^*)$ is the same as $R^{+}(z,\Pi)$ except for the replacements : theta functions 
$\theta(z)$ by $\theta^*(z):=-z^{-1/2}(z;p^*)_\infty(p^*/z;p^*)_\infty(p^*;p^*)_\infty$ and the dynamical 
parameters $\Pi_{k,l}$ by $\Pi^*_{k,l}=q^{2(P_{\ep_k}-P_{\ep_l})}$. 
\end{dfn}
We regard $L^+(z) \in \End_\FF \widehat{V}_z \otimes \cE$.
We treat \eqref{RLL} as a formal Laurent series in $z_1$ and $z_2$. The coefficients of $z_1, z_2$ are well defined in the $p$-adic topology.

It is also convenient to introduce the dynamical $L$-operator $L^+(z,\Pi^*)$ 
defined by\cite{JKOS,KojimaKonno}
\bea
&&L^+(z,\Pi^*)=L^+(z)e^{\sum_{i=1}^N\pi_V(h_{\vep_i})\otimes Q_{\ep_i}},\lb{DL}
\ena
where $\pi_V(h_{\epsilon_i})=E_{i,i}$. We then have the fully dynamical $RLL$-relation\cite{Felder}
\bea
&&R^{+(12)}(z_1/z_2,{\Pi}^*q^{2h^{(3)}})L^{+(1)}(z_1,\Pi^*)L^{+(2)}(z_2,{\Pi^* q^{2\pi_V(h)^{(1)}}})\nn\\
&&\qquad\qquad=L^{+(2)}(z_2,\Pi^*)L^{+(1)}(z_1,{\Pi^*q^{2\pi_V(h)^{(2)}}})R^{+*(12)}(z_1/z_2,\Pi^*). 
\lb{DRll}
\ena

\subsection{Isomorphism between $U_{q,p}(\glnh)$ and $E_{q,p}(\glnh)$}
The two elliptic algebras $U_{q,p}(\glnh)$ and $E_{q,p}(\glnh)$ are isomorphic\cite{KonnoASPM}.
We briefly summarize their isomorphism. 

Let us introduce  $L^-(z)=\sum_{1\leq i,j\leq N}E_{ij}\otimes L^-_{ij}(z)$ by\cite{JKOStg,KonnoASPM}   
\bea
&&L^-(z)=\left(\Ad(q^{-2{\theta}_V(P)})\otimes \id\right)\left(q^{2{T}_{V}}L^+(z p^* q^{c})\right)
,\lb{LmfromLp}\\
&&{\theta}_V(P)=-\sum_{j=1}^{N-1}\left(\frac{1}{2}\pi_V(h_j)\pi_V(h^j)+P_{j}\pi_V(h^j)\right),\lb{thetaV}\\
&&{T}_{V}
=\sum_{j=1}^{N-1}\pi_V(h_j)\otimes h^j.\lb{TV}
\ena
Here $(\Ad X)Y=XYX^{-1}$, $h^j=h_{\bar{\Lambda}_j}\ (j\in I)$
, $\pi_V(h_j)=E_{jj}-E_{j+1j+1}$ and 
$\pi_V(h^j)=\sum_{i=1}^j \pi_V(h_{\bep_i})\ (j\in I)$. 
Hence $L^+$ and $L^-$ are not independent operators in the elliptic algebra in contrast to the quantum affine algebra $U_{q}(\glnh)$. 

One can verify the following.  
\begin{prop}\lb{derRLL}
The $L$ operators $L^+(z)$ and $L^-(z)$ satisfy the following relations.  
\be
&&\hspace{-1cm}R^{-(12)}(z_1/z_2,\Pi)L^{-(1)}(z_1)L^{-(2)}(z_2)=L^{-(2)}(z_2)L^{-(1)}(z_1)R^{-*(12)}(z_1/z_2,\Pi^*),\lb{mRLL}\\
&&\hspace{-1cm}R^{\pm(12)}(q^{\pm{c}}z_1/z_2,\Pi)L^{\pm(1)}(z_1)L^{\mp(2)}(z_2)=L^{\mp(2)}(z_2)L^{\pm(1)}(z_1)R^{\pm*(12)}(q^{\mp{c}}z_1/z_2,\Pi^*).
\lb{RLLpm}
\en
\end{prop}

\begin{df}~~We define the Gauss components $E^\pm_{l,j}(z), F^\pm_{j,l}(z), K^\pm_m(z) \ (1\leq j<l\leq N, 1\leq m\leq N)$ of the $L$-operator
${L}^\pm(z)$ of  $\cE$ as follows.
\begin{eqnarray}
&&{L}^\pm(z)=
\left(\begin{array}{ccccc}
1&F_{1,2}^\pm(z)&F_{1,3}^\pm(z)&\cdots&F_{1,N}^\pm(z)\\
0&1&F_{2,3}^\pm(z)&\cdots&F_{2,N}^\pm(z)\\
\vdots&\ddots&\ddots&\ddots&\vdots\\
\vdots&&\ddots&1&F_{N-1,N}^\pm(z)\\
0&\cdots&\cdots&0&1
\end{array}\right)\left(
\begin{array}{cccc}
K^\pm_1(z)&0&\cdots&0\\
0&K^\pm_2(z)&&\vdots\\
\vdots&&\ddots&0\\
0&\cdots&0&K^\pm_{N}(z)
\end{array}
\right)\nn\\
&&\qquad\qquad\qquad\qquad\qquad\qquad\qquad\times
\left(
\begin{array}{ccccc}
1&0&\cdots&\cdots&0\\
E^\pm_{2,1}(z)&1&\ddots&&\vdots\\
E^\pm_{3,1}(z)&
E^\pm_{3,2}(z)&\ddots&\ddots&\vdots\\
\vdots&\vdots&\ddots&1&0\\
E^\pm_{N,1}(z)&E^\pm_{N,2}(z)
&\cdots&E^\pm_{N,N-1}(z)&1
\end{array}
\right).\lb{def:lhat}
\ena
In particular we call $E^\pm_{j+1,j}(z), F^\pm_{j,j+1}(z), K^\pm_m(z)$ the basic half currents.
\end{df}
In \cite{KonnoASPM}, it was shown that there are a sequence of subalgebras
\bea
&&E_{q,p}(\widehat{\gl}_1)\subset E_{q,p}(\widehat{\gl}_2)\subset \cdots \subset E_{q,p}(\widehat{\gl}_N),\lb{seqsubalg}
\ena
where $E_{q,p}(\widehat{\gl}_m)$ is generated by the coefficients of 
the basic half currents $E^+_{a+1,a}(u), F^+_{a,a+1}(u), K_b^+(u)$  $(N-m+1\leq a\leq N-1,\ N-m+1\leq b\leq N)$.  

In addition, the following lemma indicates that the whole Gauss components of $L^\pm(z)$ can be determined recursively by the basic half currents. 
\begin{lem}[\cite{KonnoASPM}, Lemma 6.10]\lb{recursion}
Let $I_{a,b}=\{\ (j,k)\ |\  a\leq j\leq b-1,\  j+1\leq k\leq b \}\setminus \{(a,b)\}$. 
For $2\leq  l+1<m\leq N$,  $E^+_{m,l}(z)$ (resp. $F^+_{l,m}(z)$) is determined by 
$\{E^+_{k,j
}(z) \ (j,k
)\in I_{l,m}, 
K^+_j(z)\ l\leq j\leq m \}$ (resp. 
$\{F^+_{j,k
}(z) \ (j,k
)\in I_{l,m}, 
K^+_j(z)\ l\leq j\leq m \}$). 
\end{lem}

Now let us define the basic half currents $e_{j+1,j}^+(z), f_{j,j+1}^+(z)$ $(1\leq j\leq N-1)$ of $\cU$ as follows. 
\begin{df}\lb{inthc}
\bea
&&e_{j+1,j}^+(z)=
a^*\oint_{C^*} \frac{dz'}{2\pi i z'}
e_{j}(z')\frac{\tes(z q^{j-c}q^{2(1-
P_{\al_j})}/z')\tes(q^2)}
{\tes(zq^{j-c}/z')
\tes(q^{2(P_{\al_j}-1)})},
\\
&&f_{j,j+1}^+(z)
=a\oint_C \frac{dz'}{2\pi i z'}
f_{j}(z')\frac{\te(zq^{j}q^{2( (P+h)_{\al_j}-1)}/z')\te(q^2)}{\te(zq^{j}/z')\te(q^{2((P+h)_{\al_j}-1)})},
\end{eqnarray}
where $C^*: |q^{j-c}z|<|z'|<|p^{*-1}q^{j-c}z|$, $C: |q^{j}z|<|z'|<|p^{-1}q^{j}z|$, and 
\bea
&&a^*=q^{-1}\frac{(p^*;p^*)_\infty}{(p^*q^2;p^*)_\infty} ,\qquad a=q^{-1}\frac{(p;p)_\infty}{(pq^{-2};p)_\infty}.
\ena
\end{df}

\begin{thm}\lb{isoUqpEqp}
The following  map gives an isomorphism $ \cU\cong \cE$  
as a topological $H$-algebra over $\FF[[p]]$.   
\be
&&e^+_{j+1,j}(z)\mapsto E^+_{j+1,j}(z),\quad f^+_{j,j+1}(z)\mapsto F^+_{j,j+1}(z),\quad k^+_{l}(z)\mapsto K^+_l(z)\quad (1\leq j\leq N-1,\ 1\leq l\leq N).
\en 

\end{thm}

\subsection{Quantum determinant}\lb{sec:qdet}

We introduce the  quantum minor determinant of $L^+(z)$.

\begin{definition} (\cite{KonnoASPM}, Proposition E.12)
For $I=\{ i_1 < i_2 < \cdots < i_k \}$, $J=\{
j_1 < j_2 < \cdots < j_k \}$, we define ${\ell}^+(z)_I^J$ as
\begin{align}
{\ell}^+(z)_I^J 
:=&\mathcal{N}_k
\sum_{\sigma \in \mathfrak{S}_{k}}
{\mathrm{sgn}}_{J}^* (\sigma,\Pi^*)
L^+_{i_1 j_{\sigma(i)}}(z)
L^+_{i_2 j_{\sigma(2)}}(zq^{-2})
\cdots
L^+_{i_k j_{\sigma(k)}}(zq^{-2(k-1)}),
\label{expansion}
\\
=&\mathcal{N}_k^{-1}
\sum_{\sigma \in \mathfrak{S}_{k}}
{\mathrm{sgn}}_{I} (\sigma,\Pi)
L^+_{i_{\sigma(k)} j_k}(zq^{-2(k-1)})
L^+_{i_{\sigma(k-1)} j_{k-1}}(zq^{-2(k-2)})
\cdots
L^+_{i_{\sigma(1)} j_1}(z),
\label{expansiontwo}
\end{align}
where $\mathfrak{S}_{k}$ is the permutation group of degree $k$, and we set 
\begin{align}
{\mathrm{sgn}}_{I} (\sigma,\Pi)
:=&\prod_{
\substack{
1 \le a < b \le k
\\
\sigma(a) > \sigma(b)
}
}
\frac{\theta(q^2\Pi_{i_{\sigma(a)},i_{\sigma(b)}})}{\theta(\Pi_{i_{\sigma(b)},i_{\sigma(a)}})}, \\
{\mathrm{sgn}}_{I}^* (\sigma,\Pi^*)
:=&\prod_{
\substack{
1 \le a < b \le k
\\
\sigma(a) > \sigma(b)
}
}
\frac{\theta^*(q^2\Pi^*_{i_{\sigma(a)},i_{\sigma(b)}})}{\theta^*(\Pi^*_{i_{\sigma(b)},i_{\sigma(a)}})},
\end{align}
and
\begin{align}
\mathcal{N}_k:=&\prod_{1 \le a < b \le k}
\sqrt{
\frac{
\rho_0^* \theta^*(q^{2a}) \theta(q^2)}{\rho_0 \theta(q^{2a}) \theta^*(q^2)
}
}.
\end{align}
\end{definition}
We have the following are fundamental properties.
\begin{proposition} (\cite{KonnoASPM}, Proposition E.14)
\begin{align}
{\ell}^+(z)_I^{\tau(J)}={\mathrm{sgn}}_J^*(\tau,P)
{\ell}^+(z)_I^{J}.
\label{exchangeJ}
\end{align}
\end{proposition}

\begin{proposition} (\cite{KonnoASPM}, Proposition E.16)
\begin{align}
\displaystyle {\ell}^+(z)_I^{J}
=\sum_{\ell=1}^k (\mathcal{N}_k^\prime)^{-1}
\prod_{1 \le a < \ell} \frac{\theta(q^2\Pi_{i_\ell,i_a})}{\theta(\Pi_{i_a,i_{\ell}})}
{\ell}^+(zq^{-2})_{\widehat{i_\ell}}^{\widehat{j_1}}
L_{i_\ell j_1}^+(z),
\label{expansionsmallerqminors}
\end{align}
where $\widehat{i_\ell}=I\backslash\{i_\ell\}$ and 
\begin{align}
\mathcal{N}_k^\prime
=\frac{\mathcal{N}_k}{\mathcal{N}_{k-1}}.
\end{align}
\end{proposition}

The quantum determinant of the $L$-operator
is given by
\begin{align}
\qdet
L^+(z)
={\ell}^+(z)_{[1,N]}^{[1,N]},
\end{align}
where $[1,N]=\{ 1,2,\dots,N \}$.

\begin{dfn}\lb{def:ABC}
Let us define
\bea
&&A_\ell(z):= \hell^+(z)^{[\ell,N]}_{[\ell,N]},\qquad \lb{def:Al}\\
&&B_m(z):=\hell^+(z)^{mm+1m+2\cdots N}_{m-1m+1m+2\cdots N},\\
&&C_m(z):=\hell^+(z)_{mm+1m+2\cdots N}^{m-1m+1m+2\cdots N},
\ena
for $1\leq \ell\leq N$ and $2\leq m\leq N$. In particular, we set
\be
&&A_N(z)=\ell^+(z)^{\{N\}}_{\{N\}}=L^+_{NN}(z)=K^+_N(z),\\
&&B_N(z)=\ell^+(z)^{\{N\}}_{\{N-1\}}=L^+_{N-1N}(z),\\
&&C_N(z)=\ell^+(z)_{\{N\}}^{\{N-1\}}=L^+_{NN-1}(z).
\en
\end{dfn}
One then finds the following relations to  the basic half currents.
\begin{prop}\lb{HCABC}
\bea
&&K^+_\ell(z)=(\cN'_{N-\ell+1})^{-1}A_\ell(z)A_{\ell+1}(zq^{-2})^{-1},\\
&&E^+_{m,m-1}(z)=A_m(z)^{-1}C_m(z),\\
&&F^+_{m-1,m}(z)=B_m(z)A_m(z)^{-1}.
\ena
Here we set $A_{N+1}(z)\equiv 1$. Hence through the isomorphism Theorem \ref{isoUqpEqp}, $A_m(z), B_m(z), C_m(z) $ are related to the Drinfeld generators of $U_{q,p}(\glnh)$.  

\end{prop}

\noindent
{\it Proof.}\ 
In \cite{KonnoASPM}, we introduced 
$L^+(z)_{a,a}=(L_{i,j}^+(z))_{a \le i,j \le N}$
and
\begin{align}
L^+(z)_{a,b}
=&\begin{pmatrix}
  L^+_{a,b}(z) & L^+_{a,a+1}(z) & \cdots & L^+_{a,N}(z) \\
 L^+_{a+1,b}(z) & L^+_{a+1,a+1}(z) & \cdots & L^+_{a+1,N}(z) \\
\vdots & \vdots &  & \vdots \\
L^+_{N,b}(z) & L^+_{N,a+1}(z) & \cdots & L^+_{N,N}(z)
\end{pmatrix}, \ \ \ a>b,
\end{align}
\begin{align}
L^+(z)_{a,b}
=&\begin{pmatrix}
  L^+_{a,b}(z) & L^+_{a,b+1}(z) & \cdots & L^+_{a,N}(z) \\
 L^+_{b+1,b}(z) & L^+_{b+1,b+1}(z) & \cdots & L^+_{b+1,N}(z) \\
\vdots & \vdots &  & \vdots \\
L^+_{N,b}(z) & L^+_{N,b+1}(z) & \cdots & L^+_{N,N}(z)
\end{pmatrix}, \ \ \ a<b.
\end{align}
The statements follows from Theorem E.23 in \cite{KonnoASPM} with the identification 
\be
&&\qdet L^+(z)_{\ell,\ell}=\hell^+(z)^{[\ell,N]}_{[\ell,N]},\\
&&\qdet L^+(z)_{m,m-1}=\hell^+(z)^{mm+1m+2\cdots N}_{m-1m+1m+2\cdots N},\\
&&\qdet L^+(z)_{m-1,m}=\hell^+(z)_{mm+1m+2\cdots N}^{m-1m+1m+2\cdots N}. 
\en
\qed
\begin{cor}
\bea
&(1)& A_m(z)=\cN_{N-m+1}K^+_m(z)K^+_{m+1}(zq^{-2})\cdots K^+_N(zq^{-2(N-m)}),\lb{AmK}\\
&(2)&\qdet \hL^+(z)=\cN_NK(z).\lb{qdetLp}
\ena
\end{cor}

\begin{prop}
The coefficients of $A_{m}(z)$ belong to the center of the subalgebra $E_{q,p}(\widehat{\gl}_{N-m+1})$ for 
$m=1,2,\cdots,N$. 
\end{prop}
\begin{proof}
The statement follows from Theorem 6.7 in  \cite{KonnoASPM} (see below \eqref{seqsubalg}) and Corollary E.24 in  \cite{KonnoASPM}.  
\end{proof}

From \eqref{def:lhat}, \eqref{RLL} and Proposition \ref{HCABC}, one obtains the following commutation relations.  They play a key role in the construction of the Gelfand-Tsetlin bases  in Sec.\ref{sec:GTB}.
\begin{prop}\lb{comABC}
\bea
&&[A_m(z),A_n(w)]=0,\lb{comAmAn}\\
&&[A_m(z),B_n(w)]=0 \qquad m\not=n,\lb{comAmBn}\\
&&[A_m(z),C_n(w)]=0 \qquad m\not=n,\lb{comAmCn}\\
&&[B_m(z),C_n(w)]=0 \qquad m\not=n,\lb{comBmCn}\\
&&[B_m(z),B_n(w)]=0 \qquad m\not=n\pm 1,\lb{comBmBn}\\
&&[C_m(z),C_n(w)]=0 \qquad m\not=n\pm 1,\lb{comCmCn}
\\
&&A_m(z)B_m(w)=\frac{\theta(q^2z/w)}{\theta(z/w)}B_m(w)A_m(z)
-\frac{\theta(\Pi^{-1}_{m-1,m}z/w)\theta(q^2)}{\theta(\Pi_{m-1,m}^{-1})\theta(z/w)}B_m(z)A_m(w) \quad z\not=w,\nn\\
&&\lb{comAmBm}\\
&&C_m(w)A_m(z)=\frac{\theta^*(q^2z/w)}{\theta^*(z/w)}A_m(z)C_m(w)
-A_m(w)C_m(z)\frac{\theta^*(\Pi^{*-1}_{m-1,m}z/w)\theta^*(q^2)}{\theta^*(\Pi_{m-1,m}^{*-1})\theta^*(z/w)}
 \quad z\not=w.\nn\\
 &&\lb{comCmAm}
\ena
\bea
&&{C}_m(z){B}_m(w) \nn \\
&&={A}_m(z){B}_m(w){A}_m(w)^{-1}{A}_m(z)^{-1}{C}_m(z){A}_m(w)
\nn \\
&&+\cN'_{N-m+1}{A}_m(z)K^+_{m-1}(w){A}_{m+1}(v-1)
\frac{\theta^*(\Pi^{*-1}_{N-m,N-m+1}z/w)\theta^{*}(q^2)}{\theta^*(\Pi^{*-1}_{N-m,N-m+1})\theta^*(z/w)}
 \nn \\
&&-\cN'_{N-m+1}A_m(z){A}_{m+1}(u-1)A_m(z)^{-1}K^+_{m-1}(z){A}_m(w)
\frac{\theta(\Pi^{-1}_{N-m,N-m+1}z/w)\theta(q^2)}{\theta(\Pi^{-1}_{N-m,N-m+1})\theta(z/w)}.
\label{comCmBm}
\ena
\end{prop}

\noindent
{\it Proof.}\ From the $(m,l-1) (m, l)$ and $(m-1,l-1) (m, l-1)$ components of (6.25) in \cite{KonnoASPM}, one has for  $l<m$
\be
&& [F^+_{l-1 l}(z_2),K^+_m(z_1)]=0,\qquad 
[F^+_{m-1 m}(z_1),K^+_{l-1}(z_2)]=0, 
\en
respectively.
Hence one obtains 
\be
&&[B_{a}(z_1),K^+_l(z_2)]=0\qquad (l\not= a-1, a).
\en
Combining this with $[B_a(z_1),\qdet L^+(z)]=0$, one obtains
\be
&&[B_a(z_1),A_l(z_1)]=0\qquad (l\not =a).
\en

Similarly \eqref{comAmCn} follows from the $ (m, l) (m,l-1)$ and $ (m, l-1) (m-1,l-1)$ components for $l<m$ of (6.25) in \cite{KonnoASPM}. 

\eqref{comBmCn} follows from the  component $(m-1, l) (m, l-1)$  $( l<m)$ of (6.25)   in \cite{KonnoASPM},  \eqref{comAmBn} and \eqref{comAmCn}.  

\eqref{comBmBn} follows from the  component $(m-1, l) (m, l+1)$  $(1\leq l<m-2\leq N-2)$ of (6.25), the  $(j, j) (j+1, j+1)$ component of (6.23)  in \cite{KonnoASPM} and \eqref{comAmBn}. 
\eqref{comCmCn} can be proved similarly.  

 \eqref{comAmBm} and \eqref{comCmAm} are obtained from (C.8) and (C.6) in \cite{KonnoASPM}, respectively.

 \eqref{comCmBm} follows from the relation 
 ( \cite{KojimaKonno} (C.6), 
\cite{KonnoASPM}, (C.19))
\begin{align}
[E_{m,m-1}^+(z),F_{m-1,m}^+(w)]
=&K_{m-1}^+(w)K_m^{+}(w)^{-1} \frac{\bar{c}(z/w,q^{-2}\Pi^*_{m-1,m})}{\overline{b}(z/w)} \nn \\
&-K_m^+(z)^{-1} K_{m-1}^{+}(z) \frac{\bar{c}(z/w,q^{-2}\Pi_{m-1,m})}{\overline{b}(z/w)}.\nn
\end{align}

\qed

\medskip
\noindent
{\it Remark.}\ This proposition can also be  obtained by the realization of the basic half currents 
(Definition \ref{inthc}) and the defining relations of $U_{q,p}(\glnh)$\cite{JKOS,KojimaKonno,KonnoASPM}. 

\subsection{Coalgebra structure}\lb{sec:Hopfalgebroid}

We  present a $H$-Hopf algebroid structure\cite{EV,KR,KonnoASPM} 
as a common coalgebra structure of $U_{q,p}(\glnh)$ and $E_{q,p}(\glnh)$.

Let $\cA$ denote $U_{q,p}(\glnh)$ or $E_{q,p}(\glnh)$. 
The elliptic algebra $\cA$  is a $H$-algebra by
\bea
&&\cA=\bigoplus_{\al,\beta\in H^*} \cA_{\al,\beta}\lb{bigradingUqp}\\
&&\cA_{\al\beta}=\left\{a\in \cU \left|\ q^{P+h}a q^{-(P+h)}=q^{\bra\al,P+h\ket}a,\quad q^{P}a q^{-P}=q^{\bra\beta,P\ket}a,\quad \forall P+h, P\in H\right.\right\}\nn
\ena
and two moment maps $\mu_l, \mu_r : \FF \to \cA_{0,0}$ defined by
\be
&&\mu_l(\hf)=f(P+h,p)\in \FF[[p]],\qquad \mu_r(\hf)=f(P,p^*)\in \FF[[p]].
\en
Let $\cD=\FF\sharp\C[\cR_Q]$. We define two $H$-algebra homomorphisms, the co-unit $\vep : \cA\to \cD$ and the co-multiplication $\Delta : \cA\to \cA \widetilde{\otimes}\cE$ by
\bea
&&\vep(L_{ij,n})=\delta_{i,j}\delta_{n,0}e^{-Q_{\ep_i} }\quad (n\in \Z),
\qquad \vep(e^{Q_\al})=e^{Q_\al},\quad \al\in \bgH^*\lb{counitUqp}\\
&&\vep(\mu_l({\hf}))= \vep(\mu_r(\hf))=\widehat{f}\;1, \lb{counitf}\\
&&\Delta(L^+_{ij}(z))=\sum_{k}L^+_{ik}(zq^{c^{(2)}})\widetilde{\otimes}
L^+_{kj}(z),\lb{coproUqp}\\
&&\Delta(e^{Q_{\al}})=e^{Q_{\al}}\tot e^{Q_{\al}},
\\
&&\Delta(\mu_l(\hf))=\mu_l(\hf)\widetilde{\otimes} 1,\quad \Delta(\mu_r(\hf))=1\widetilde{\otimes} \mu_r(\hf).\lb{coprof}
\ena

\begin{prop}\lb{counitcopro}
The maps $\vep$ and $\Delta$ satisfy
\bea
&&(\Delta\tot \id)\circ \Delta=(\id \tot \Delta)\circ \Delta,\lb{coaso}\\
&&(\vep \tot \id)\circ\Delta =\id =(\id \tot \vep)\circ \Delta.\lb{vepDelta}
\ena
\end{prop}
We also define  an algebra antihomomorphism (the antipode) $S : \cA\to \cA$ by
\bea
&&S(L^+_{ij}(z))=(L^+(z)^{-1})_{ij},\lb{SonL}\\
&&S(e^Q)=e^{-Q},\quad S(\mu_r(\hf))=\mu_l(\hf),\quad S(\mu_l(\hf))=\mu_r(\hf).
\ena
One then finds that the $H$-algebra $\cA$ equipped with $(\Delta,\vep,S)$ is a $H$-Hopf algebroid\cite{KonnoASPM}.

\section{Finite-Dimensional Representations}
We say that a representation of $U_{q,p}(\glnh)$ (or $E_{q,p}(\glnh)$) is level 0 if the central elements  $q^{\pm c/2}$ act as 1 on it. Here and in the following sections we consider the level-0 highest weight representations of $U_{q,p}(\glnh)$ (or $E_{q,p}(\glnh)$) and $U_{q,p}(\slnh)$.

\subsection{Level-0 highest weight representation}
\begin{prop}\lb{prop:comkls}
In a level-0 representation, the coefficients of $k^+_l(z) \ (1\leq l\leq N)$ (resp. $\psi^+_j(z)  \ (1\leq j\leq N-1)$)  generate a commutative subalgebra of $U_{q,p}(\widehat{\gl}_N)$  (resp. $U_{q,p}(\slnh)$). 
We call it the Gelfand-Tsetlin subalgebra. In particular, $A_m(z)\ (1\leq m\leq N)$ 
are commutative each other.   
\end{prop}
\begin{proof}\ For $U_{q,p}(\widehat{\gl}_N)$, the statement follows from \eqref{kpl}, \eqref{comcElcEl} and \eqref{comcEkcEl}.
For $U_{q,p}({\slnh})$, use \eqref{ellboson} and \eqref{def:psijp}. Noting $p^*=p$ at the level-0 representation, one can prove the statement, i.e. the commutativity of $K^+_l(z)$'s, directly for $E_{q,p}(\glnh)$ 
by using \eqref{RLL}. 
\end{proof}


\begin{dfn}\lb{def:hwr}
A representation $V$ of $\cU=U_{q,p}(\glnh)$ is called a level-0 highest weight representation, if 
there exists $\zeta(\not=0)\in V$ such that 
\begin{itemize}
\item[(1)] $V$ is generated by $\zeta$
\item[(2)] $q^{c/2}\cdot \zeta=\zeta$
\item[(3)] $\ds{e_{j,m}\cdot \zeta=0\qquad \forall j\in \{1,2,\cdots,N-1\},\ m\in \Z}$
\item[(4)] $\ds{k^+_l(z)\cdot \zeta=\la_l(z) \zeta,\quad  l\in \{1,2,\cdots,N\}}$, 
 where $\la_l(z)$ are formal Laurent series in $z$ of the form
 \bea
 &&\la_l(z)=\sum_{m\in \Z_{>0}}\la_{l,-m}z^m+\sum_{m\in \Z_{\geq 0}}\la_{l,m}p^mz^{-m},\qquad
 \lb{def:lal} \\
 &&\la_{l,m}=\sum_{r\geq 0}\la^{(r)}_{l,m}p^r \in \C[[p]],\qquad m\in \Z. \lb{lapowerp}
 \ena
\end{itemize}  
The vector $\zeta$ is called the highest weight vector of $V$, and the $N$-tuple of formal series $\la(z)=(\la_1(z),\cdots,\la_N(z))$ is the highest weight of $V$. 
\end{dfn}

\begin{prop}
The conditions (3) and (4) for the highest weight vector $\zeta$ are equivalent to the 
following ones, respectively. 
\begin{itemize}
\item[(3)'] $\ds{{L}^+_{lk}(z)\cdot \zeta=0, \ \ \ 1 \le k < l \le N}$
\item[(4)'] $\ds{{L}^+_{ll}(z)\cdot \zeta=\lambda_l(z)\zeta, \ \ \ 1 \le l \le N}$.
\end{itemize}
\end{prop}

\noindent
{\it Proof.}\ The statement follows from Lemma \ref{recursion}, Theorem \ref{isoUqpEqp} and 
the formula
\bea
&&L^+_{lk}(z)=K^+_l(z)E^+_{l,k}(z)+\sum_{m>l}^NF^+_{l,m}(z)K^+_m(z)E^+_{m,k}(z),\lb{Lplk}\\
&&L^+_{ll}(z)=K^+_l(z)+\sum_{m>l}^NF^+_{l,m}(z)K^+_m(z)E^+_{m,l}(z).\lb{Lpll}
\ena
\qed

\begin{df}
Let $\la(z)=(\la_1(z),\cdots,\la_N(z))$ be an arbitrary tuple of formal Laurent series of the form 
\eqref{def:lal}. The Verma module $M(\la(z))$ is the quotient of $U_{q,p}(\glnh)$ by the left ideal generated by $e_{j,m}$, $k_{l,m}-\la_{l,m}$ $(1\leq j\leq N-1, 1\leq l\leq N, m\in \Z)$ and $q^{\pm c/2}-1$. 
Equivalently,  $M(\la(z))$ is the quotient of $E_{q,p}(\glnh)$ by the left ideal generated by $\bL_{ji,m}$, $\bL_{ll,m}-\la_{l,m}$ $(1\leq i<j\leq N, 1\leq l\leq N, m\in \Z)$  and $q^{\pm c/2}-1$. 
\end{df}

Let us denote $\cU=U_{q,p}(\glnh)$ or $\cE=E_{q,p}(\glnh)$ by $\cA$.  
By definition, $M(\la(z))$ is a highest weight representation of  $\cA$ 
with the highest weight $\la(z)$ and the highest weight vector $1_{\la(z)}$, which is the image of the element $1$ of $\cA$.   Moreover, if $L$ is a highest weight representation of $\cA$ with the highest weight  $\la(z)$ and the highest vector $\zeta$, then the map $1_{\la(z)} \to \zeta$  defines a surjective $\cA$-module homomorphism $M(\la(z))\to L$. Hence, $L$ is isomorphic to the quotient of $M(\la(z))$ by the kernel of this homomorphism.

\begin{df}
Let $\cN^+$ (resp. $\cN^-$) and $\cH$  be the subalgebras of $\cU$ generated by $e_{j,m}$ (resp. $f_{j,m}$) $(1\leq j\leq N-1, m\in \Z)$ and by $q^{\pm  c/2}$, $k_{l,m}$ $(1\leq l\leq N, m\in \Z)$, respectively. 
\end{df}

\begin{prop}\lb{tridecomp}
We have
\be
&&\cU=\cN^-\,\cH\,\cN^+.
\en
\end{prop}

Thanks to the embeddings \eqref{embeddings}, we regard $M(\la(z))$ as a $U_q(\gl_N)$-module. In particular, under the identification $\cK_{\ep_l}=q^{-h_{\ep_l}}$ \eqref{cKlkl0}, 
we set  for $\nu\in \bgH^* $,  
\bea
&& M(\la(z))_\nu=\{\xi\in M(\la(z)) \ |\ q^{-h_{\ep_l}}
\cdot \xi=q^{-(\nu,\ep_l)}\xi\quad (1\leq l\leq N)\}. 
\ena
For $\la(z)=(\la_1(z),\cdots,\la_N(z))$ with \eqref{def:lal}, let us define $\la\in \bgH^*$ by  
\bea
&&\la^{(0)}_{l,0}=q^{-\la_l},\qquad 1\leq l\leq N. \lb{def:la}
\ena
Here we set $\la_l=(\la,\ep_l)$. 
Then we have 
\be
&&M(\la(z))_{\la}=\C 1_{\la(z)}.
\en
We call $\nu$ a weight of the $U_q(\gln)$-module $M(\la(z))$ if $(M(\la(z)))_\nu\not=\{0\}$ and  set 
\be
&&{\rm Wt}(M(\la(z)))=\{\mu\in \bgH^* \ |\ M(\la(z))_\mu\not=\{0\} \}.
\en
We then have the weight decomposition as a $U_q(\gln)$-module. 
 \bea
&& M(\la(z))=\bigoplus_{\nu\in \bgH^*}M(\la(z))_\nu.
\ena

\begin{df}
The irreducible highest weight representation $L(\la(z))$ of $\cA$  with the highest weight $\la(z)$ is defined as the quotient  of $M(\la(z))$ by the maximal proper submodule.
\end{df}

The following classical result is fundamental through this paper.
\begin{thm}
For $\la\in \bgH^*$, let $L(\la)$  be the irreducible highest weight representation of $U_q(\gln)$ 
with $\zeta(\not=0)\in L(\la)$ such that 
\be
&&L(\la)=U_q(\gln)\cdot \zeta,\\
&&\cX^+_j\cdot \zeta =0, \qquad (1\leq j\leq N-1),\\
&&\cK_{\ep_l}\cdot \zeta=q^{-\la_l}\zeta,\qquad  (1\leq l\leq N).
\en
The representation $L(\la)$ is finite-dimensional if and only if $\la_j-\la_{j+1}
\in \Z_{\geq 0}$ $(1\leq j\leq N-1)$. 
\end{thm}

Noting the embeddings in \eqref{embeddings}, one can prove the following statement in the same way as 
the Yangian case \cite{Molev} and the quantum affine algebra case \cite{CPBook}.
\begin{thm}
Every level-0 finite-dimensional irreducible representation $L$ of $\cA$ is a highest weight representation.  Moreover, $L$ contains a unique highest vector up to a constant factor.
\end{thm}

\noindent
{\it Proof.}\ 
Let us consider the following subspace of $L$. 
\be
&&L^0=\{ \xi\in L\ |\ e_{j,m}\cdot \xi=0 \quad (1\leq j\leq N-1, \ m\in \Z)\}
\en
Then the finiteness of $\dim L$ yields $L^0\not=\{0\}$. 
From Proposition \ref{prop:comkls}, all $k^+_l(z)$ are simultaneously diagonalizable. 
In addition, from Proposition 5.5 in \cite{FKO}, one has
\bea
&&K^+_{\ep_l}e_j(z)=q^{-\delta_{l,j}+\delta_{l,j+1}}e_j(z)K^+_{\ep_l},\qquad 
K^+_{\ep_l}f_j(z)=q^{\delta_{l,j}-\delta_{l,j+1}}f_j(z)K^+_{\ep_l},\lb{Keplfj}\\
&&[\cE^{l}_m, e_j(z)]=\frac{q^{-cm}(q^m-q^{-m})z^m}{m}\frac{1-p^m}{1-p^{*m}}e_j(z)(q^{ m}\delta_{l,j}-\delta_{l,j+1}),\\
&&[\cE^{l}_m, f_j(z)]=-\frac{(q^m-q^{-m})z^m}{m}f_j(z)(q^{m}\delta_{l,j}-\delta_{l,j+1})
\ena
for $1\leq l\leq N, \ 1\leq j\leq N-1$. Hence the action of $k^+_l(z)$ preserves $L_0$.    
Thanks to Proposition \ref{tridecomp}, any simultaneous eigenvector $\zeta\in L^0$ for
these operators will satisfy the conditions of Definition \ref{def:hwr}. 
 
Finally, by Proposition \ref{tridecomp} the vector space $L$ is spanned by the vectors of the form 
\be
&&f_{j_1,m_1}f_{j_2,m_2}\cdots f_{j_r,m_r}\cdot \zeta,\qquad r\in \Z_{\geq 0},\ 
1\leq j_1,\cdots, j_r \leq N-1,\ m_1,\cdots,m_r\in \Z. 
\en
 Then \eqref{Keplfj} implies that the weight space
$L^0$ is one dimensional and spanned by the vector $\zeta$. Moreover, if $\nu$ is a weight of $L$ and $\nu\not=\la$, then $\nu$ strictly precedes $\la$. This proves that the highest vector $\zeta$ of $L$ is determined uniquely, up to a constant factor.
\qed

\begin{df}[\cite{Rosengren}]
We say that $f$ is a theta function of order $n$ and norm $t$ if there exist constants 
$a_1,\cdots, a_n$ and $C$ with $t=a_1\cdots a_n$ such that 
\be
&&f(z)=C\Theta_p(z/a_1)\cdots \Theta_p(z/a_n).
\en
\end{df}
Note that the above condition is equivalent to that $f$ is an entire function satisfying
\bea
&&f(e^{2\pi i}z)=f(z),
\qquad f(pz)=(-)^nt\,
z^{-n}f(z).
\ena

\begin{df}
For $\la\in \bgH^*$, we set $\la_l=(\la,\ep_l)\ (1\leq l\leq N)$ and
\be
&&P^+=\{ \la\in \bgH^*\ |\ \la_j-\la_{j+1}\in \Z_{\geq 0}\quad (1\leq j\leq N-1)\ \}.
\en
For $\la\in P^+$, we set
\be
&&\cP^\la=\{ (P_1(z),\cdots,P_{N-1}(z))\ |\ P_{j}(z)\ \mbox{is a theta function of order $\la_j-\la_{j+1}$}\quad (1\leq j\leq N-1)  \}
\en
and $\cP=\cup_{\la\in P^+}\cP^\la$. 
\end{df}

\begin{thm}\lb{DrinfeldTheta}
Let  $L(\la(z))$ be the level-0 irreducible highest weight representation of $\cA$ with the highest weight 
$\la(z)$ and the highest weight vector $\zeta$. 
 $L(\la(z))$ is finite-dimensional if and only if there exists the $N-1$-tuple of theta functions $(P_1(z),\cdots,P_{N-1}(z))\in \cP^\la$ such that 
\bea
&&\frac{\la_{j}(q^{-j}z)}{\la_{j+1}(q^{-j}z)}=q^{-(\la_j-\la_{j+1})}\frac{P_j(q^2z)}{P_j(z)}.\lb{lajlajp1}
\ena
\end{thm}

\noindent
{\it Proof of only if part.}\ From Theorem  \ref{Uq2Uqp} and Corollary \ref{UqbyUqp}, one can regard $L(\la(z))$ as a $U_{q}(\glnh)$-module. The highest weight vector $\zeta$ in  $L(\la(z))$ satisfies
\be
&&\pi_p(X^+_{j,m}e^{-Q_j})\cdot \zeta=0,\qquad 1\leq j\leq N-1,\ m\in \Z,\\
&&\pi_p(K^\pm_l(z)e^{-Q_{\ep_l}})\cdot \zeta =\La^\pm_l(z)\zeta,
\en 
with certain power series 
\be
&&\La^\pm_l(z)=\sum_{m\in \Z_{\geq0}}\La^\pm_{l,\mp m}z^{\pm m}. 
\en
In particular, one has
\be
&&\La^\pm_{l,0}=q^{\mp \la_l}.
\en
From Sec.3.5 in \cite{Molev}, one can show that, up to a multiplication by a formal power series in $z$ (resp. $z^{-1}$), $\La^+_l(z)$ (resp. $\La^-_l(z)$) $(1\leq l\leq N)$ are polynomials in $z$ (resp. $z^{-1}$).  Furthermore, there exists the $N-1$-tuple of polynomials $(P^+_1(z),\cdots,P^+_{N-1}(z))$  in $z$  with $\deg P^+_j(z)=\la_j-\la_{j+1}$ and $P^+_j(0)=1$,  such that 
\bea
&&\frac{\La^+_j(q^{-j}z)}{\La^+_{j+1}(q^{-j}z)}=q^{-(\la_j-\la_{j+1})}\frac{P^+_j(q^2z)}{P^+_j(z)
}
\ena
as a power series in $z$, as well as
\bea
&&\frac{\La^-_j(q^{-j}z)}{\La^-_{j+1}(q^{-j}z)}=q^{-(\la_j-\la_{j+1})}\frac{P^+_j(q^2z)}{P^+_j(z)
}
\ena
as a power series in $z^{-1}$.  Namely $P^+_j(z)$'s are the Drinfeld polynomials 
which specify the finite-dimensional irreducible $U_q(\glnh)$-module $L(\la(z))$.  
Let us set $r_j=\la_j-\la_{j+1}$ and suppose that $P^+_j(z)$ is factored as
\bea
&& P^+_j(z)=\prod_{k=1}^{r_j}(1-z/a_{j,k}),\qquad a_{j,k}\in \C^\times.
\ena
In particular, as the $U_q(\slnh)$-module  one obtains from \eqref{PhijKjKjp1}
\bea
&&\Phi^+_j(z)\cdot \zeta=q^{-r_j}\frac{P^+_j(q^2z)}{P^+_j(z)}\zeta.
\ena
Then by using \eqref{def:Phipm} and comparing the coefficient in each power of $z$, one finds 
\bea
&&q^{-h_{j}}\cdot \zeta=q^{-r_j}\zeta,\\
&&a_{j,m}\cdot \zeta=\frac{[m]}{m}q^{-m}\sum_{k=1}^{r_j}a_{j,k}^{m}\zeta,\qquad m\in \Z\backslash\{0\}.\lb{aljm}
\ena
Therefore noting  $p=p^*$ on $L(\la(z))$ so that  $\al_{j,m}=a_{l,m}$, one obtains 
from \eqref{def:psijp} 
\bea
&&\psi^+_j(z)\cdot \zeta=q^{-r_j}\prod_{k=1}^{r_j}\frac{\Theta_p(q^2z/a_{j,k})}{\Theta_p(z/a_{j,k})}
\zeta. 
\ena
Hence the only if part follows from \eqref{psijkjkjp1} by taking
\bea
&&P_j(z)=\prod_{k=1}^{r_j}{\Theta_p(z/a_{j,k})}.
\ena
\qed

\begin{cor}
Every level-0 finite-dimensional irreducible representation of $U_{q,p}(\slnh)$ contains a unique, up to a constant factor, vector $\zeta\not=0$ such that
\bea
&&e_j(z)\cdot \zeta=0,\qquad  for\ 1\leq j\leq N-1.
\ena
 Moreover, this vector satisfies
 \bea
 &&\psi^+_j(z)\cdot \zeta=q^{-(\la_j-\la_{j+1})}\frac{P_j(q^2z)}{P_j(z)}\zeta,\qquad 
for 1\leq j\leq N-1,
\ena
where each $P_j(z)$ is a theta function of order $\la_j-\la_{j+1}$.
 The tuple of theta functions $(P_1(z),\cdots,P_{N-1}(z))$ determines the representation up to an isomorphism.
\end{cor}


\medskip

Following \cite{CPBook}, let us denote by $L(\bfP)$ the level-0 finite-dimensional irreducible representation $L(\la(z))$ of $U_{q,p}(\glnh)$ associated to $\bfP=(P_1(z),\cdots,P_{N-1}(z))\in \cP^\la$ through \eqref{lajlajp1}, and say that $\bfP$ is its highest weight.
The following series of statements describe the behavior of the $N-1$-tuple of theta functions ${\bf P}$ under tensor products. For two formal Laurent series $\la(z), \mu(z)$ satisfying \eqref{def:lal} and associated tuples of theta functions ${\bf P} =(P_1(z),\cdots,P_{N-1}(z))\in \cP^\la$, ${\bf Q} = (Q_1(z),\cdots,Q_{N-1}(z)) \in \cP^{\mu}$  with $\la, \mu\in P^+$, let 
$\bfP\otimes\bfQ \in \cP^{\la+\mu}$ be the $N-1$-tuple of theta functions  $(P_1(z)Q_1(z),\cdots,P_{N-1}(z)Q_{N-1}(z))
$. This is related to  $\la(z)\mu(z)=(\la_1(z)\mu_1(z),\cdots,\la_N(z)\mu_N(z))$ in the following way.
 \begin{prop}\lb{LPtotLQ}
Let $\zeta_{\bfP}$ and $\zeta_{\bfQ}$ be $U_{q,p}(\glnh)$-highest weight vectors in $L(\bfP)$ and $L(\bfQ)$, respectively. Then, in $L(\bfP)\tot L(\bfQ)$, we have
\bea
&&\Delta(e_{j,m})\cdot(\zeta_{\bfP}\tot \zeta_{\bfQ})=0, \qquad \\
&&\Delta(k^+_{l}(z))\cdot (\zeta_{\bfP}\tot \zeta_{\bfQ}) = \la_l(z)\mu_l(z)(\zeta_{\bfP}\tot \zeta_{\bfQ}),
\ena
for all $1\leq j\leq N-1$, $m \in\Z$, $1\leq l\leq N$. The formal Laurent series $\la_l(z)\mu_l(z)$ satisfies
\bea
&&\la^{(0)}_{l,0}\mu^{(0)}_{l,0}=q^{-\la_l-\mu_l},\\
&&\frac{\la_j(q^{-j}z)\mu_j(q^{-j}z)}{\la_{j+1}(q^{-j}z)\mu_{j+1}(q^{-j}z)}=\frac{P_j(q^2z)Q_j(q^2z)}{P_j(z)Q_j(z)},\qquad 
1\leq j\leq N-1.
\ena 
\end{prop}

\noindent
\begin{proof}\ The statement follows from the following comultiplication formulas obtained from \eqref{coproUqp}, \eqref{Lplk} and \eqref{Lpll}.
\bea
&&\Delta(L^+_{j,j-1}(z))=\sum_{k<j}^NL^+_{jk}(z)\tot L^+_{kj-1}(z)+\sum_{k\geq j}^NL^+_{jk}(z)\tot L^+_{kj-1}(z),\\
&&\Delta(L^+_{ll}(z))=L^+_{jj}(z)\tot L^+_{jj}(z)+\sum_{k<j}^NL^+_{jk}(z)\tot L^+_{kj}(z)
+\sum_{k>j}^NL^+_{jk}(z)\tot L^+_{kj}(z).
\ena
\end{proof}

\begin{cor}\lb{tensorPQ}
Let $\bfP\in \cP^\la$, $\bfQ\in \cP^\mu$ . Then, $L(\bfP\otimes \bfQ)$ is isomorphic, as a representation of $U_{q,p}(\glnh)$, to a quotient of the subrepresentation of $L(\bfP)\tot L(\bfQ)$ generated by the tensor product of the highest weight vectors in $L(\bfP)$ and $L(\bfQ)$.
\end{cor}
\noindent
{\it Proof.}\ From Proposition \ref{LPtotLQ}, $L(\bfP)\tot L(\bfQ)$ is the finite-dimensional highest weight representation generated by $\zeta_{\bfP}\tot \zeta_{\bfQ}$ so that it has the maximal proper $U_{q,p}(\glnh)$-submodule $K(\la(z))$. The quotient module of  $L(\bfP)\tot L(\bfQ)$ by $K(\la(z))$ gives the unique finite-dimensional irreducible representation associated to $\bfP\otimes \bfQ$, which is nothing but $L(\bfP\otimes\bfQ)$.  
\qed

\begin{df}
We define a representation $L(\bfP)$ of $U_{q,p}(\glnh)$ to be fundamental if  for some $j\in   \{1,\cdots,N-1\}$, $\bfP=(1,\cdots,P_j(z),\cdots,1)$ with $P_j(z)$ being a theta function of order $1$.  
\end{df}

\begin{cor}\lb{tensorFR}
For any $\bfP \in \cP^\la$, the representation $L(\bfP)$ of $U_{q,p}(\glnh)$ is isomorphic to a subquotient of a tensor product of the fundamental representations. 
\end{cor}
\noindent
{\it Proof.}\ Let $L(\bfP_j)$ be the fundamental representations with $\bfP_j=(1,\cdots,P_j,\cdots,1)$ $(1\leq j\leq N-1)$. The statement follows from $\bfP=\bfP_1\otimes \cdots\otimes \bfP_{N-1}$. \qed

\subsection{Proof of the if part of Theorem \ref{DrinfeldTheta}}

The following is immediate from the Definition of the algebras.  
\begin{prop}\lb{autobyf}
For any formal  Laurent series $f(z)$ in $z$ satisfying 
\bea
&&f(z)=-z^{-1/2}\left(1+\sum_{m>0}f_{-m}z^{m}+\sum_{m>0}f_mp^mz^{-m}\right),\lb{condf}
\ena
 one has an automorphism of $E_{q,p}(\glnh)$ by 
 \bea
 &&L^+(z)\mapsto f(z)L^+(z)
 \ena
with the other generators remain the same. This is equivalent to an 
 automorphism of  $U_{q,p}(\glnh)$ defined by
 \bea
 &&k^+_l(z)\mapsto f(z)k^+_l(z)\qquad (1\leq l\leq N)
 \ena
 with the other generators remain the same.  
\end{prop}
\begin{proof}
Note that 
\be
&&k_{l,0}^{(0)}=\bar{L}_{ll,0}=(f(z)k^+_l(z))_0^{(0)}=(f(z)L^+_{ll}(z))^{(0)}_0.
\en
\end{proof}

To show the if part, we follows the argument in \cite{Molev}, Theorem 3.4.1. 
\begin{proof}
Suppose $P_j(z)$ is a theta function of order $r_j\in \Z_{>0}$ satisfying \eqref{lajlajp1} for $1\leq j\leq N-1$. Setting $\bfP=(P_j(z))_{j\in I}$, we have $L(\la(z))=L(\bfP)$.  
Without loss of generality one can set 
\bea
&&P_j(z)=\prod_{k=1}^{r_j} \Theta_p(z/a_{j,k})\qquad 
\ena 
with some $a_{j,k}
 \in \mathbb{C}^\times$. Set $\ds{\sum_{j=1}^{N-1}r_j=n}$. 
For $j\in \{1,\cdots,N-1\}$, define 
\bea
&&P^{(\sum_{a=1}^{j-1}r_a+k)}_j(z)=\Theta_p(z/a_{j,k}) \qquad 1\leq k\leq r_j, 
\ena
 and consider the set of the  $N-1$-tuple of theta functions
\be
&&\bfP^{(s)}=(1,\cdots,1,P^{(s)}_j(z),1,\cdots,1),\qquad \sum_{a=1}^{j-1}r_a+1\leq s\leq \sum_{a=1}^jr_a, 
\en
$s=1,2,\cdots,n$. 
Let $L(\bfP^{(s)})$ be the irreducible highest weight $U_{q,p}(\glnh)$-module  with the highest weight $\bfP^{(s)}$.  There exist the formal Laurent series $\mu^{(s)}(z)=(\mu^{(s)}_1(z),\cdots,\mu^{(s)}_N(z))$ such that 
\bea
&&\frac{\mu^{(s)}_{j}(q^{-j}z)}{\mu^{(s)}_{j+1}(q^{-j}z)}=q^{-1}\frac{P^{(s)}_j(q^2z)}{P^{(s)}_j(z)},\qquad 
\mu^{(s)}_j-\mu^{(s)}_{j+1}=1,\lb{musjmusjp1}\\
&&\frac{\mu^{(s)}_{a}(q^{-a}z)}{\mu^{(s)}_{a+1}(q^{-a}z)}=1,\qquad \mu^{(s)}_a-\mu^{(s)}_{a+1}=0\quad (a\not=j),\lb{musamusap1}
\ena
and $L(\mu^{(s)}(z))=L(\bfP^{(s)})$.  Here we set $(\mu^{(s)}_l)^{(0)}_{0}=q^{-\mu^{(s)}_l}$ as before. In fact, let us define $\mu(z)=(\mu_1(z),\cdots,\mu_N(z))$ by (Proof of Theorem 3.4.1 in   \cite{Molev})
\bea
&&\mu_j(z)=q^{-\la_j}P_1(z)\cdots P_{j-1}(z)P_j(q^2z)\cdots P_{N-1}(q^2z),\lb{def:muj}
\ena
and $\mu^{(s)}_j$ and $a^{(s)}$ $(1\leq s\leq n)$ as follows. 
\bea
&&q^{-2\mu^{(\sum_{a=1}^{i-1}r_j+k)}_j}a^{(\sum_{a=1}^{i-1}r_j+k)}=\left\{\mmatrix{q^2a_{i,k}&\ (i<j)\cr
a_{i,k}&\ (i\geq j)}
\right.,\\
&&\la_j=\sum_{s=1}^{n}\mu^{(s)}_j.
\ena
Then one finds
\bea
&&\mu_j(z)=\prod_{s=1}^n\mu^{(s)}_j(z),\lb{mujprodmusj}\\
&&\mu^{(s)}_j(z)=
q^{-\mu^{(s)}_j}\Theta_p(q^{2(\mu^{(s)}_j+1)}z/a^{(s)})
,\lb{musj}
\ena
and $\mu^{(s)}_j(z)$ satisfy \eqref{musjmusjp1} and \eqref{musamusap1}. 

Thus obtained representation $L(\mu^{(s)}(z))$ is hence an evaluation representations of $U_{q,p}(\glnh)$ associated with the irreducible finite-dimensional highest weight representation $L(\mu^{(s)})$
of $U_q(\gln)$ with $\mu^{(s)}=(\mu^{(s)}_1,\cdots,\mu^{(s)}_N)$ by \eqref{evaluationhom}.  
Noting $\bfP=\bfP^{(1)}\otimes \cdots\otimes \bfP^{(n)}$, from Corollary \ref{tensorFR}, $L(\la(z))$ is isomorphic to a  subquotient of 
\be
&&L(\mu^{(1)}(z))\tot  L(\mu^{(2)}(z))\tot \cdots \tot L(\mu^{(n)}(z)).
\en
Therefore $L(\la(z))$ is finite-dimensional. 

\end{proof}

\section{Representations of $U_{q,p}(\glth)$}
For $U_{q,p}(\glth)$, one can construct all the fundamental representations explicitly. 
Hence Corollary \ref{tensorFR} allows us to classify all the level-0 finite-dimensional irreducible representations.

Let $l$ be a non-negative integer. 
Let $V_l=\oplus_{m=0}^l \C v^l_m$, $V_{l}(a)=V_l[a,a^{-1}]$ with $a\in \C^\times$, where by convention we set $v^l_m=0$ for $m<0$ or $m>l$. 
Let us define the level-0 action $\pi_{l,a}: U_{q,p}(\glth) \to \End_\FF(V_l(a))$ by
(Appendix C \cite{JKOS})
\bea
&&\pi_{l,a}(\al_{n})\cdot v^l_m=\frac{a^n}{n(q-q^{-1})}\left((q^n+q^{-n})q^{(l-2m)n}-(q^{(l+1)n}+q^{-(l+1)n})\right)v^l_m,\\
&&
\pi_{l,a}(e(z))v^l_m=
\frac{q^{-(l-m)}}{1-q^{2}}
\frac{(q^{ 2(l-m+1)};p)_\infty(pq^{-2m};p)_\infty}
{(p;p)_\infty(pq^{- 2};p)_\infty}
\delta\left(q^{l-2m+ 1}\frac{a}{z}\right)v^l_{m-1},
\lb{fd91}\\
&&
\pi_{l,a}(f(z))v^l_m=
\frac{q^{-m}}{1-q^{2}}
\frac{(q^{2(m+1)};p)_\infty(pq^{-2(l-m)};p)_\infty}
{(p;p)_\infty(pq^{2};p)_\infty}
\delta\left(q^{l-2m- 1}\frac{a}{z}\right)v^l_{m+1}.
\lb{fd9}
\ena
In particular, for $\la_1, \la_2\in \C$ satisfying $\la_1-\la_2=l$, one has
\bea
&&\hspace{-0.2cm}\pi_{l,a}(k^+_1(z))\cdot v^l_m=q^{-\la_1+m}\frac{\Theta_p(q^{l}z/a)}{\Theta_p(q^{-l+2m}z/a)}\frac{\Gamma(q^{l}z/a;q^4,p)\Gamma(q^4q^{-l-2}z/a;q^4,p)}{\Gamma(q^4q^{l-2}z/a;q^4,p)\Gamma(q^{-l}z/a;q^4,p)}v^l_m,
\lb{kp1}\\
&&\hspace{-0.8cm}\pi_{l,a}(k^+_2(z))\cdot v^l_m=q^{-\la_2-m}\frac{\Theta_p(q^{-l+2m+2}z/a)}{\Theta_p(q^{l+2}z/a)}\frac{\Gamma(q^4q^{l}z/a;q^4,p)\Gamma(q^4q^{-l-2}z/a;q^4,p)}{\Gamma(q^4q^{l-2}z/a;q^4,p)\Gamma(q^4q^{-l}z/a;q^4,p)}v^l_m.
\lb{kp2}
\ena

\begin{thm}
The $(\pi_{l.a},V_l(a))$ is the $l+1$-dimensional irreducible representation $L(P_{l,a})$ of $U_{q,p}(\glth)$ with the theta function $P_{l,a}(z)$ of order $l$ given by
\bea
&&P_{l,a}(z)=\prod_{k=1}^l\Theta_p(q^{-l+2k-1}z/a).
\ena
In particular, the fundamental representations of $U_{q,p}(\glth)$ are the two dimensional representations $V_1(a)$, for arbitrary $a \in \C^\times$ associated with 
\bea
&&P_{1,a}(z)=\Theta_p(z/a).
\ena
We call $V_l(a)$ an evaluation representation of $U_{q,p}(\glth)$.  
\end{thm}
\noindent
{\it Proof.}\ 
From \eqref{kp1} and \eqref{kp2}, one obtains
\bea
&&\pi_{l,a}(k^+_1(z)) \cdot v^l_0=
{\theta(q^{l+2}z/a)}f(z)^{-1}v^l_0,
\quad\lb{gl2kp1l0}\\
&&\pi_{l,a}(k^+_2(z))\cdot v^l_0=
{\theta(q^{-l+2}z/a)}f(z)^{-1}v^l_0
\lb{gl2kp2l0}
\ena
with 
\be
&&f(z)^{-1}=\frac{1}{\theta(q^{l+2}z/a)}\frac{\Gamma(q^4q^{l}z/a;q^4,p)\Gamma(q^4q^{-l-2}z/a;q^4,p)}{\Gamma(q^4q^{l-2}z/a;q^4,p)\Gamma(q^4q^{-l}z/a;q^4,p)}.
\en
Hence up to an automorphism \eqref{autobyf}, we have 
\be
&&\la_1(z)=\theta(q^{l+2}z/a),
\qquad
\la_2(z)=\theta(q^{-l+2}z/a).
\en
Therefore
\bea
&&
\pi_{l,a}(\psi^+(z))\cdot v^l_0=
\frac{\la_1(q^{-1}z)}{\la_2(q^{-1}z)}v^l_0
=q^{-(\la_1-\la_2)}
\frac{\Theta_p(q^{l+1}z/a)}{\Theta_p(q^{-l+1}z/a)}v^l_0.\nn
\ena
\qed

Let us call the set $\Sigma_{l,a}= \{aq^{-l+1}, aq^{-l+3}, \cdots , aq^{l-1}\}$ of roots of $P_{l,a}$ the $q$-segment of length $l$ and centre $a$\cite{CP1994}. 
 
 \begin{df}
 Two $q$-segments $\Sigma_1$, $\Sigma_2$ are said to be in general position if either
 \begin{itemize}
 \item[(i)] $\Sigma_1\cup \Sigma_2$ is not a $q$-segment, or
 \item[(ii)]  $\Sigma_1\subseteq \Sigma_2$ or  $\Sigma_2\subseteq \Sigma_1$.
 \end{itemize}
 \end{df}

 \begin{prop}\lb{tensorVla}
 Let $l_1,l_2,\cdots,l_n\in \N$, $a_1,a_2,\cdots,a_n \in \C^\times, n \in \N$. 
 Then, the tensor product $V_{l_1}(a_1)\tot V_{l_2}(a_2)\tot \cdots \tot V_{l_n}(a_n)$ is irreducible as a representation of $U_{q,p}(\glth)$ iff the q-segments $\Sigma_{l_1,a_1}, \cdots,\Sigma_{l_n,a_n}$ are pairwise in general position. \end{prop}
 One can prove the statement in the same way as in the case of $U_q(\slth)$\cite{CP} or the Yangian $Y(\glt)$\cite{CPYangian,Molev}.

 \begin{thm}
 Every level-0  finite-dimensional irreducible representation of $U_{q,p}(\slth)$ 
 is isomorphic to a tensor product of evaluation representations.
 \end{thm}
 \noindent
 {\it Proof.}\ Let $L(P)$ be an irreducible finite-dimensional representation associated with the theta function $P(z)$. Consider the multiset $\Sigma$ formed by the zeros of $P(z)$ in the fundamental  parallelogram. 
 Following a proof of Theorem 4.3 in \cite{CP1994}, one has a unique decomposition
\be
&&\Sigma=\Sigma_{l_1,a_1}\cup \Sigma_{l_2,a_2}\cup \cdots \cup \Sigma_{l_n,a_n}  
\en
for some $l_1,l_2,\cdots,l_k\in \N$, $a_1,a_2,\cdots,a_n \in \C^\times, n \in \N$, 
and all $q$-segments $\Sigma_{l_i,a_i}$'s are pairwise in general position. Then  the statement follows from Corollary \ref{tensorPQ} and Proposition \ref{tensorVla} by 
\bea
&&P(z)=P_{l_1,a_1}(z)P_{l_2,a_2}(z)\cdots P_{l_n,a_n}(z).\lb{Pevla}
\ena 
 \qed

\subsection{The Gelfand-Tsetlin bases} 
We give the  Gelfand-Tsetlin basis of the level-0 finite-dimensional irreducible 
representation 
\bea
&&L(P)=V_{l_1}(\ta_1)\tot V_{l_2}(\ta_2)\tot \cdots \tot V_{l_n}(\ta_n),
\ena
where the theta function $P$ is given in \eqref{Pevla} with replacing $a_i$ with 
\be
&&\ta_i=q^{-\al_i-\beta_i+2}a_i.
\en 
Here  $\al_{i},\bet_{i}\ (1\leq i\leq n)$ 
 are complex numbers satisfying 
\bea
&&\al_{i}-\bet_{i}=l_i\in \Z_{\geq 0}.
\ena
The highest weight vector is given by
\bea
&&\zeta=v^{l_1}_0\tot v^{l_2}_0\tot\cdots \tot v^{l_n}_0.
\ena

Let  $\ga=(\ga_1,\cdots,\ga_n)$ be a $n$-tuple of complex numbers satisfying
\bea
 &&\al_{i}- \ga_i \in \Z_{\geq 0},\qquad  \ga_{i}- \bet_i \in \Z_{\geq 0}, \qquad
 (1\leq i\leq n).\lb{condga}
\ena
From \eqref{gl2kp1l0} and \eqref{gl2kp2l0} with $a=\ta_i$ we have 
\be
&&\pi_{l_i,a_i}(k^+_1(z))\cdot v^{l_i}_0=\theta(q^{2\al_i}z/a_i)
f(z)^{-1}v^{l_i}_0,\qquad\\
&&\pi_{l_i,a_i}(k^+_2(z))\cdot v^{l_i}_0=\theta(q^{2\beta_i}z/a_i)
f(z)^{-1}
v^{l_i}_0.
\en
Hence up to an automorphism of the form \eqref{autobyf}, one has
\bea
&&k^+_1(z)\cdot \zeta=\prod_{i=1}^n\theta(q^{2\al_i}z/a_i)
\zeta,\lb{gl2kp1z}\\
&&k^+_2(z)\cdot \zeta=\prod_{i=1}^n\theta(q^{2\beta_i}z/a_i)
\zeta.\lb{gl2kp2z}
\ena
Let $A_1(z), A_2(z), B_2(z), C_2(z)$ be operators defined in Definition \ref{def:ABC} for the $N=2$ case.
Define 
\bea
&&\xi_{\ga}=\prod_{i=1}^nB_2(q^{-2(\ga_i-1)}a_i)B_2(q^{-2(\ga_i-2)}a_i)\cdots B_2(q^{-2\bet_{i}}a_i)\cdot \zeta,\\
&&\xi_{\bet}=\zeta,
\ena 
where $\bet=(\bet_{1},\cdots,\bet_{k})$. 

\begin{thm}\lb{gl2GTbasis}
The set of vectors $\{\xi_\ga\}$ with $\ga$ satisfying \eqref{condga} forms a basis of $L(P)$. Moreover, the generators of $U_{q,p}(\glth)$ act on $\xi_\ga$ as follows
\bea
&&A_1(z)\cdot \xi_\ga=\prod_{i=1}^n\theta(q^{2\al_i}z/a_i)\theta(q^{2(\bet_i-1)}z/a_i)
\, \xi_\ga,\lb{A1xiga}\\
&&A_2(z)\cdot \xi_\ga=\prod_{i=1}^n\theta(q^{2\ga_i}z/a_i)\, \xi_\ga,\lb{A2xiga}\\
&&B_2(q^{-2\ga_j}a_j)\cdot \xi_\ga=\xi_{\ga+\delta_j}, \lb{Bxiga}\\
&&C_2(q^{-2\ga_j}a_j)\cdot \xi_\ga=-\Delta^{(n-1)}\left(\frac{\theta(q^2\Pi^*)}{\theta(\Pi^*)}\right)\cdot
\prod_{i=1}^n\theta(q^{2(\al_i-\ga_j+1)})\theta(q^{2(\bet_i-\ga_j)})\, \xi_{\ga-\delta_j},\lb{Cxiga}
\ena
for $j\in \{1,\cdots,n\}$, where we set $\ga\pm \delta_j=(\ga_1,\cdots,\ga_j\pm 1,\cdots,\ga_n)$. The vector $\xi_{\ga\pm\delta_j}$ is considered to be zero if 
$\ga\pm\delta_j$ does not satisfy \eqref{condga}.  

\end{thm}
\noindent
{\it Proof.}\ 
From \eqref{AmK},\eqref{comAmBn}, \eqref{gl2kp1z} and \eqref{gl2kp2z}, one obtains \eqref{A1xiga}. 

\eqref{A2xiga} can be proved inductively on the number of operators $B_2(\bullet)$ in $\xi_\ga$.  Noting $\xi_\ga=B_2(q^{-2(\ga_1-1)}a_1)\cdot\xi_{\ga-\delta_1}$, from \eqref{comAmBm}, one has
\be
A_2(z)B_2(q^{-2(\ga_1-1)}a_1)\cdot\xi_{\ga-\delta_1}
&=&\frac{\theta(q^2q^{2(\ga_1-1)}z/a_1)}{\theta(q^{2(\ga_1-1)}z/a_1)}B_2(q^{-2(\ga_1-1)}a_1)A_2(z)\cdot \xi_{\ga-\delta_1}\nn\\
&&\qquad-\frac{\theta(q^2)\theta(\Pi^{-1}q^{2(\ga_1-1)}z/a_1)}{\theta(\Pi^{-1})\theta(q^{2(\ga_1-1)}z/a_1)}B_2(z)A_2(q^{-2(\ga_1-1)}a_1)\cdot \xi_{\ga-\delta_1}.
\en
The second term vanishes by the induction hypothesis for some $j$ 
\bea
&&A_2(z)\cdot \xi_{\ga-\delta_j}=\theta(q^{2(\ga_j-1)}z/a_j)\prod_{i=1\atop \not=j}^n\theta(q^{2\ga_i}z/a_i)\, \xi_{\ga-\delta_j}.\lb{xigamdj}
\ena

\eqref{Bxiga} follows from the definition of $\xi_\ga$. 

\eqref{Cxiga} with $\ga_j> \bet_j$ follows from $\xi_\ga=B_2(q^{-2(\ga_j-1)}a_j)\cdot\xi_{\ga-\delta_j}$ and \eqref{xigamdj} and \eqref{comCmBm}.
\be
&&C_2(q^{-2\ga_j}a_j)\cdot\xi_\ga=C_2(q^{-2\ga_j}a_j)B_2(q^{-2(\ga_j-1)}a_j)\cdot\xi_{\ga-\delta_j}\nn\\
&&=A_2(q^{-2\ga_j}a_j)f^+_{12}(q^{-2(\ga_j-1)}a_j)e^+_{21}(q^{-2\ga_j}a_j)A_2(q^{-2(\ga_j-1)}a_j)\cdot \xi_{\ga-\delta_j}\nn\\
&&\qquad +A_1(q^{-2(\ga_j-1)}a_j)
\Delta^{(n-1)}\hspace{-0.1cm}\left(\frac{\theta(q^2)\theta(\Pi^{*-1}q^{-2})}{\theta(\Pi^{*-1})\theta(q^{-2})}\right)
\cdot \xi_{\ga-\delta_j}\nn\\
&&\qquad -A_1(q^{-2\ga_j}a_j)A_{2}(q^{-2}q^{-2\ga_j}a_j)^{-1}A_{2}(q^{-2(\ga_j-1)}a_j)\cdot \Delta^{(n-1)}\hspace{-0.1cm}\left(\frac{\theta(q^2)\theta(\Pi^{-1}q^{-2})}{\theta(\Pi^{-1})\theta(q^{-2})}\right)
\cdot \xi_{\ga-\delta_j}\nn\\
&&=-
\Delta^{(n-1)}\hspace{-0.1cm}\left(\frac{\theta(\Pi^{*}q^{2})}{\theta(\Pi^{*})}\right)
\prod_{i=1}^n\theta(q^{2(\al_i-\ga_j+1)})\theta(q^{2(\bet_i-\ga_j)}) \xi_{\ga-\delta_j}.
\en
Note that $p^*=p$ on $L(P)$. 

For the case $\ga_j=\bet_j$ for some $j$, one has
\be
&&C_2(q^{-2\bet_j}a_j)\cdot \xi_\ga=0.
\en 
In fact, let us take $j=1$ and consider $\ga=(\bet_1,\ga_2,\cdots,\ga_n)$ 
without loss of  generality. Then one has 
\be
&&C_2(q^{-2\bet_1}a_1)\cdot \xi_{(\bet_1,\ga_2,\cdots,\ga_n)}\\
&&=C_2(q^{-2\bet_1}a_1)B_2(q^{-2(\ga_2-1)}a_2)\cdot \xi_{\ga-\delta_2}\\
&&=A_2(q^{-2\bet_1}a_1)B_2(q^{-2(\ga_2-1)}a_2)A_2(q^{-2(\ga_2-1)}a_2)^{-1}
A_2(q^{-2\bet_1}a_1)^{-1}C_2(q^{-2\bet_1}a_1)A_2(q^{-2(\ga_2-1)}a_2)\xi_{\ga-\delta_2}\\
&&+A_2(q^{-2\bet_1}a_1)A_1(q^{-2(\ga_2-1)}a_2)A_2(q^{-2(\ga_2-1)}q^{-2}a_2)^{-1}
\Delta\left(\frac{\theta(q^2)\theta(\Pi^{*-1}q^{-2\bet_1}q^{2(\ga_2-1)}a_1/a_2)}
{\theta(\Pi^{*-1})\theta(q^{-2\bet_1}q^{2(\ga_2-1)}a_1/a_2)}\right)\xi_{\ga-\delta_2}\\
&&-A_1(q^{-2\bet_1}a_1)A_2(q^{-2\bet_1}q^{-2}a_1)^{-1}A_2(q^{-2(\ga_2-1)}a_2)
\Delta\left(\frac{\theta(q^2)\theta(\Pi^{-1}q^{-2\bet_1}q^{2(\ga_2-1)}a_1/a_2)}
{\theta(\Pi^{-1})\theta(q^{-2\bet_1}q^{2(\ga_2-1)}a_1/a_2)}\right)\xi_{\ga-\delta_2}.
\en
This vanishes because $A_2(q^{-2(\ga_2-1)}a_2)$ and $A_1(q^{2\bet_1}a_1)$ vanish 
on $\xi_{\ga-\delta_2}$.

We now prove that the vectors $\xi_\ga$  form a basis of $L(P)$. 
Firstly, $\xi_\ga\not=0$ follows from Lemma \ref{C2B2} given in the below.  
Secondly, the vectors $\xi_\ga$ are linearly independent because they are eigenvectors for $A_2(z)$ with distinct eigenvalues. The number of these vectors is
\be
&&\prod_{i=1}^n ( \al_i -\bet_i + 1) 
\en
which coincides with the dimension of $L(P)$, thus proving the claim.

It remains to verify that $B_2(q^{-2\al_j}a_j)\xi_{\ga}=0$. The argument used for the proof of \eqref{A2xiga} shows that if the vector
$\xi'=B_2(q^{-2\al_j}a_j)\xi_\ga$  is nonzero, then $\xi'$ is an eigenvector for 
$A_2(z)$ with the eigenvalue 
\be
&&\theta(q^{2\ga_1}z/a_1)\cdots\theta(q^{2\ga_{j-1}}z/a_{j-1})
\theta(q^{2(\al_j+1)}z/a_j) \theta(q^{2\ga_{j+1}}z/a_{j+1})\cdot 
\theta(q^{2\ga_n}z/a_n).
\en
However, as we have seen above, the module  $L$ admits a basis which consists of eigenvectors of $A_2(z)$ with distinct eigenvalues. We come to a contradiction since none of these eigenvalues coincides with the eigenvalue of 
$\xi'$. So, $\xi'=0$.

\qed

\begin{lem}\lb{C2B2}
\bea
&&\stackrel{\curvearrowleft}{\prod_{1\leq i\leq n}}
C_2(q^{-2(\bet_i+1)}a_i)\cdots C_2(q^{-2(\ga_i-1)}a_r)C_2(q^{-2\ga_i}a_i)
\cdot \xi_\ga\nn\\
&&\qquad=(-)^{\sum_{i=1}^n(\ga_i-\bet_i)}\Delta^{(n-1)}\hspace{-1mm}\left(\frac{\theta(\Pi^*q^{2\sum_{i=1}^n(\ga_i-\bet_i)})}{\theta(\Pi^*)}\right)\nn\\
&&\qquad\times\prod_{i=1}^n\prod_{j=1}^n\prod_{a_j=0}^{\ga_j-\bet_j-1}
\theta(q^{2(\al_i-\ga_j+a_j+1)})\theta(q^{2(\bet_i-\ga_j+a_j)})
\, \zeta.
\ena

\end{lem}
\noindent
{\it Proof.}\ Use \eqref{Cxiga} repeatedly. \qed

\section{The Gelfand-Tsetlin Bases for $U_{q,p}(\glnh)$-modules}\lb{sec:GTB}

In this section, we consider general  level-0 finite-dimensional irreducible representations specified by the Gelfand-Tsetlin patterns, and construct their bases.  
Our construction is an elliptic analogue of the one studied by Nazarov-Tarasov \cite{NT}.


Let $k^+_l(z), A_l(z), B_m(z), C_m(z)\ (1\leq l\leq N,\ 2\leq m\leq N)$ as in Sec.\ref{sec:qdet}.  
Recall that on the level-0 representations, we have
\be
&&p^*=p,\quad \theta^*(z)=\theta(z),\quad \cN_k=\cN'_k=1\ (1\leq k\leq N). 
\en

For $s\in \{1,\cdots,n\}$, let $L(\la^{(s)}(z))$ be a level-0 finite-dimensional irreducible representation of $U_{q,p}(\glnh)$ with the highest weight $\la^{(s)}(z)=(\la^{(s)}_1(z),\cdots,\la^{(s)}_N(z))$. We consider an irreducible representation $L(\la(z))$ given as a subquotient of the tensor product.  
\bea
&&L(\la^{(1)}(z))\tot\cdots\tot L(\la^{(n)}(z)).
\ena
The highest weight $\la(z)$ is given by 
\bea
&&\la(z)=(\la_1(z),\cdots,\la_N(z)),\\
&&\la_j(z)=\prod_{s=1}^n\la^{(s)}_j(z).
\ena
From Theorem \ref{DrinfeldTheta}, 
there exist theta functions $P^{(s)}_j(z)$ of order $r^{(s)}_j\equiv \la^{(s)}_j-\la^{(s)}_{j+1}\in \Z_{\geq 0}$ such that 
\bea
&&\frac{\la^{(s)}_j(q^{-j}z)}{\la^{(s)}_{j+1}(q^{-j}z)}=q^{-r^{(s)}_j}\frac{P^{(s)}_j(q^2z)}{P^{(s)}_j(z)}.
\ena 
In the following we consider the finite-dimensional irreducible representation $L(\la(z))$ whose highest weight is given by $\bfP=(P_1(z),\cdots,P_{N-1}(z))$
\bea
&&P_j(z)=\prod_{s=1}^nP^{(s)}_j(z),\qquad P^{(s)}_j(z)=\prod_{k=1}^{r^{(s)}_j}\Theta_p(q^{2(\la^{(s)}_j-k)}q^{-j}z/a^{(s)}) \lb{DTglN}
\ena
without loss of generality. 
Hence up to an automorphism \eqref{autobyf}, we have 
\bea
&&\la^{(s)}_j(z)=\theta(q^{2\la^{(s)}_j}z/a^{(s)}).
\ena

We call the following set of complex numbers  $\la^{(s)}_{l,j}\ (1\leq  j\leq l \leq N)$ the Gelfand-Tsetlin pattern $\bla^{(s)}$ associated with $\la^{(s)}=(\la^{(s)}_1,\cdots,\la^{(s)}_N)$, if
\be
&&\la^{(s)}_{N,j}=\la^{(s)}_j\qquad (1\leq j\leq N),\\
&&\la^{(s)}_{l,j}- \la^{(s)}_{l-1,j}\in \Z_{\geq 0},\quad  \la^{(s)}_{l-1,j}- \la^{(s)}_{l,j+1}
\in \Z_{\geq 0}.
\en
Schematically we have
\be
\bla^{(s)}&=&\mmatrix{
\la^{(s)}_{N,1}&&\la^{(s)}_{N,2}&\cdots&\cdots&\cdots&\la^{(s)}_{N,N-1}&&\la^{(s)}_{N,N}\quad\cr
      &\la^{(s)}_{N-1,1}&&\la^{(s)}_{N-1,2}&\cdots&\cdots&&\ \ \ \la^{(s)}_{N-1,N-1}&\cr
      &\quad\ddots&&\ddots&&&&\ \rotatebox[origin=c]{85}{$\ddots$}\quad&\cr
      &                &\ddots&&\ddots&&&\rotatebox[origin=c]{85}{$\ddots$}\qquad\qquad&\cr
      &&&&&&&&\cr
      &     &                 &         &\la^{(s)}_{2,1}&&\la^{(s)}_{2,2}&&\cr
      &&&&&&&&\cr
      &&&&&\quad \la^{(s)}_{1,1}&&&\cr
}
\en
In particular, we denote by $\bla^{(s)}_0$ the Gelfand-Tsetlin pattern where 
$\la^{(s)}_{l,j}=\la^{(s)}_{N,j}\ (j\leq l\leq N-1)$ for each $j\in \{1,\cdots,N-1 \}$. 

To describe the Gelfand-Tsetlin bases,
it is convenient to set $\nu_i^{(s)}:=\lambda_{N+1-i}^{(s)}$\ 
$(1\leq s\leq n)$ and  $\nu_i(z):=\lambda_{N+1-i}(z)$.
Hence we have
\begin{align}
\nu_j(z)=\prod_{s=1}^n\theta(q^{2\nu_j^{(s)}}z/a^{(s)}) 
\ \ \qquad 1\leq j\leq N.
\end{align}
The Gelfand-Tsetlin patterns 
$\bnu^{(s)}=(\nu_{l,j}^{(s)})_{1\leq j\leq l\leq N}$, $s=1,\dots,n$ are collections of complex numbers $\nu_{l,j}^{(s)}$ 
$(1 \le j \le l \le N)$ satisfying
\bea
&&\nu_{N,j}^{(s)}=\nu_{j}^{(s)},\lb{condnu1}\\
&&\nu_{l,j}^{(s)}-\nu_{l-1,j-1}^{(s)} \in \mathbb{Z}_{\geq 0},\quad
\nu_{l-1,j-1}^{(s)} - \nu_{l,j-1}^{(s)} \in \mathbb{Z}_{\geq0}.\lb{condnu2}
\ena
In the following, to make the presentation simple, we write the condition \eqref{condnu2} as $\nu_{l,j}^{(s)}\geq 
\nu_{l-1,j-1}^{(s)} \geq \nu_{l,j-1}^{(s)} $.  Schematically we have
\be
\bnu^{(s)}&=&\mmatrix{
\nu^{(s)}_{N,N}&&\nu^{(s)}_{N,N-1}&\cdots&\cdots&\cdots&\nu^{(s)}_{N,2}&&\nu^{(s)}_{N,1}\quad\cr
      &\nu^{(s)}_{N-1,N-1}&&\nu^{(s)}_{N-1,N-2}&\cdots&\cdots&&\ \ \nu^{(s)}_{N-1,1}&\cr
      &\quad\ddots&&\ddots&&&&\ \rotatebox[origin=c]{85}{$\ddots$}\quad&\cr
      &                &\ddots&&\ddots&&&\rotatebox[origin=c]{85}{$\ddots$}\qquad\qquad&\cr
      &&&&&&&&\cr
      &     &                 &         &\nu^{(s)}_{2,2}&&\nu^{(s)}_{2,1}&&\cr
      &&&&&&&&\cr
      &&&&&\quad \nu^{(s)}_{1,1}&&&\cr
}
\en
By definition the set of all Gelfand-Tsetlin patterens $\{\bla^{(s)}\}$ has a bijection to the one of all $\{\bnu^{(s)}\}$ for each $s\in \{1,\cdots,n\}$. 
 
We introduce the following order for the set of indices $\cI=\{(l,j),\  1 \le j \le l \le N\ \}$:
\begin{align}
(\ell,j) \prec (\ell',j^\prime) \Leftrightarrow
j < j^\prime \ \mathrm{or} \ (j=j^\prime \ \mathrm{and} \ \ell> \ell').
\end{align}

Let us also set $\overline{A}_m(z):=A_{N+1-m}(z)$, $\overline{B}_m(z):=B_{N+1-n}(z)$. 
We rewrite the commutation relations \eqref{comAmBm}, \eqref{comAmBn} and \eqref{comCmBm} as
\begin{align}
&[\overline{A}_m(z),\overline{B}_n(w)]=0, \ \ \ (m \neq n), \label{newcrtwo}\\
&\overline{A}_m(z) \overline{B}_m(w)
=\frac{\theta(q^2z/w)}{\theta(z/w)} \overline{B}_m(w) \overline{A}_m(z)
-\frac{\theta(q^2)\theta(\Pi^{-1}_{N-m,N-m+1}z/w)}{\theta(\Pi^{-1}_{N-m,N-m+1})\theta(z/w)}
\overline{B}_m(z) \overline{A}_m(w)
, \label{newcrone} \\[1mm]
&\overline{C}_m(z)\overline{B}_m(w) \nn \\
&=\overline{A}_m(z)\overline{B}_m(w)\overline{A}_m(w)^{-1}\overline{A}_m(z)^{-1}\overline{C}_m(z)\overline{A}_m(w)
\nn \\
&\quad+\overline{A}_m(z)\overline{A}_{m+1}(w)\overline{A}_m(w q^{-2})^{-1}\overline{A}_{m-1}(wq^{-2})\frac{\bar{c}(\Pi^*_{N-m,N-m+1}z/w)}{\overline{b}(z/w)}
 \nn \\
&\quad-\overline{A}_{m-1}(z q^{-2})\overline{A}_{m+1}(z)\overline{A}_m(z q^{-2})^{-1}\overline{A}_m(w)
\frac{\bar{c}(\Pi_{N-m,N-m+1}z/w)}{\overline{b}(z/w)} .
\label{combCmbBm}
\end{align}

From  Theorem \ref{isoUqpEqp} and \eqref{AmK}, one has 
\be
&&\bA_m(z)=k^+_{N+1-m}(z)k^+_{N+1-(m-1)}(zq^{-2})\cdots k^+_N(zq^{-2(m-1)})
\en
on the level-0 representation $L(\la(z))$. Note that $\cN_k=1$ on it.  
Hence the action of $\bA_m(z)$ on the highest weight vector $\zeta$ 
is given by
\begin{align}
\overline{A}_m(z) \cdot\zeta&=\nu_1(q^{-2(m-1)}z) \nu_2(q^{-2(m-2)}z)
\cdots \nu_m(z) \zeta \nonumber \\
&=\prod_{s=1}^n \prod_{t=1}^m \theta(q^{2(\nu_t^{(s)}+t-m)}z/a^{(s)}) \zeta.
\label{actionhwrewrite}
\end{align}

We define a vector ${\xi}_{\bnu}$ labeled by the $n$-tuple of Gelfand-Tsetlin patterns $\bnu=(\bnu^{(1)},\bnu^{(2)},\dots,\bnu^{(n)})$ with 
$\bnu^{(s)}=(\nu_{\ell,j}^{(s)})_{1\leq j\leq l\leq N}$ as
\begin{align}
{\xi}_{\bnu}
:=\stackrel{\curvearrowright}
{\prod_{(\ell,j)\in \cI
}
}
\Bigg(\prod_{s=1}^n
\prod_{\substack{1 \le t \le \nu_{\ell, j}^{(s)}-\nu_{N,j}^{(s)}}}
\overline{B}_\ell(q^{-2(\nu_{N,j}^{(s)}+t-\ell+j-1)}a^{(s)})
\Bigg) \cdot\zeta. \label{GZNTtypebasis}
\end{align}
We show that a set of vectors $\{\xi_{\bnu}\}$ with $\bnu$ satisfying \eqref{condnu1} and \eqref{condnu2} forms a basis of the representation $L(\nu(z))$. 

For this purpose, it is crucial to note the following facts. 
\begin{itemize}
\item[(1)] Let $\bnu^{(s)}_0$ $(1\leq s\leq n)$ be the Gelfand-Tsetlin  pattern satisfying  
$\nu_{j,j}^{(s)}=\nu_{j+1,j}^{(s)}=\cdots=\nu_{N,j}^{(s)}$ for $1\leq j\leq N$, 
and set $\bnu_0=(\bnu^{(1)}_0,\cdots,\bnu^{(n)}_0)$. Then one has 
\be
&&{\xi}_{\bnu_0}=\zeta. 
\en
\item[(2)] For any  Gelfand-Tsetlin pattern $\bnu\neq \bnu_0$, there exist a set of integers $\ell, j, r$ satisfying $2\leq j\leq \ell\leq N$,\ $1\leq r\leq n$ and $\bnu=\bnu^{(\ell,j;r)}$, where 
 $\bnu^{(\ell,j;r)}$ denotes the $n$-tuple of the  Gelfand-Tsetlin patterns consisting of $\nu^{(s)}$ $(1\leq s\leq n)$ satisfying
\begin{itemize}
\item[i)] $\nu_{k,k}^{(s)}=\nu_{k+1,k}^{(s)}=\cdots=\nu_{N,k}^{(s)}$ for $1\leq k\leq j-1$, \ $1\leq s\leq n$
\item[ii)] $\nu_{\ell,j}^{(r)} > \nu_{\ell+1,j}^{(r)}=\cdots=\nu_{N,j}^{(r)}$ and 
$\nu_{\ell,j}^{(s)} \ge \nu_{\ell+1,j}^{(s)}=\cdots=\nu_{N,j}^{(s)}$
for all $s(\not=r)
$. 
\end{itemize}
Namely $(\ell,j)$ is the minimum pair in $\cI$ satisfying $\nu^{(r)}_{\ell,j}>\nu^{(r)}_{N,j}$. So we  have 
\be
\xi_{\bnu^{(\ell,j;r)}}&=&\bB_\ell(q^{-2(\nu^{(r)}_{\ell,j}-\ell+j-1)}a^{(r)})
\bB_\ell(q^{-2(\nu^{(r)}_{\ell,j}-\ell+j-2)}a^{(r)})\cdots \bB_\ell(q^{-2(\nu^{(r)}_{N,j}-\ell+j)}a^{(r)})\nn\\
&&\qquad \times\stackrel{\curvearrowright}
{\prod_{(\ell,j)\prec (\ell',j')}}
\Bigg(\prod_{s=1\atop \neq r}^n
\prod_{\substack{1 \le t \le \nu_{\ell', j'}^{(s)}-\nu_{N,j'}^{(s)}}}
\overline{B}_{\ell'}(q^{-2(\nu_{N,j'}^{(s)}+t-\ell'+j'-1)}a^{(s)})
\Bigg) \cdot\zeta.
\en
\end{itemize}  
\begin{proposition} \label{eigenvalueprop}
For a Gelfand-Tsetlin pattern $\bnu\neq \bnu_0$, let $\ell,j, r$  be the set of integers satisfying the 
conditions in (2). 
We have the following actions of operators. 
\begin{align}
&\overline{A}_m(z)\cdot 
{\xi}_{\bnu}
=\prod_{s=1}^n \prod_{t=1}^m \theta(q^{2(\nu_{m,t}^{(s)}+t-m)}z/a^{(s)})
\xi_{\bnu},
\label{eigenvectors}
\\
&\bB_\ell(q^{-2(\nu^{(r)}_{(\ell,j)}-\ell+j)}a^{(r)})\cdot \xi_{\bnu
}=\xi_{\bnu+\delta^{(r)}_{(\ell,j)}},\lb{actbB}
\\
&\bC_\ell(q^{-2(\nu^{(r)}_{(\ell,j)}-\ell+j)}a^{(r)})\cdot \xi_{\bnu
}\nn\\
&\quad =-\Delta^{(n-1)}\left(\frac{\theta(q^2\Pi^*_{N-\ell,N+1-\ell})}{\theta(\Pi^*_{N-\ell,N+1-\ell})}\right)\nn\\
&\qquad\times\prod_{s=1}^n\left(\prod_{t=1}^{\ell+1}
\theta(q^{2(\nu^{(s)}_{\ell+1,t}-\nu^{(r)}_{\ell,j}+t-j)})
\prod_{u=1}^{\ell-1}\theta(q^{2(\nu^{(s)}_{\ell-1,u}-\nu^{(r)}_{\ell,j}+u-j+1)})\right)\xi_{\bnu
-\delta^{(r)}_{(\ell,j)}}.\lb{actbC}
\end{align}
In particular, the coefficient of \eqref{actbC} does not vanish if $\nu^{(r)}_{N,N}-\nu^{(s)}_{N,N}\neq \Z$ modulo the fundamental parallelogram for $r\neq s$.  
\end{proposition}

\begin{proof}
Let us consider \eqref{eigenvectors}. This can be proved in the same way as \cite{NT}, Theorem 3.2.
We show by induction on the number of operators
$\overline{B}_\ell(q^{-2(\nu_{N,j}^{(s)}+t-\ell+j-1)}a^{(s)})$
in \eqref{GZNTtypebasis}. The case $\bnu=\bnu_0$, the statement follows from  \eqref{actionhwrewrite}. 
Next, for $\bnu\neq \bnu_0$, 
 one can assume 
$\bnu=\bnu^{(\ell,j;r)}$ without loss of generality. 
Set 
\bea
&&\bnu^{(r)}-\delta_{(\ell,j)}:=(\nu^{(r)}_{1,1},\cdots,\nu^{(r)}_{\ell,j}-1,\cdots,\nu^{(r)}_{N,N})
\ena
and 
\bea
&&\bnu^{(\ell,j;r)}-\delta^{(r)}_{(\ell,j)}:=(\bnu^{(1)},\cdots,\bnu^{(r)}-\delta_{(\ell,j)},\cdots,\bnu^{(n)}).
\ena
One has the following relation. 
\begin{align}
\xi_{\bnu^{(\ell,j;r)}}
=
\overline{B}_\ell(q^{-2(\nu_{\ell,j}^{(r)}-\ell+j-1)}a^{(r)})\cdot
\xi_{\bnu^{(\ell,j;r)}-\delta^{(r)}_{(\ell,j)}}.
\label{relationNTbasis}
\end{align}
Let us set 
\begin{align}
\overline{a}_{m,
\bnu}(z):=\prod_{s=1}^n \prod_{t=1}^m \theta(q^{2(\nu_{m,t}^{(s)}+t-m)}z/a^{(s)}).
\label{eigenvalues}
\end{align}
For $m(\neq \ell)$, by the induction assumption
\begin{align}
\overline{A}_m(z)\cdot
\xi_{\bnu^{(\ell,j;r)}-\delta^{(r)}_{(\ell,j)}
}=
\overline{a}_{m,\bnu^{(\ell,j;r)}-\delta^{(r)}_{(\ell,j)}
}(z)
\xi_{\bnu^{(\ell,j;r)}-\delta^{(r)}_{(\ell,j)}
}, \label{inductiveassumption}
\end{align}
and the commutativity of $\bA_m(z)$  with $\bB_\ell(\bullet)$, one has
\begin{align}
\overline{A}_m(z)\cdot \xi_{\bnu^{(\ell,j;r)}
}
&=\overline{A}_m(z)
\overline{B}_\ell(q^{-2(\nu_{\ell,j}^{(r)}-\ell+j-1)}a^{(r)})\cdot
\xi_{\bnu^{(\ell,j;r)}-\delta^{(r)}_{(\ell,j)}
} \nonumber \\
&=\overline{B}_\ell(q^{-2(\nu_{\ell,j}^{(r)}-\ell+j-1)}a^{(r)})
\overline{A}_m(z)\cdot
\xi_{\bnu^{(\ell,j;r)}-\delta^{(r)}_{(\ell,j)}
} \nonumber \\
&=\overline{a}_{m,\bnu^{(\ell,j;r)}-\delta^{(r)}_{(\ell,j)}
}(z)\overline{B}_\ell(q^{-2(\nu_{\ell,j}^{(r)}-\ell+j-1)}a^{(r)})
\cdot\xi_{\bnu^{(\ell,j;r)}-\delta^{(r)}_{(\ell,j)}
} \nonumber \\
&=\overline{a}_{m,\bnu^{(\ell,j;r)}
}(z)
\xi_{\bnu^{(\ell,j;r)}
}.
\end{align}
 In the last line, we used  the equality
\begin{align}
\overline{a}_{m,
\bnu^{(\ell,j;r)}
}(z)=\overline{a}_{m,
\bnu^{(\ell,j;r)}-\delta^{(r)}_{(\ell,j)}
}(z). \label{sameeigenvalue}
\end{align}
This is  due to the fact that 
$\ell$ is the largest index of the $\bB$-operators appearing in $\txi_{\bnu^{(\ell,j;r)}-\delta^{(r)}_{(\ell,j)}}$.

For  $m=\ell$, note the following relation
\begin{align}
\overline{a}_{m,\bnu^{(\ell,j;r)}
}(z)=
\frac{\theta(q^{2(\nu_{m,j}^{(r)}-m+j)}z/a^{(r)})}{\theta(q^{2(\nu_{m,j}^{(r)}-m+j-1)}z/a^{(r)})
}
\overline{a}_{m,\bnu^{(\ell,j;r)}-\delta^{(r)}_{(\ell,j)}
}(z), \label{relationeigenvalues}
\end{align}
from \eqref{eigenvalues}.
From the inductive assumption
\eqref{inductiveassumption}, we get
\begin{align}
&\overline{A}_\ell(q^{-2(\nu_{\ell,j}^{(r)}-\ell+j-1)}a^{(r)})\cdot
\xi_{\bnu^{(\ell,j;r)}-\delta^{(r)}_{(\ell,j)}
} 
=\overline{a}_{\ell,\bnu^{(\ell,j;r)}-\delta^{(r)}_{(\ell,j)}
}(q^{-2(\nu_{\ell,j}^{(r)}-\ell+j-1)}
a^{(r)})
\xi_{\bnu^{(\ell,j;r)}-\delta^{(r)}_{(\ell,j)}
}=0. \label{zeroaction}
\end{align}
Setting $w$ in \eqref{newcrone} to $w=q^{-2(\nu_{\ell,j}^{(r)}-\ell+j-1)}a^{(r)}$
and making both sides act on  $\xi_{\bnu^{(\ell,j;r)}-\delta^{(r)}_{(\ell,j)}}$, 
we get from \eqref{relationeigenvalues} and \eqref{zeroaction}, 
\begin{align}
&\overline{A}_\ell(z)
\overline{B}_\ell(q^{-2(\nu_{\ell,j}^{(r)}-\ell+j-1)}a^{(r)})\cdot
\xi_{\bnu^{(\ell,j;r)}-\delta^{(r)}_{(\ell,j)}
} \nonumber \\
&=\frac{\theta(q^{2(\nu_{\ell,j}^{(r)}-\ell+j)}z/a^{(r)})}{\theta(q^{2(\nu_{\ell,j}^{(r)}-\ell+j-1)}z/a^{(r)})}
\overline{B}_\ell(q^{-2(\nu_{\ell,j}^{(r)}-\ell+j-1)}a^{(r)})
\overline{A}_\ell(z)\cdot
\xi_{\bnu^{(\ell,j;r)}-\delta^{(r)}_{(\ell,j)}
} \nonumber \\
&=\frac{\theta(q^{2(\nu_{\ell,j}^{(r)}-\ell+j)}z/a^{(r)})}{\theta(q^{2(\nu_{\ell,j}^{(r)}-\ell+j-1)}z/a^{(r)})}
\overline{B}_\ell(q^{-2(\nu_{\ell,j}^{(r)}-\ell+j-1)}a^{(r)})
\overline{a}_{\ell,\bnu^{(\ell,j;r)}-\delta^{(r)}_{(\ell,j)}
}(z)
\xi_{\bnu^{(\ell,j;r)}-\delta^{(r)}_{(\ell,j)}
} \nonumber \\
&=\overline{a}_{\ell,\bnu^{(\ell,j;r)}
}(z)\overline{B}_\ell(q^{-2(\nu_{\ell,j}^{(r)}-\ell+j-1)}a^{(r)})
\xi_{\bnu^{(\ell,j;r)}-\delta^{(r)}_{(\ell,j)}
}. \label{beforeendofproof}
\end{align}
Then the statement follows from \eqref{relationNTbasis}. 

The statement \eqref{actbB} follows from the definition of $\xi_{\bnu}$.

Using \eqref{combCmbBm}, \eqref{relationNTbasis}, \eqref{eigenvectors} and \eqref{zeroaction}, one obtains \eqref{actbC}. The non-vanishment of \eqref{actbC} follows from 
due to \eqref{condnu2} and the conditions $i)$ and $ii)$ for $\bnu^{(\ell,j;r)}$. 
Namely the latter 2 conditions yield $\nu^{(r)}_{\ell,j}>\nu^{(r)}_{\ell+1,j}$ and 
$\nu^{(r)}_{\ell+1,j}\geq \nu^{(r)}_{\ell,j-1}=\nu^{(r)}_{\ell-1,j-1}$.

\end{proof}

The following is an elliptic version of  Theorem 5.2.4 in \cite{Molev}.
\begin{theorem} \label{basistheorem}
If $ \nu_{N,N}^{(r)} - \nu_{N,N}^{(s)} \not\in \mathbb{Z}$ modulo the fundamental parallelogram for $r \neq s$,
the vectors $\xi_{\bnu}$ with 
 $\bnu=(\bnu^{(1)},\dots,\bnu^{(n)})$ satisfying \eqref{condnu1} and \eqref{condnu2} form a basis of the representation $L(\la(z))$.
\end{theorem}
\begin{proof}
Note that by definition we have $\la(z)=\nu(z)$ and $\la^{(s)}(z)=\nu^{(s)}(z)$ $(1\leq s\leq n)$,  
and  $\ds{\dim L(\la(z))\leq \prod_{s=1}^n \dim L(\la^{(s)}(z))}$. 
From \eqref{eigenvectors},  the vectors $\xi_{\bnu}$ labeled by the different Gelfand-Tsetlin patterns belong to the different simultaneous eigenvalues of $\bA_\ell(z)$ $(1\leq \ell\leq N)$, so that they are linearly 
independent.   
By applying \eqref{actbC} repeatedly, one can remove all $\bB_\ell$-operators from $\xi_{\bnu}$ inductively,   
and obtains $\zeta$ with non-zero coefficient under the assumption.  Hence 
we have $\xi_{\bnu} \neq 0$. 
The total number of the vectors $\xi_{\bnu}$, i.e.
the total number of the Gelfand-Tsetlin patterns is $\prod_{k=1}^n \mathrm{dim} L(\la^{(k)}(z))$. 
This is due to the fact that the finite-dimensional irreducible representation $L(\la^{(s)}(z))$ of $U_{q,p}(\glnh)$ is obtained by the evaluation homomorphism \eqref{evaluationhom} to the $U_q(\gl_N)$-module $L(\la^{(s)})$ and to the classical result that $\dim L(\la^{(s)})$ coincides with the number of the Gelfand-Tsetlin patterns $\bla^{(s)}$. 
Finally, to keep the condition $\nu^{(r)}_{\ell+1,j+1}-\nu^{(r)}_{\ell,j}\in \Z_{\geq 0}$ one needs to show  $\ds{\bB_\ell(q^{-2(\nu^{(r)}_{\ell+1,j+1}-\ell+j)}a^{(r)})\cdot \xi_{\bnu}=0}$. In fact, this vanishment follows from the same argument as in the proof of Theorem \ref{gl2GTbasis}. 
\end{proof}

\color{black}

\section{Tensor Product of the Vector Representations}
In this section,
we consider the tensor product of the vector representations and 
construct the Gelfand-Tsetlin bases in terms of the $L$-operator. 
We derive the relation between them and the ones in 
in the previous section
for the case of tensor product of the vector representations.

Before considering the vector representation, we prepare the following. 
\begin{dfn}
We call that a vector $\eta$ is
a singular vector of weight $\mu(z)=(\mu_k(z),\dots,\mu_{N}(z))$
with respect to the subalgebra
$E_{q,p}(\widehat{\gl}_{N-k+1})$ (see \eqref{seqsubalg}) if it satisfies
\begin{align}
L^+_{ij}(z) \cdot \eta&=0, \ \ \ k \le j < i \le N, \\
L^+_{ii}(z) \cdot \eta&=\mu_i(z)\eta, \ \ \ k \le i \le N.
\end{align}
\end{dfn}
We have the following statement analogous to \cite{Molev} Lemma 5.2.1.
\begin{lemma} \label{singularlemma} 
For $k\in \{1,2,\cdots,N\}$, let $\eta$ be a singular vector of weight $\mu(z)=(\mu_k(z),\dots,\mu_{N}(z))$
with respect to the subalgebra $E_{q,p}(\widehat{\gl}_{N-k+1})$. 
Assume that $\eta$ satisfies $L^+_{kk}(\alpha) \cdot \eta=0$
for some $\alpha \in {\mathbb C}^\times$.
Then $L^+_{k-1,k}(\alpha) \cdot \eta$ is also a singular vector
with respect to  $E_{q,p}(\widehat{\gl}_{N-k+1})$,
with weight given by
\begin{align}
\Bigg(
\frac{\theta(q^2 z/\alpha)}{\theta(z/\alpha)} \mu_k(z),
\mu_{k+1}(z), \dots, \mu_N(z)
\Bigg).
\end{align}
\end{lemma}

\begin{proof}
What we need to show are the following:
\begin{align}
L^+_{ij}(z) L^+_{k-1,k}(\alpha) \cdot \eta&=0, \ \ \
k \le j < i \le N, \label{toshowsingularone} \\
L^+_{kk}(z) L^+_{k-1,k}(\alpha) \cdot \eta&=
\frac{1
}{\overline{b}(z/\alpha)} \mu_k(z)
L^+_{k-1,k}(\alpha) \cdot \eta, \label{toshowsingulartwo} \\
L^+_{ii}(z) L^+_{k-1,k}(\alpha) \cdot \eta&=
\mu_i(z) L^+_{k-1,k}(\alpha) \cdot \eta, \ \ \
k+1 \le i \le N. \label{toshowsingularthree}
\end{align}
These relations can be shown by using the properties and assumption
for the singular vector $\eta$,
together with the following relations which are particular matrix elements
of the defining relations \eqref{RLL}
\begin{align}
&\overline{b}(z_1/z_2) L_{ij}^+(z_1)
L_{k-1,k}^+(z_2)
+
\overline{c}(z_1/z_2,\Pi_{k-1,i}) L_{k-1,j}^+(z_1)
L_{ik}^+(z_2) \nonumber \\
=&L_{k-1,j}^+(z_2) L_{ik}^+(z_1)
c(z_1/z_2,\Pi^*_{kj})
+
L_{k-1,k}^+(z_2) L_{ij}^+(z_1)
\overline{b}(z_1/z_2), \ \ \ k+1 \le j \le i \le N,
\label{matrixlelementsforsinularone}
\\
&\overline{b}(z_1/z_2) L_{ik}^+(z_1)
L_{k-1,k}^+(z_2)
+
\overline{c}(z_1/z_2,\Pi_{k-1,i}) L_{k-1,k}^+(z_1)
L_{ik}^+(z_2) \nonumber \\
=&L_{k-1,k}^+(z_2) L_{ik}^+(z_1)
, \ \ \ k \le i \le N. \label{matrixlelementsforsinulartwo}
\end{align}

\eqref{toshowsingularone} for $k+1 \le j < i \le N$
can be derived by setting $z_1=z$, $z_2=\alpha$ in
\eqref{matrixlelementsforsinularone} and acting both hand sides on $\eta$.
We get
\begin{align}
&\overline{b}(z/\alpha) L_{ij}^+(z)
L_{k-1,k}^+(\alpha) \cdot \eta
+
\overline{c}(z/\alpha,\Pi_{k-1,i}) L_{k-1,j}^+(z)
L_{ik}^+(\alpha) \cdot \eta \nonumber \\
=&L_{k-1,j}^+(\alpha) L_{ik}^+(z)
c(z/\alpha,\Pi^*_{kj}) \cdot \eta
+
L_{k-1,k}^+(\alpha) L_{ij}^+(z)
\overline{b}(z/\alpha) \cdot \eta \nonumber \\
=&c(z/\alpha,\Pi^*_{kj})
L_{k-1,j}^+(\alpha) L_{ik}^+(z) \cdot \eta
+\overline{b}(z/\alpha)
L_{k-1,k}^+(\alpha) L_{ij}^+(z) \cdot \eta,
\end{align}
and using $L_{ik}^+(\alpha) \cdot \eta=
L_{ik}^+(z) \cdot \eta=L_{ij}^+(z) \cdot \eta=0
$ since $\eta$ is a singular vector and $\overline{b}(z/\alpha) \neq 0$,
we get \eqref{toshowsingularone} for $k+1 \le j < i \le N$.

\eqref{toshowsingularone} with $j=k$ follows from
\eqref{matrixlelementsforsinulartwo} for $k+1 \le i \le N$.
We set $z_1=z$, $z_2=\alpha$ in
\eqref{matrixlelementsforsinulartwo} and act both hand sides on $\eta$
to get
\begin{align}
&\overline{b}(z/\alpha) L_{ik}^+(z)
L_{k-1,k}^+(\alpha) \cdot \eta
+
\overline{c}(z/\alpha,\Pi_{k-1,i}) L_{k-1,k}^+(z)
L_{ik}^+(\alpha) \cdot \eta 
=
L_{k-1,k}^+(\alpha) L_{ik}^+(z) \cdot \eta,
\end{align}
from which we find \eqref{toshowsingularone} when $j=k$ follows by noting
${L}_{ik}^+(\alpha) \cdot \eta={L}_{ik}^+(z) \cdot \eta=0$ and $\overline{b}(z/\alpha) \neq 0$.

\eqref{toshowsingulartwo} follows from the case $k=i$ of
\eqref{matrixlelementsforsinulartwo}.
Setting $z_1=z$, $z_2=\alpha$ and acting on $\eta$ yields
\begin{align}
&\overline{b}(z/\alpha) L_{kk}^+(z)
L_{k-1,k}^+(\alpha) \cdot \eta
+
\overline{c}(z/\alpha,\Pi_{k-1,k}) L_{k-1,k}^+(z)
L_{kk}^+(\alpha) \cdot \eta 
=
L_{k-1,k}^+(\alpha) L_{kk}^+(z) \cdot \eta,
\end{align}
and using $L_{kk}^+(z) \cdot \eta=\mu_k(z) \eta$, $\overline{b}(z/\alpha) \neq 0$
and the additional assumption on the singular vector
$L_{kk}^+(\alpha) \cdot \eta=0$ gives
\eqref{toshowsingulartwo}.

Finally, \eqref{toshowsingularthree} follows from
the case $j=i$ of
\eqref{matrixlelementsforsinularone}.
Setting $z_1=z$, $z_2=\alpha$ and acting on $\eta$,
we get
\begin{align}
&\overline{b}(z/\alpha) L_{ii}^+(z)
L_{k-1,k}^+(\alpha) \cdot \eta
+
\overline{c}(z/\alpha,\Pi_{k-1,i}) L_{k-1,i}^+(z)
L_{ik}^+(\alpha) \cdot \eta \nonumber \\
=&c(z/\alpha,\Pi^*_{ki})
L_{k-1,i}^+(\alpha) L_{ik}^+(z) \cdot \eta
+\overline{b}(z/\alpha)
L_{k-1,k}^+(\alpha) L_{ii}^+(z) \cdot \eta,
\end{align}
and we note \eqref{toshowsingularthree} follows by
using $L_{ik}^+(\alpha) \cdot \eta=
L_{ik}^+(z) \cdot \eta=0$, $L_{ii}^+(z) \cdot \eta=\mu_i(z) \eta$
and $\overline{b}(z/\alpha) \neq 0$.

\end{proof}

\subsection{The vector representation}
Let  $(\pi_w,\widehat{V}_w)$ be the vector representation of $U_{q,p}(\widehat{\gl}_N)$ with 
$\widehat{V}=\bigoplus_{\mu=1}^N\FF v_\mu$ and $\widehat{V}_w= \widehat{V}[w,w^{-1}]$. 
See \cite{KonnoJintone}, B.2. We assume $e^{ Q_\alpha} v_\mu=v_\mu$ ($\alpha \in \bgH^*
$). The  action of the $L$-operator ${L}^+(z)$ on $\widehat{V}_w$ is given by
\begin{align}
\displaystyle \pi_w(L_{ij}^{+}(z))v_\nu
=\pi_w (L_{ij}^{+} (z,\Pi^*))e^{-Q_{\overline{\epsilon}_j}}
v_\nu
=\sum_{\mu=1}^N
&{\tR}(z/w,\Pi^*)_{i \mu}^{j \nu} v_\mu
\label{levelzeroactionofL}
\end{align}
where we take 
\begin{align}
&\tR(z,\Pi^*)=a(z) {\bR}(z,\Pi^*)\lb{def:tR}
\end{align}
with  $a(z)=\theta(q^2z)$. 
See Figure \ref{figDR}.
Then the  action of $L^+(z)$ on $\widehat{V}_z \widetilde{\otimes} 
\widehat{V}_{w_1} \widetilde{\otimes} \cdots
\widetilde{\otimes} \widehat{V}_{w_n}$ is given by the following (\cite{KonnoJinttwo}, Proposition 4.1). 
\begin{proposition}  \label{coproduct}
\begin{align}
&(\pi_{w_1} \otimes \cdots \otimes \pi_{w_n})
\Delta^{\prime (n-1)} (L^+(z)) \nonumber \\
=&\overline{R}^{+(0n)}(z/w_n,\Pi^* q^{2 \sum_{j=1}^{n-1} h^{(j)}})
\overline{R}^{+(0 n-1)}(z/w_{n-1},\Pi^* q^{2 \sum_{j=1}^{n-2} h^{(j)}})
\cdots
\overline{R}^{+(01)}(z/w_1,\Pi^*).
\label{actioncoproduct}
\end{align}
Here $\Delta'$ denotes the opposite comultiplication of $\Delta$ in Sec.\ref{sec:Hopfalgebroid}. 
\end{proposition}
For brevity, we denote $(\pi_{w_1} \otimes \cdots \otimes \pi_{w_n})
\Delta^{\prime (n-1)} (L^+(z))$  by $L^+(z)$.  We set
\be
&&L^+(z)=\sum_{k,l=1}^NE_{k,l}\tot L^+_{kl}(z)
\en
with $L^+_{ij}(z)\in \End_\FF(
\widehat{V}_{w_1} \widetilde{\otimes} \cdots
\widetilde{\otimes} \widehat{V}_{w_n})$.

\begin{figure}[htbp]
\centering
\includegraphics[width=12cm]{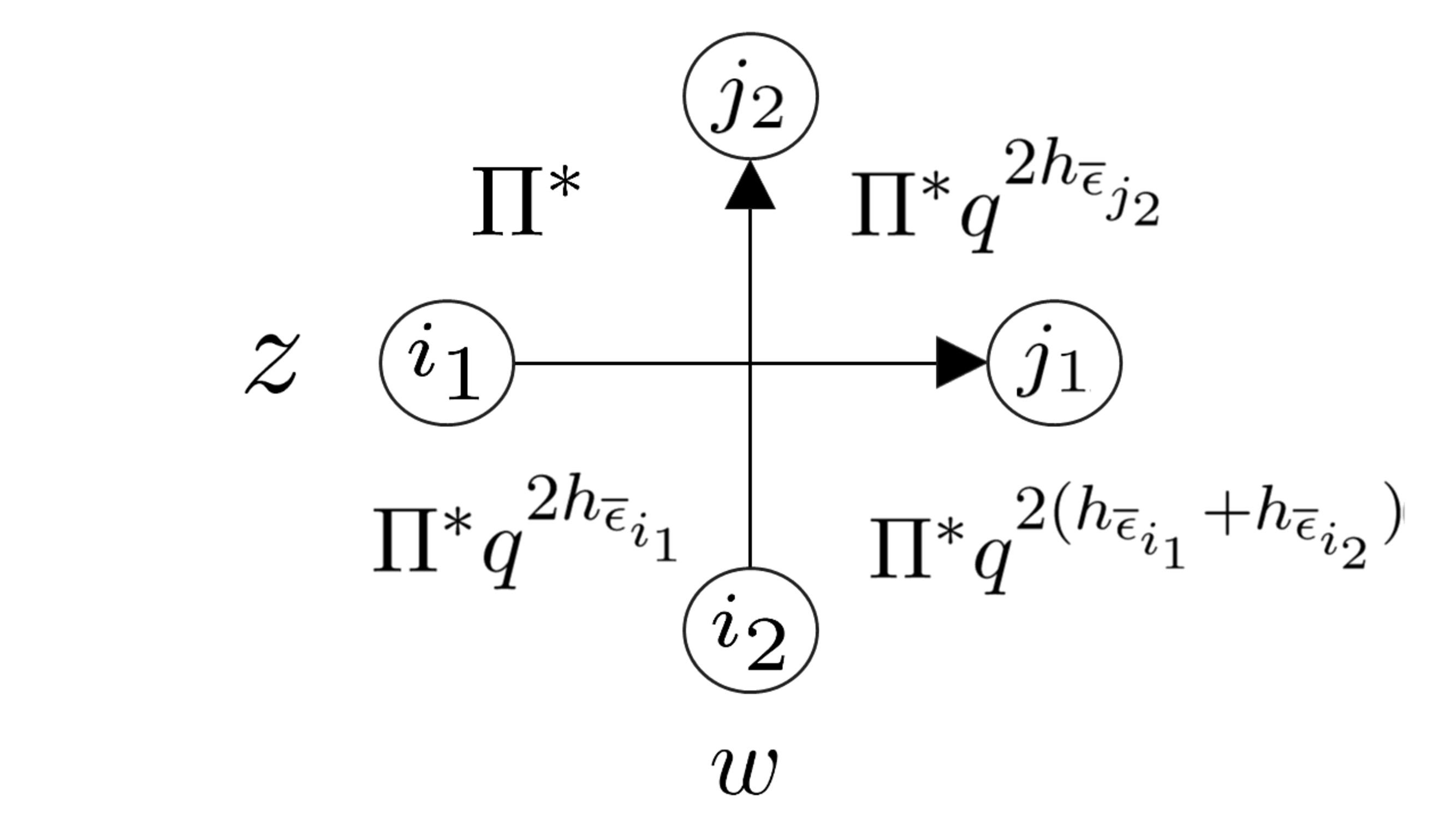}
\caption{Graphical description of the matrix elements $\tR(z/w,\Pi^*)_{i_1i_2}^{j_1j_2}$.
}
\label{figDR}
\end{figure}

Now let us consider the action of $L^+_{kl}(z)$ on the standard basis  of
$\widehat{V}_{w_1} \widetilde{\otimes} \cdots
\widetilde{\otimes} \widehat{V}_{w_n}$ i.e. 
$\{ v_{\bmu}:=v_{\mu_1} \widetilde{\otimes} \cdots \widetilde{\otimes} v_{\mu_n}
\ | \ \bmu=({\mu_1},\dots,{\mu_n})\in [1,N]^n \ 
 \}$.
We then set  for $\bal\in [1,N]^n$ 
\begin{align}
\displaystyle
L^+_{kl}(z) \cdot v_{\boldsymbol{ \alpha }}
=\sum_{\boldsymbol{\beta}\in [1,N]^n
} T_{kl}^+(z;w_1,\dots,w_n;\Pi^*)_{\boldsymbol{\beta}}^{\boldsymbol{\alpha}} v_{\boldsymbol{ \beta }}.
\end{align}
Making this action repeatedly, one obtains 
\begin{align}
\displaystyle
&L^+_{k_m l_m}(z_m) L^+_{k_{m-1} l_{m-1}}(z_{m-1})
\cdots
L^+_{k_1 l_1}(z_1) v_{\boldsymbol{\alpha}} \nonumber \\
=&\sum_{
\boldsymbol{\beta}\in [1,N]^n
}
Z_{\bfK
\boldsymbol{\beta}}^{\bfL
\boldsymbol{\alpha}}
(z_1,\dots,z_m;w_1,\dots,w_n;\Pi^*)
v_{\boldsymbol{\beta}}\lb{LsonSB}
\end{align}
with 
\begin{align}
&Z_{\bfK
\boldsymbol{\beta}}^{\bfL
\boldsymbol{\alpha}}
(z_1,\dots,z_m;w_1,\dots,w_n;\Pi^*) \nonumber \\
=&
\sum_{\boldsymbol{\alpha}_{1},\dots,\boldsymbol{\alpha}_{m-1}\in [1,N]^n
} T_{k_m l_m}^+(z_m;z_1,\dots,z_n;\Pi^*)_{\boldsymbol{\beta}}^{\boldsymbol{\alpha}_{m-1}}
T_{k_{m-1} l_{m-1}}^+(z_{m-1};w_1,\dots,w_n;\Pi^*q^{2\bra \bep_{l_m},h\ket
})_{\boldsymbol{\alpha}_{m-1}}^{\boldsymbol{\alpha}_{m-2}}
\cdots \nonumber \\
&\qquad\qquad \cdots T_{k_1 l_1}^+(z_1;w_1,\dots,w_n;\Pi^*q^{2\sum_{i=2}^{m}\bra\bep_{l_i},h\ket
})_{\boldsymbol{\alpha}_{1}}^{\boldsymbol{\alpha}},
\label{partitionfunctionnotation}
\end{align}
where $
\bfK=(k_1,\dots,k_m)$, $\bfL=
(l_1,\dots,l_m)$, $\boldsymbol{\alpha}=(\alpha_1,\dots,\alpha_n)$, $\boldsymbol{\beta}=(\beta_1,\dots,\beta_n)\in [1,N]^n$. Note that from \eqref{actioncoproduct} each coefficient $Z_{\bfK\bbe}^{\bfL\bal}$ is given as a product of $R$-matrices $\tR^+(z,\Pi^*)$ with the dynamical shifts.  See Figure \ref{figurepartitionfunction}. 
The dynamical parameters in the matrix element $\tR^+(z,\Pi^*q^{2\bra\bep_{a},h\ket}
)_{ij}^{kl}$ should be treated as $\Pi^*_{i,j}q^{2\bra\bep_{a},h_{i,j}\ket}=
\Pi^*_{i,j}q^{2(\delta_{a,i}-\delta_{a,j})}
$. 

\begin{figure}[h] 
\centering
\includegraphics[width=12cm]{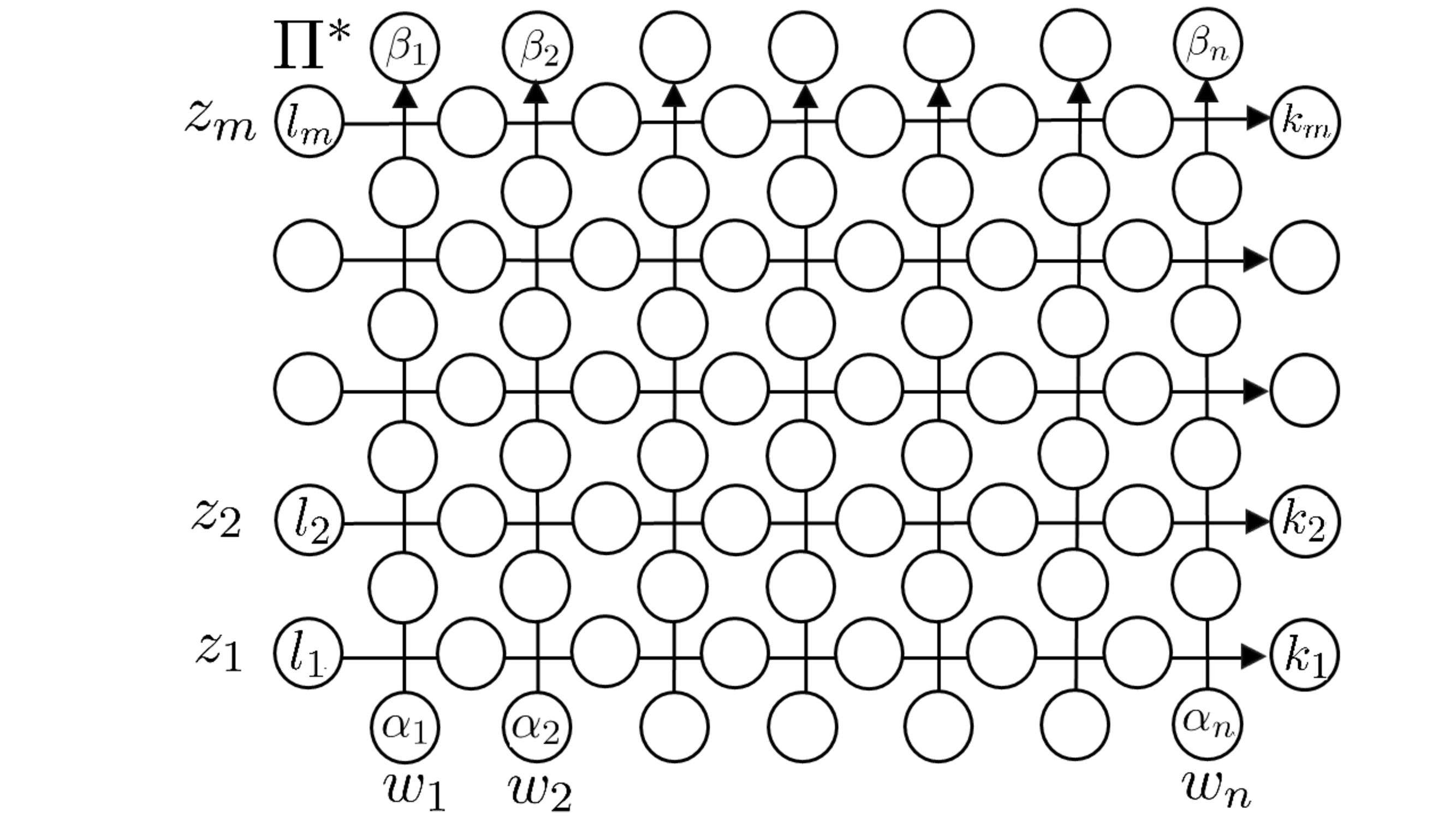}
\caption{ Graphical description of $Z_{\bfK
\boldsymbol{\beta}}^{\bfL
\boldsymbol{\alpha}}
(z_1,\dots,z_m;w_1,\dots,w_n;\Pi^*)$.
}
\label{figurepartitionfunction}
\end{figure}

\subsection{The Gelfand-Tsetlin bases for $\widehat{V}_{w_1} \widetilde{\otimes} \cdots
\widetilde{\otimes} \widehat{V}_{w_n}$
}\lb{GTBtensor}
Let us consider $\cV=\widehat{V}_{w_1} \widetilde{\otimes} \cdots
\widetilde{\otimes} \widehat{V}_{w_n}$ with generic evaluation parameters
$w_1,\cdots,w_n$.  By the action \eqref{actioncoproduct}, $\cV$ is a $N^n$-dimensional irreducible 
representation of  $U_{q,p}(\widehat{\gl}_N)$.  It is the highest weight representation with 
the highest vector $\zeta=v_{(1^n)}$. Here $(1^n)=(1,\cdots,1)$. 
Using Proposition \ref{coproduct},
we can check
\begin{align}
\displaystyle
L_{11}^+(z) \cdot \zeta=\prod_{l=1}^n a(z/w_l) \zeta
=\prod_{l=1}^n \theta(q^2 z/w_l) \zeta,
\end{align}
and
\begin{align}
\displaystyle
L_{ii}^+(z) \cdot \zeta=\prod_{l=1}^n a(z/w_l)\overline{b}(z/w_l) \zeta
=\prod_{l=1}^n \theta(z/w_l) \zeta,
\end{align}
for $i=2,\dots,N$,
hence the highest weight $\lambda(z)=(\lambda_1(z),\dots,\lambda_N(z))$ is given by
\begin{align}
\lambda_1(z)&=\prod_{l=1}^n \theta(q^2 z/w_l), \\
\lambda_i(z)&=\prod_{l=1}^n \theta(z/w_l) ,\ \ \ i=2,\dots,N.
\end{align}

In order to describe the Gelfand-Tsetlin bases, 
let us introduce  a partition $I=(I_1,\dots,I_N)$ of $[1,n]=\{1,\dots,n \}$, i.e.
\begin{align}
I_1 \cup \cdots \cup I_N=[1,n], \ \ \ I_k \cap I_l =\phi \ (k \neq l).
\end{align}
Let $\cI_n$ be a set of all partitions of $[1,n]$.  
For $\bmu=(\mu_1,\mu_2,\dots,\mu_n)\in [1,N]^n$, we define 
$I_l=\{ i \in [1,n] \ | \ \mu_i=l \}$ ($l=1,\dots,N$). Then $I=(I_1,\dots,I_N)\in \cI_n$. 
We often write thus obtained $I$ as $I=I_\bmu$, and the corresponding standard base as 
$v_I=v_\bmu=v_{\mu_1}\tot\cdots\tot v_{\mu_n}$. 

In addition, for an index set $S=\{i_1,i_2,\dots,i_r \}$ with $i_1<i_2<\dots<i_r$ and the $L$-operators $L^+_{kl}(z_{i_a})$ $(i_a\in S)$,  
we write
\be
&&L^+_{kl}(z_S):=\prod_{j \in S}^{\curvearrowleft} L^+_{kl}(z_j)= L^+_{kl}(z_{i_r})\cdots L^+_{kl}(z_{i_2})L^+_{kl}(z_{i_1}).
\en

Now we construct the Gelfand-Tsetlin bases in terms of the $L$-operators.
\begin{dfn}
For $I\in \cI_n$, let us define a vector $\txi_I\in \cV$ by
\begin{align}
\txi_I
=L_{N-1,N}^+(w_{I_N})
L_{N-2,N-1}^+(w_{I_{N-1} \cup I_N})
\cdots
L_{12}^+(w_{I_2 \cup \cdots \cup I_N}) \cdot \zeta.
\label{GZvectorrepresentation}
\end{align}
\end{dfn}
\noindent
{\it Example.}\ Consider the case $N=3$, $n=9$, $I=I_{231213231}$.
Then
$I_1=\{3<5<9 \}$, $I_2=\{1<4<7 \}$, $I_3=\{2<6<8 \}$. 
We have 
\begin{align}
\txi_I&=L_{23}^+(w_{I_3})L_{12}^+(w_{I_2 \cup I_3}) \cdot \zeta \nn \\
&=
L_{23}^+(w_8) L_{23}^+(w_6) L_{23}^+(w_2)
L_{12}^+(w_8) L_{12}^+(w_7) L_{12}^+(w_6) L_{12}^+(w_4) L_{12}^+(w_2) L_{12}^+(w_1)
\cdot \zeta. 
\end{align}

We show that the set of vectors $\{\txi_I\}_{I\in \cI_n}$ forms a basis of $\cV$, on which all the generators $A_\ell(z)$  of the Gelfand-Tsetlin subalgebras of $E_{q,p}(\glnh)$ are simultaneously diagonalized.  
\begin{proposition}\lb{eigenAl}
The vector $\txi_I$ is an eigenvector of $A_l(z)$, $l=1,\dots,N$
\begin{align}
A_l(z) \cdot \txi_I
=\prod_{k=1}^{N-l+1} \lambda_{N-l+1,k}^I(q^{-2k+2}z) \txi_I,
\label{actionaoperators}
\end{align}
where $\lambda_{jk}^I(z)$, $1 \le k \le j \le N$
are given by
\begin{align}
\lambda_{j1}^I(z)&=\prod_{m=1}^n \theta(q^{2\delta_{m \in I_{N-j+1} \cup \cdots \cup I_N}} z/w_m), \\
\lambda_{jk}^I(z)&=\prod_{m=1}^n \theta(z/w_m), \ \ \ k \neq 1.
\end{align}
Here 
\be
&&\delta_{s\in S}=\left\{\mmatrix{1&\ (s\in S)\cr
0&\ (s\not\in S)\cr}\right.
\en
for any index set $S$. 

\end{proposition}

\begin{proof}

Define $\txi_I^{(l)}$ as
\begin{align}
\txi_I^{(l)}
:=L_{N-l,N-l+1}^+(w_{I_{N-l+1} \cup \cdots \cup I_N})
L_{N-l-1,N-l}^+(w_{I_{N-l} \cup \cdots \cup I_N})
\cdots
L_{12}^+(w_{I_2 \cup \cdots \cup I_N}) \cdot \zeta.
\end{align}
Then $\txi_I$ can be expressed as
\begin{align}
\txi_I
=L_{N-1,N}^+(w_{I_N}) \cdots
L_{l,l+1}^+(w_{I_{l+1} \cup \cdots \cup I_N}) \cdot
\txi_I^{(N-l+1)}. \label{singulardecomposition}
\end{align}

We first examine the action of $A_l(z)$ on $\txi_I^{(N-l+1)}$. 
By applying Lemma \ref{singularlemma} repeatedly,
one finds that $\txi_I^{(N-l+1)}$ is a singular vector
of weight $(\lambda_{N-l+1,1}^I(z),
\lambda_{N-l+1,2}^I(z), \cdots, \lambda_{N-l+1,N-l+1}^I(z)
)$ with respect to the subalgebra $E_{q,p}(\widehat{\gl}_{N-l+1})$.
Note $\lambda_{Nk}^I(z)=\lambda_k(z)$, $k=1,\dots,N$, where 
$(\lambda_1(z),\lambda_2(z),\dots,\lambda_N(z))$
is the weight for the highest vector $\txi_I^{(N)}=\zeta$. 
Since $\txi_I^{(N-l+1)}$ is a singular vector, we have
\begin{align}
L^+_{ij}(z) \cdot \txi_I^{(N-l+1)}&=0, \ \ \qquad
l \le j < i \le N, \label{singularfactone} \\
L^+_{l+k-1,l+k-1}(z) \cdot 
\txi_I^{(N-l+1)}
&=
\lambda_{N-l+1,k}^I(z) 
\txi_I^{(N-l+1)}, \ \ \qquad
k=1,\dots,N-l+1, \label{singularfacttwo}
\end{align}
From \eqref{expansion}, we have
\begin{align}
&A_l(z)
=\sum_{\sigma \in \mathfrak{S}_{N-l+1}}
\mathrm{sgn}_{{[}l,N{]}}^* (\sigma,\Pi^*)
L^+_{l \sigma(l)}(z)
L^+_{l+1 \sigma(l+1)}(q^{-2}z)
\cdots
L^+_{N, \sigma(N)}(q^{-2N+2l} z). \label{expansionforA}
\end{align}
Using \eqref{singularfactone}, \eqref{singularfacttwo}
and \eqref{expansionforA}, we find that $\txi_I^{(N-l+1)}$
is an eigenvector of $A_l(z)$:
\begin{align}
A_l(z) \cdot \txi_I^{(N-l+1)}
=\prod_{k=1}^{N-l+1} \lambda_{N-l+1,k}^I(q^{-2k+2} z) \txi_I^{(N-l+1)}.
\label{actionpartialxi}
\end{align}

Next, we consider the action of $A_l(z)$ on $\txi_I$.
Recall $A_l(z)$ is the center of the subalgebra $E_{q,p}(\widehat{\gl}_{N-l+1})$ generated by $L^+_{ij}(w), \ l \le i,j \le N$
\begin{align}
[A_l(z), L^+_{ij}(w)]=0, \ \ \ l \le i,j \le N.
\label{centersubalgebra}
\end{align}
Using \eqref{singulardecomposition}, \eqref{actionpartialxi}
and \eqref{centersubalgebra},
we find that $\txi_I$ is an eigenvector of $A_l(u)$:
\begin{align}
A_l(z) \cdot \txi_I&=
A_l(z) L_{N-1,N}^+(w_{I_N}) \cdots
L_{l,l+1}^+(w_{I_{l+1} \cup \cdots \cup I_N}) \cdot
\txi_I^{(N-l+1)} \nonumber \\
&=L_{N-1,N}^+(w_{I_N}) \cdots
L_{l,l+1}^+(w_{I_{l+1} \cup \cdots \cup I_N})
A_l(z) \cdot \txi_I^{(N-l+1)} \nonumber \\
&=\prod_{k=1}^{N-l+1} \lambda_{N-l+1,k}^I(q^{-2k+2} z)
L_{N-1,N}^+(w_{I_N}) \cdots
L_{l,l+1}^+(w_{I_{l+1} \cup \cdots \cup I_N}) \cdot
\txi_I^{(N-l+1)} \nonumber \\
&=\prod_{k=1}^{N-l+1} \lambda_{N-l+1,k}^I(q^{-2k+2} z)
\txi_I.
\end{align}

\end{proof}

\begin{thm}
The set of vectors $\{\txi_I\}_{I\in \cI_n}$ forms a basis of $\cV$. 
\end{thm}
\begin{proof}
From Proposition \ref{eigenAl}, each vector $\txi_I$ belongs to the different simultaneous eigenvalues of $A_l(z)$ $(1\leq l\leq N)$ so that   $\txi_I\ {(I\in \cI_n)}$ are linearly independent. The number of these vectors is 
\be
&&\sum_{(|I_1|,\cdots,|I_N|)\in \N^N\atop \sum_{a=1}^N|I_a|=n
}\frac{n!}{|I_1|!\cdots |I_N|!}=N^n. 
\en
\end{proof}

In \cite{KonnoJinttwo}, the Gelfand-Tsetlin bases of $\cV$ was constructed 
in a different way. Let us denote them by $\xi_I'$ $(I\in \cI_n)$. 
The construction is as follows. 
Let us define $\widetilde{S}_i(\Pi^*)$ by 
\be
&&\widetilde{S}_i(\Pi^*):=\cP^{(i i+1)}\bR^{(i i+1)}(z_{i}/z_{i+1},\Pi^* q^{2{\sum_{j=1}^{i-1}h^{(j)}}})s^z_i, 
\en
where $\cP$ and $s^z_i$ are the following permutation operators  
\be
&&\cP: v\tot w \mapsto w\tot v, \qquad s^z_i f(\cdots,z_i ,z_{i+1},\cdots )=f(\cdots,z_{i+1}, z_i ,\cdots )
\en
for any function $f(z_1,\cdots,z_n)$. 
Then by using the dynamical Yang-Baxter equation \eqref{DYBE} and the unitarity relation for $\bR(z,\Pi^*)$ one can show the following.
\begin{prop}
\be
&&\hspace{-2cm}\tS_i(\Pi^*)\tS_{i+1}(\Pi^*)\tS_i(\Pi^*)=\tS_{i+1}(\Pi^*)\tS_{i}(\Pi^*)\tS_{i+1}(\Pi^*),\\
&&\hspace{-2cm}\tS_i(\Pi^*)\tS_j(\Pi^*)= \tS_j(\Pi^*)\tS_i(\Pi^*) \qquad\qquad (|i-j|>1)\\
&&\hspace{-2cm}\tS_{i}(\Pi^*)^2=1.
\en
\end{prop}
For $I,J\in \cI_{n}$ with $|I_l|=|J_l|$ $(l=1,\cdots,N)$, let $I^{(l)}=\{i^{(l)}_1<\cdots<i^{(l)}_{\la^{(l)}}\}$ and $J^{(l)}=\{j^{(l)}_1<\cdots<j^{(l)}_{\la^{(l)}}\}$ 
$(l=1,\cdots,N)$.
We define a partial ordering $\leqslant$ by 
\be
I\leqslant J \Leftrightarrow i^{(l)}_a \leq j^{(l)}_a\qquad \forall l, a. 
\en
One then can construct the Gelfand-Tsetlin bases $\xi'_I$ $(I\in \cI_{n})$ by
\bea
&&\hspace{-2cm} \xi'_{I^{max}}:=v_{I^{max}},\qquad 
\xi'_{s_i(I)}:=\tS_i(\Pi^*)\xi_I,\lb{def:GT}
\ena  
where  $s_i(I)=I_{\cdots,\mu_{i+1},\mu_i,\cdots}$ for $I=I_{\cdots,\mu_i,\mu_{i+1},\cdots}$, and 
\be
&& I^{max}=I_{\tiny\underbrace{N\cdots N}_{|I_N|}\ \cdots \underbrace{1\cdots 1}_{|I_1|}}. 
\en  
In \cite{KonnoJinttwo}, the level-0 action of the Drinfeld generators  on $\xi_I^\prime$ was explicitly obtained. 
In particular the $K^+_l(z)$ $(1\leq l\leq N)$ are simultaneously diagonalized as
\begin{align}
\displaystyle
K^+_j(z) \cdot \xi_I^\prime
&=\prod_{m=1}^n\theta(q^2z/w_m)\prod_{k=1}^{j-1} \prod_{a \in I_k}
\frac{\theta(z/w_a)}{\theta(q^2z/w_a)} \prod_{l=j+1}^N \prod_{b \in I_l}
\frac{\theta(q^{-2}z/w_b)}{\theta(z/w_b)} \xi_I^\prime.\label{actionK}
\end{align}
Here we modified the formula obtained in \cite{KonnoJinttwo} by 
considering the difference of the vector representations of $L^+(z)$ used there and in this section:
in \cite{KonnoJinttwo} it is defined by the $R$-matrix ${\bR}(z/w,\Pi^*)
$ \eqref{ellipticrmatrix} instead of $\tR(z/w,\Pi^*)=a(z/w){\bR}(z/w,\Pi^*)$ \eqref{def:tR}. 
One can rewrite this 
as
\begin{align}
\displaystyle
K^+_j(z) \cdot \xi_I^\prime
&=\frac{\prod_{m=1}^n \theta(q^{2 \delta_{m \in I_j \cup \cdots \cup I_N}} z/w_m)
\prod_{m=1}^n \theta(q^{-2}z/w_m)
}{
\prod_{m=1}^n \theta(q^{2\delta_{m \in I_{j+1} \cup \cdots \cup I_N}-2} z/w_m)}
\xi_I^\prime.
\end{align}
Hence we have
\begin{align}
A_l(z) \cdot \xi_I^\prime
&=K^+_{l}(z) K^+_{l+1}(q^{-2}z) \cdots
K^+_{N}(q^{-2N+2l} z) \cdot \xi_I^\prime \nonumber \\
&=\prod_{m=1}^n \theta(q^{2\delta_{m \in I_{j} \cup \cdots \cup I_N}} z/w_m)
\prod_{m=1}^n \prod_{l=1}^{N-j}
\theta(q^{-2l}z/w_m) \xi_I^\prime,
\end{align}
which is in agreement with
\eqref{actionaoperators}. 
Therefore the difference between $\txi_I$ and $\xi_I^\prime$ is a multiplication by a scalar function. 
We determine this scalar function in Sec.\ref{sec:CBM}

\subsection{Relation between $\xi_\bnu$ and $\txi_I$
}

Let us compare the Gelfand-Tsetlin bases $\xi_{\bnu}$'s \eqref{GZNTtypebasis} in the previous section
with $\txi_I$'s \eqref{GZvectorrepresentation}. 
First, let us specialize 
the construction in the previous subsection
to the tensor product of the vector representations. 
This means we restrict the tuples of Gelfand-Tsetlin patterns
$\bla=(\bla^{(1)}, \dots, \bla^{(n)})$ to those  labelled by $I$.
By comparing the eigenvalues of  $A_l(z)$ (also recall $\overline{A}_{N+1-l}(z)=A_l(z)$), we find the following correspondence: $a^{(s)}=w_s$, and 
\be
&&\lambda_{j,1}^{(s)}=
\delta_{s \in I_{N-j+1} \cup \cdots \cup I_N },\\
&&\lambda_{j,k}^{(s)}=0,
\qquad k \neq j \qquad
\en
for $j=1,\dots,N$. 
In terms of $\bnu=(\bnu^{(1)}, \dots, \bnu^{(n)})$,
the correspondence is
\be
&&\nu_{j,j}^{(s)}=
\delta_{s \in I_{N-j+1} \cup \cdots \cup I_N },\\
&&\nu_{j,k}^{(s)}=0, 
\qquad k =1,\dots,j-1
\en
for $j=1,\dots,N$.
Then   the $\bB_\ell$-operators which contribute to the product in \eqref{GZNTtypebasis} are only those specified by $(\ell, j, s)$ satisfying $\ell=j$, $\nu^{(s)}_{j,j}=\delta_{s\in I_{N-j+1} \cup \cdots \cup I_N }=1$, i.e. 
$\overline{B}_j(w_s)$ $(1\leq j \leq N-1)$ with $s \in I_{N-j+1} \cup \cdots \cup I_N$. 
Let us denote  this $\xi_{\bnu}$ by ${\xi}_I$. Noting 
$\overline{B}_j(w)=B_{N+1-j}(w)=\ell^+(w)_{\{N-j,N-j+2,\dots,N \}}^{[N+1-j,N]}$, 
we  have 
\begin{align}
{\xi}_I=\ell^+(w_{I_N})_{\{ N-1 \}}^{\{N \}}
\ell^+(w_{I_{N-1} \cup I_N})_{\{ N-2,N \}}^{[N-1,N]}
\cdots
\ell^+(w_{I_{2} \cup \cdots \cup I_N})_{\{ 1,3,\dots,N \}}^{[2,N]} \cdot \zeta. \label{specializationNTtype}
\end{align}

Let us derive the scalar coefficient which relates 
${\xi}_I$ \eqref{specializationNTtype} and $\txi_I$ \eqref{GZvectorrepresentation}.
We first show the following.

\begin{proposition}
Let $I$ and $J$ be two partitions of $[1,n]$ such that
$I_{l+2}=\cdots=I_N=\phi$, $J_{l+2}=\cdots=J_N=\phi$,
$J_l=I_l \cup \{ k \}$, $J_{l+1}=I_{l+1} \backslash \{
k \}$, $J_j=I_j \ ( j \neq l,l+1)$. We have
\begin{align}
\displaystyle
\ell^+(w_k)_{
\{l,l+2,\dots,N \}}^{[l+1,N]} \cdot \txi_J
=\prod_{j=1}^{n} \prod_{a=1}^{N-l-1}
\theta(q^{-2a} w_k/w_j) \txi_I.
\label{actionvectordrinfeldanother}
\end{align}
\end{proposition}

\begin{proof}
Using \eqref{expansionsmallerqminors},
we have
\begin{align}
\displaystyle
\ell^+(w_k)_{
\{l,l+2,\dots,N \}}^{[l+1,N]} \cdot \txi_J
=
\sum_{j=1}^{N-l} \prod_{1 \le a < j}
\frac{\theta(q^2 \Pi_{i_j,i_a})}{\theta(\Pi_{i_a,i_j})}
\ell^+(q^{-2}w_k)_{ \{l,l+2,\dots,N \} \backslash \{ i_j \}  }^{[l+2,N]}
L^+_{i_j,l+1}(w_k) \cdot \txi_J, \label{formalaction}
\end{align}
where $i_1=l$, $i_a=l+a$ ($a=2,\dots,N-l$).
From $\displaystyle L_{ij}^+(z) \cdot \txi_J=0$, $l+1 \le j < i \le N$,
we note that only the summand corresponding to $j=1$ in
\eqref{formalaction} survives.
Noting also $L^+_{l,l+1}(w_k) \cdot \txi_J=\txi_I$, we get
\begin{align}
\displaystyle
\ell^+(w_k)_{
\{l,l+2,\dots,N \}}^{[l+1,N]} \cdot \txi_J
=
\ell^+(q^{-2} w_k)_{[l+2,N]}^{[l+2,N]}
L^+_{l,l+1}(w_k) \cdot \txi_J
=
\ell^+(q^{-2 } w_k)_{[l+2,N]}^{[l+2,N]} \cdot \txi_I.
\label{formalactiontwo}
\end{align}
Noting that our vector $\txi_I$ here is $\txi_I^{(N-l)}$, 
one can evaluate the right hand side of \eqref{formalactiontwo}
by using \eqref{expansion} and the following properties.
\begin{align}
L^+_{ij}(z) \cdot \txi_I&=0, \ \ \
l+1 \le j < i \le N, \\
L^+_{l+k,l+k}(z) \cdot
\txi_I
&=
\lambda_{N-l,k}^I(z)
\txi_I, \ \ \
k=1,\dots,N-l.
\end{align}
These are obtained by replacing $l$ in \eqref{singularfactone}
and \eqref{singularfacttwo} by $l+1$. 
Then the  result is 
\begin{align}
&\ell^+(q^{-2} w_k)_{[l+2,N]}^{[l+2,N]} \cdot \txi_I
\nonumber \\
=&\sum_{\sigma \in \mathfrak{S}_{N-l-1}}
\mathrm{sgn}_{[\ell+2,N]}^* (\sigma,\Pi^*)
L^+_{l+2, \sigma(l+2)}(q^{-2} w_k)
L^+_{l+3, \sigma(l+3)}(q^{-4}w_k)
\cdots
L^+_{N, \sigma(N)}(q^{-2N+2+2l} w_k) \cdot \txi_I
\nonumber \\
=&\lambda_{N-l,2}^I(q^{-2} w_k)
\lambda_{N-l,3}^I(q^{-4} w_k) \cdots
\lambda_{N-l,N-l}^I(q^{-2N+2+2l} w_k) \txi_I \nonumber \\
=&\prod_{j=1}^{n} \prod_{a=1}^{N-l-1} \theta(q^{-2a} w_k/w_j) \txi_I.
\label{evaluationxiI}
\end{align}
From \eqref{formalactiontwo} and \eqref{evaluationxiI},
we get \eqref{actionvectordrinfeldanother}.

\end{proof}
From \eqref{actionvectordrinfeldanother} and $L^+_{l,l+1}(w_k) \cdot \txi_J=\txi_I$, we have the following 
relation.
\begin{proposition}
\begin{align}
\displaystyle {\xi}_I=
\prod_{j=1}^n \prod_{l=1}^{N-1} \prod_{a=1}^{N-l-1}
\theta(q^{-2a} 
w_{I_{l+1} \cup \cdots \cup I_N}/w_j) \txi_I.
\label{differenceoverallfactorMNT}
\end{align}
\end{proposition}

\subsection{The change of basis matrix
}\lb{sec:CBM}

Now let us make a connection between $\txi_I$ and $\xi'_I$. 
As shown in Sec.\ref{GTBtensor}, their difference is a multiplication by a scalar function
\bea
&&\txi_I=N(\bfw)\xi_I'.\lb{txi2xip} 
\ena
We determine a function $N(\bfw)$ of the evaluation parameters $\bfw=(w_1,\cdots,w_n)$.   
As a byproduct we obtain an explicit formula for the partition function of the 2-dimensional square lattice model defined by the elliptic $R$-matrix $\tR(z,\Pi^*)$ \eqref{def:tR} in terms of the elliptic weight function.  
  
For this purpose let us  consider the change of basis matrix $(\tX_{IJ})_{I,J\in \cI_n}$ from the standard basis $\{v_I\}_{I\in \cI_n}$  to the Gelfand-Tsetlin basis $\{\txi_I\}_{I\in \cI_n}$.
\be
&&\txi_I=\sum_{J\in \cI_n}\tX_{IJ}v_J.
\en
From \eqref{LsonSB} and \eqref{GZvectorrepresentation} with $\zeta=v_{(1^n)}$, the matrix element $\tX_{IJ}$ for $J=J_{\bmu}$ with $\bmu=(\mu_1,\dots,\mu_n)$ is given by 
\begin{align}
\tX_{IJ}
=Z_{\bfK\; \bmu}^{\bfL (1^n)}(w_{I_2 \cup \cdots \cup I_N},
w_{I_3 \cup \cdots \cup I_N},\dots,w_{I_N}
; w_1,\dots,w_n;\Pi^*),\lb{XIJ}
\end{align}
where
$
\bfK=(
1^{|I_2 \cup \cdots \cup I_N|},
2^{|I_3 \cup \cdots \cup I_N|}, \dots,
(N-1)^{|I_N|}
)
$
and
$
\bfL=(
2^{|I_2 \cup \cdots \cup I_N|},
3^{|I_3 \cup \cdots \cup I_N|}, \dots,
N^{|I_N|}
)
$. 

As a statistical model, $\tX_{IJ}$ is a partition function of the 2-dimensional $n\times n$ square lattice model defined by 
assigning the $R$-matrix elements $\tR(z_i/w_j,\Pi^*q^{2h_{\bullet}})_{i_1i_2}^{j_1j_2}$  on the vertex at $(i,j)$  
(Figure \ref{figDR})\footnote{We count the row from the bottom to the top and the column from the left to the right.}  as statistical weights. Here $(z_1,\cdots,z_n)$ is taken as $(w_{I_2 \cup \cdots \cup I_N},
w_{I_3 \cup \cdots \cup I_N},\dots,w_{I_N})$. 
 On each link, the overlapped indices of the neighbor $R$-matrices are summed over $1,\cdots,N$
(Figure \ref{figurepartitionfunction}).

On the other hand, in \cite{KonnoJinttwo} the change of basis matrix from $\{v_I\}$ to $\{\xi'_I\}$
was obtained explicitly in terms of the elliptic weight functions. 
Let us recall them.  Let us set 
$I^{(l)}:=I_1 \cup \cdots \cup I_l$ and
denote its elements as 
$I^{(l)}=\{ i_1^{(l)}<\cdots<i_{\lambda^{(l)}}^{(l)} \}$, where  
we set $\lambda_l=|I_l|$ and  $\lambda^{(l)}:=\lambda_1+\cdots+\lambda_l$. 
To each $i_a^{(l)}$ ($l=1,\dots,N$, $a=1,\dots,\lambda^{(l)}$),
we associate a variable $t_a^{(l)} \equiv t(i_a^{(l)})$
with $t_a^{(N)}=w_a$ $(a=1,\dots,n)$ and set $\bt=(t_a^{(l)})$
$(l=1,\dots,N, a=1,\dots,\lambda^{(l)})$
and $\bfw=(w_a)$ $(a=1,\dots,n)$. 
Then the elliptic weight functions \cite{KonnoJintone,KonnoJinttwo} are given by
\begin{align}
\widetilde{W}_I(\bt,\bfw,\Pi^*)
&=\mathrm{Sym}_{t^{(1)}} \cdots \mathrm{Sym}_{t^{(N-1)}}
\widetilde{U}_I(t,w,\Pi^*), \\
\widetilde{U}_I(t,w,\Pi^*)
&=\prod_{l=1}^{N-1} \prod_{a=1}^{{\lambda}^{(l)}}
\Bigg(
\frac{\theta(q^{-2 C_{\mu_s,l+1}(s)}
\Pi^*_{\mu_s,l+1}
t_b^{(l+1)}/t_a^{(l)}) \theta(q^2) }
{\theta(q^2 t_b^{(l+1)}/t_a^{(l)})
\theta(q^{-2 C_{\mu_s,l+1}(s)}
\Pi^*_{\mu_s,l+1})} \Bigg|_{i_b^{(l+1)}=i_a^{(l)}=s}
\nonumber \\
&\times
\prod_{\substack{
b=1
\\
i_b^{(l+1)} > i_a^{(l)}
}}^{\lambda^{(l+1)}}
\frac{\theta(t_b^{(l+1)}/t_a^{(l)})}{\theta(q^2 t_b^{(l+1)}/t_a^{(l)})}
\prod_{b=a+1}^{\lambda^{(l)}} 
\frac{\theta(q^{-2 }t_a^{(l)}/t_b^{(l)})}{\theta(t_a^{(l)}/t_b^{(l)})}
\Bigg),
\end{align}
where we set
$C_{\mu_s,l+1}(s):=\sum_{j=s+1}^n \langle \overline{\epsilon}_{\mu_j},
h_{\mu_s,l+1} \rangle$ $(\mu_s \le l)$,
and $\mathrm{Sym}_{t^{(l)}}$ denotes symmetrization over the variables
$t_1^{(l)},\dots,t_{\lambda^{(l)}}^{(l)}$.

Define the specialization $\bt=\bfw_I$ by  
$t_a^{(l)}=w_{i_a^{(l)}} (l=1,\dots,N-1,a=1,\dots,\lambda^{(l)})$. 
One obtains the lower triangular matrix $X=(\tW_J(\bfw_I,\bfw,\Pi^*))_{I,J\in \cI_{\la}}$\cite{KonnoJintone}. 
Here we put the matrix elements  in the decreasing order $I^{max}\geqslant\cdots \geqslant I^{min}$.
For $\la=(\la_1,\cdots,\la_N)\in \N^N$, $\mathcal{I}_\lambda$ denotes the set of all 
partition $I$ of $[1,n]$ satisfying $I_l=\la_l$ $l=1,\dots, N$. 
Then one obtains the following statement.
 \begin{theorem} (\cite{KonnoJinttwo} Thm 4.5)
\begin{align}
\displaystyle \xi_I^\prime
=\sum_{J \in \mathcal{I}_\lambda} \widetilde{W}_J(\bfw_I,\bfw,\Pi^* q^{2 \sum_{j=1}^n \langle
\overline{\epsilon}_{\mu_j},h \rangle }) v_J.  \label{GZbasisweightfunction}
\end{align}
\end{theorem}

From \eqref{txi2xip} and the invertibility of the change of basis matrices, one obtains the following formula 
for the partition function. 
\begin{prop}
\be
&&Z_{\bfK\; \bmu}^{\bfL (1^n)}(w_{I_2 \cup \cdots \cup I_N},
w_{I_3 \cup \cdots \cup I_N},\dots,w_{I_N}
;\bfw
;\Pi^*)=N(\bfw)\widetilde{W}_J(\bfw_I,\bfw,\Pi^* q^{2 \sum_{j=1}^n \langle
\overline{\epsilon}_{\mu_j},h \rangle })
\en
with $J=J_{\bmu}$, $
\bfK=(
1^{|I_2 \cup \cdots \cup I_N|},
2^{|I_3 \cup \cdots \cup I_N|}, \dots,
(N-1)^{|I_N|}
)
$
and
$
\bfL=(
2^{|I_2 \cup \cdots \cup I_N|},
3^{|I_3 \cup \cdots \cup I_N|}, \dots,
N^{|I_N|}
)
$.
\end{prop}

\noindent
{\it Example.}\ The case of $N=2$, $n=3$, $\lambda=(2,1)$: 
one has
\begin{align}
\displaystyle \widetilde{W}_{I_{211}}(w_{I_{112}},w,\Pi^*)&
=\frac{\theta(q^{-2}\Pi^*_{1,2}w_3/w_1)\theta(q^2)}{\theta(q^2 w_3/w_1)\theta(q^{-2}\Pi^*_{1,2})}, \\
\displaystyle \widetilde{W}_{I_{121}}(w_{I_{112}},w,\Pi^*)&
=\frac{\theta(w_3/w_1)\theta(\Pi^*_{1,2}w_{3}/w_2)\theta(q^2)}{
\theta(q^2 w_3/w_1)\theta(q^2 w_3/w_2)\theta(\Pi^*_{1,2})}, \\
\displaystyle \widetilde{W}_{I_{111}}(w_{I_{112}},w,\Pi^*)&
=\frac{\theta(w_3/w_1)\theta(w_3/w_2)}{\theta(q^2 w_{3}/w_1)\theta(q^2 w_3/w_2)},
\end{align}
and
\begin{align}
\displaystyle
\xi_{112}^\prime=&
\displaystyle \widetilde{W}_{I_{211}}(w_{I_{112}},w,\Pi^* q^{2 <2 \overline{\epsilon}_{1}+\overline{\epsilon}_2 ,h>})v_{211}
+\displaystyle \widetilde{W}_{I_{121}}(w_{I_{112}},w,\Pi^* q^{2 <2 \overline{\epsilon}_{1}+\overline{\epsilon}_2 ,h>})v_{121} \nn \\
&+\displaystyle \widetilde{W}_{I_{111}}(w_{I_{112}},w,\Pi^* q^{2 <2 \overline{\epsilon}_{1}+\overline{\epsilon}_2 ,h>})v_{111} \nn \\
=&
\frac{\theta(\Pi^*_{1,2}w_3/w_1)\theta(q^2)}{\theta(q^2 w_3/w_1)\theta(\Pi^*_{1,2})} v_{211}
+\frac{\theta(w_3/w_1)
\theta(q^2 \Pi^*_{1,2}w_{3}/w_2)\theta(q^2)}{\theta(q^2 w_3/w_1)
\theta(q^2 w_3/w_2)  \theta(q^2 \Pi^*_{1,2})} v_{121} \nn \\
&+\frac{\theta(w_3/w_1)\theta(w_3/w_2)}{\theta(q^2 w_{3}/w_1)
\theta(q^2 w_3/w_2)} v_{112}.
\end{align}
On the other hand,
by direct computation, we find
\begin{align}
\txi_{112}=&\frac{\theta(\Pi^*_{1,2}w_3/w_1)\theta(q^2)^2 \theta(q^2 w_3/w_2)}
{\theta(\Pi^*_{1,2})} v_{211}
+\frac{\theta(w_3/w_1)\theta(q^2 \Pi^*_{1,2} w_{3}/w_2)\theta(q^2)^2}
{\theta(q^2 \Pi^*_{1,2})} v_{121} \nn \\
&+\theta(w_3/w_1)\theta(w_3/w_2)\theta(q^2) v_{112}.
\end{align}
The relation between
$\txi_{112}$ and $\xi_{112}^\prime$ is $\txi_{112}=\theta(q^2)
\theta(q^2 w_3/w_2)\theta(q^2 w_3/w_1) \xi_{112}^\prime$.

In the rest of this section, we determine $N(\bfw)$ as the ratio of the diagonal elements 
of $\tX$ and $X$. For $X$, we have the result.  
\begin{proposition} (\cite{KonnoJintone}, Prop 5.1)
\begin{align}
\widetilde{W}_I(\bfw_I,\bfw,\Pi^*) 
=\prod_{1 \le k < l < N} \prod_{a \in I_k} \prod_{\substack{b \in I_l \\ a<b}}
\frac{\theta(w_b/w_a)}{\theta(q^2 w_b/w_a)}.
\label{Wtildediagonalelements}
\end{align}
\end{proposition}

We show that the diagonal elements $\tX_{II}$ 
can be calculated recursively and have a factored expression.
To illustrate the calculation, let us show an example. 

\noindent
{\it Example.}\ Take $N=3$, $n=9$, $I=I_{231213231}$. From the example in Sec.\ref{GTBtensor}, we have 
\begin{align}
\tX_{II}=&Z_{(1^6,2^3),(231213231)}^{(2^6,3^3),(1^9)}(w_1,w_2,w_4,w_6,w_7,w_8,w_2,w_6,w_8;w_1,\dots,w_9;\Pi^*) \nonumber \\
=&
\sum_{\boldsymbol{\alpha}_{1},\dots,\boldsymbol{\alpha}_{8}\in [1,3]^9
} T_{23}^+(w_8;w_1,\dots,w_9;\Pi^*)_{(231213231)}^{\boldsymbol{\alpha}_{8}}
\nonumber \\
&\times T_{23}^+(w_6;w_1,\dots,w_9;\Pi^*q^{2h_{\overline{\epsilon}_{3}}})_{\boldsymbol{\alpha}_{8}}^{\boldsymbol{\alpha}_{7}}
T_{23}^+(w_2;w_1,\dots,w_9;\Pi^*q^{2h_{2\overline{\epsilon}_{3}}})_{\boldsymbol{\alpha}_{7}}^{\boldsymbol{\alpha}_{6}}
\nonumber \\
&\times
T_{12}^+(w_8;w_1,\dots,w_9;\Pi^*q^{2h_{3\overline{\epsilon}_{3}}})_{\boldsymbol{\alpha}_{6}}^{\boldsymbol{\alpha}_{5}}
T_{12}^+(w_7;w_1,\dots,w_9;\Pi^*q^{2h_{\overline{\epsilon}_{2}+
3\overline{\epsilon}_{3}}})_{\boldsymbol{\alpha}_{5}}^{\boldsymbol{\alpha}_{4}}
\nonumber \\
&\times
T_{12}^+(w_6;w_1,\dots,w_9;\Pi^*q^{2h_{2\overline{\epsilon}_{2}+
3\overline{\epsilon}_{3}}})_{\boldsymbol{\alpha}_{4}}^{\boldsymbol{\alpha}_{3}}
T_{12}^+(w_4;w_1,\dots,w_9;\Pi^*q^{2h_{3\overline{\epsilon}_{2}+
3\overline{\epsilon}_{3}}})_{\boldsymbol{\alpha}_{3}}^{\boldsymbol{\alpha}_{2}}
\nonumber \\
&
\times
T_{12}^+(w_2;w_1,\dots,w_9;\Pi^*q^{2h_{4\overline{\epsilon}_{2}+
3\overline{\epsilon}_{3}}})_{\boldsymbol{\alpha}_{2}}^{\boldsymbol{\alpha}_{1}}
T_{12}^+(w_1;w_1,\dots,w_9;\Pi^*q^{2h_{5\overline{\epsilon}_{2}+
3\overline{\epsilon}_{3}}})_{\boldsymbol{\alpha}_{1}}^{(1^9)}.
\label{beforedecomposition}
\end{align}
See Figure \ref{figuredigonalelementone} for a graphical description
of \eqref{beforedecomposition}.

\begin{figure}[] 
\centering
\includegraphics[width=16cm]{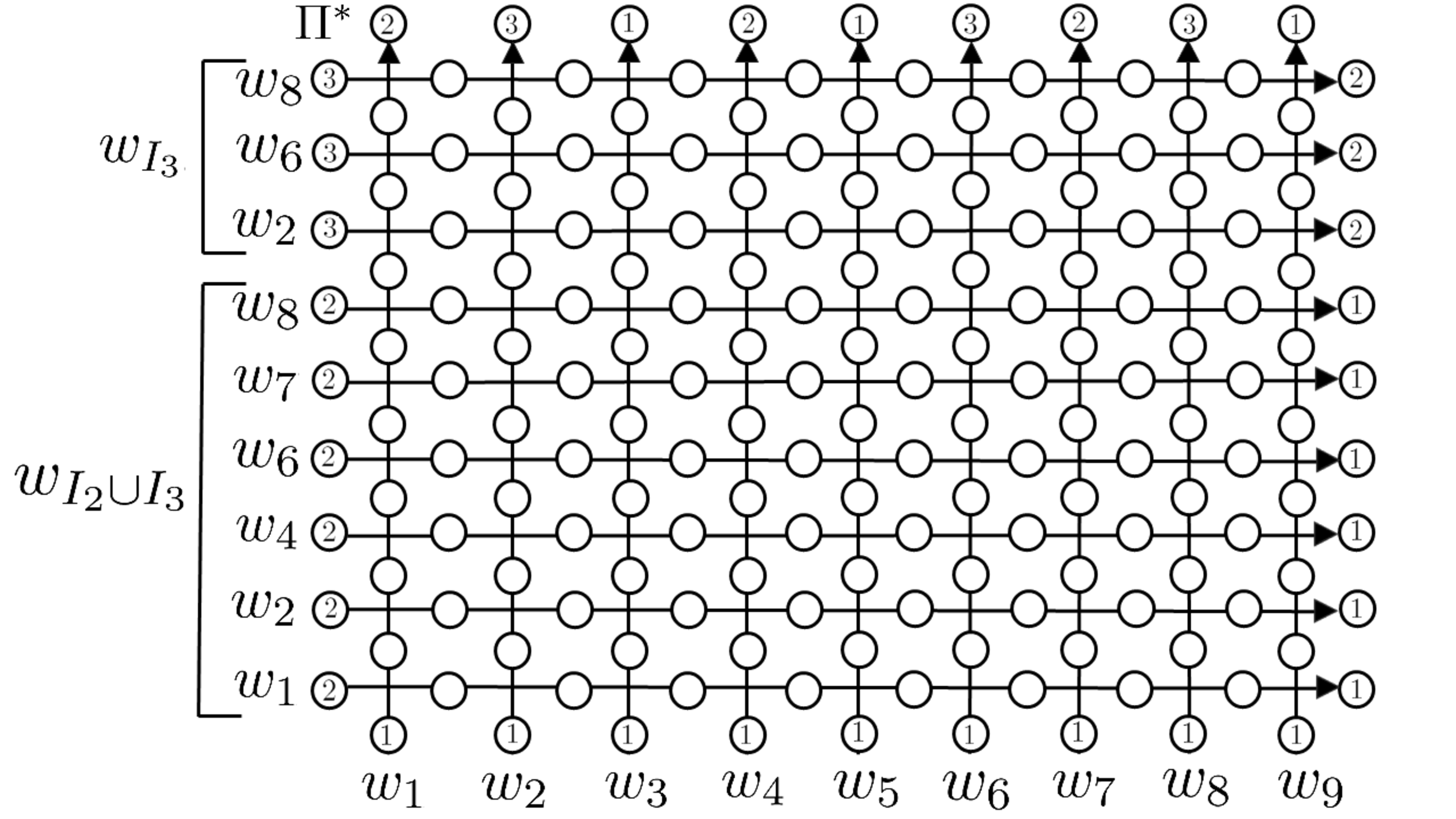}
\caption{ 
The partition function corresponding to
 $\tX_{II}$
for $N=3$, $I=231213231$.
}
\label{figuredigonalelementone}
\end{figure}

Taking $\boldsymbol{\alpha}_j=(\alpha_j,\boldsymbol{\alpha}_j^\prime)$ with $\al_j\in [1,3]$ and $\bal_j'\in [1,3]^8$, $j=1,\dots,8$, 
we decompose \eqref{beforedecomposition} as
\begin{align}
\displaystyle
&\sum_{\alpha_1,\dots,\alpha_8,\mu_1,\dots,\mu_9\in [1,3]}
C(\al_1,\dots,\al_8;\mu_1,\dots,\mu_9)Z(\alpha_1,\dots,\alpha_8;\mu_1,\dots,\mu_9), 
\label{afterdecomposition}
\end{align}
where we set 
\begin{align}
&C(\alpha_1,\dots,\alpha_8;\mu_1,\dots,\mu_9) \nonumber \\
&=
\tR(w_8/w_1,\Pi^*)_{\mu_9 2}^{3 \alpha_8}
\tR(w_6/w_1,\Pi^* q^{2h_{\overline{\epsilon}_3}})_{\mu_8 \alpha_8}^{3 \alpha_7}
\tR(w_2/w_1,\Pi^*  q^{2h_{2\overline{\epsilon}_3}})
_{\mu_7 \alpha_7}^{3 \alpha_6} \nonumber \\
&\times \tR(w_8/w_1,\Pi^* q^{2h_{3\overline{\epsilon}_3}})
_{\mu_6 \alpha_6}^{2 \alpha_5}
\tR(w_7/w_1,\Pi^*q^{2h_{ \overline{\epsilon}_2+3 \overline{\epsilon}_3}})
_{\mu_5 \alpha_5}^{2 \alpha_4} \nonumber \\
&\times \tR(w_6/w_1,\Pi^*q^{2h_{2 \overline{\epsilon}_2+3 \overline{\epsilon}_3}})_{\mu_4 \alpha_4}^{2 \alpha_3}
\tR(w_4/w_1,\Pi^*q^{2h_{3 \overline{\epsilon}_2+3 \overline{\epsilon}_3}})
_{\mu_3 \alpha_3}^{2 \alpha_2} \nonumber \\
&\times \tR(w_2/w_1,\Pi^*q^{2h_{4 \overline{\epsilon}_2+3 \overline{\epsilon}_3}})_{\mu_2 \alpha_2}^{2 \alpha_1}
\tR(w_1/w_1,\Pi^*q^{2h_{5 \overline{\epsilon}_2+3 \overline{\epsilon}_3}})_{\mu_1 \alpha_1}^{21},  \label{decompositionfactor}\\
&Z(\alpha_1,\dots,\alpha_8;\mu_1,\dots,\mu_9)\nn\\
&=\sum_{\boldsymbol{\alpha}_{1}^\prime,\dots,\boldsymbol{\alpha}_{8}^\prime\in [1,3]^8
}
T_{2 \mu_9}^+(w_8;w_2,\dots,w_9;\Pi^*q^{2h_{\overline{\epsilon}_2}})_{(31213231)}^{\boldsymbol{\alpha}_{8}^\prime}
\nonumber \\
&\times T_{2 \mu_8}^+(w_6;w_2,\dots,w_9;\Pi^*q^{2h_{\overline{\epsilon}_{3}+\overline{\epsilon}_{\alpha_8}}})_{\boldsymbol{\alpha}_{8}^\prime}^{\boldsymbol{\alpha}_{7}^\prime}
T_{2 \mu_7}^+(w_2;w_2,\dots,w_9;\Pi^*q^{2h_{2\overline{\epsilon}_{3}+\overline{\epsilon}_{\alpha_7}}})_{\boldsymbol{\alpha}_{7}^\prime}^{\boldsymbol{\alpha}_{6}^\prime}
\nonumber \\
&\times
T_{1 \mu_6}^+(w_8;w_2,\dots,w_9;\Pi^*q^{2h_{3\overline{\epsilon}_{3}+\overline{\epsilon}_{\alpha_6}}})_{\boldsymbol{\alpha}_{6}^\prime}^{\boldsymbol{\alpha}_{5}^\prime}
T_{1 \mu_5}^+(w_7;w_2,\dots,w_9;\Pi^*q^{2h_{\overline{\epsilon}_{2}+
3\overline{\epsilon}_{3}+\overline{\epsilon}_{\alpha_5}}})_{\boldsymbol{\alpha}_{5}^\prime}^{\boldsymbol{\alpha}_{4}^\prime}
\nonumber \\
&\times
T_{1 \mu_4}^+(w_6;w_2,\dots,w_9;\Pi^*q^{2h_{2\overline{\epsilon}_{2}+
3\overline{\epsilon}_{3}+\overline{\epsilon}_{\alpha_4}}})_{\boldsymbol{\alpha}_{4}^\prime}^{\boldsymbol{\alpha}_{3}^\prime}
T_{1 \mu_3}^+(w_4;w_2,\dots,w_9;\Pi^*q^{2h_{3\overline{\epsilon}_{2}+
3\overline{\epsilon}_{3}+\overline{\epsilon}_{\alpha_3}}})_{\boldsymbol{\alpha}_{3}^\prime}^{\boldsymbol{\alpha}_{2}^\prime}
\nonumber \\
&
\times
T_{1 \mu_2}^+(w_2;w_2,\dots,w_9;\Pi^*q^{2h_{4\overline{\epsilon}_{2}+
3\overline{\epsilon}_{3}+\overline{\epsilon}_{\alpha_2}}})_{\boldsymbol{\alpha}_{2}^\prime}^{\boldsymbol{\alpha}_{1}^\prime}
T_{1 \mu_1}^+(w_1;w_2,\dots,w_9;\Pi^*q^{2h_{5\overline{\epsilon}_{2}+
3\overline{\epsilon}_{3}+\overline{\epsilon}_{\alpha_1}}})_{\boldsymbol{\alpha}_{1}^\prime}^{(1^8)}.
\end{align}
The $C(\alpha_1,\dots,\alpha_8;\mu_1,\dots,\mu_9)$ part corresponds to the 1st column (the left most column ) in Figure \ref{figuredigonalelementone}.
We now show that $\alpha_1,\dots,\alpha_8$, $\mu_1,\dots,\mu_9$
in \eqref{afterdecomposition} are determined uniquely. 
In fact, applying the property
\begin{align}
{\tR}(z,P)_{k l}^{ij} = 0 \ \ \ \mathrm{unless} \ \ \ i=k, \ j=l
\ \ \ \mathrm{or} \ \ \ i=l, \ j=k
\label{icerule}
\end{align}
to each factor in  $C(\alpha_1,\dots,\alpha_8;\mu_1,\dots,\mu_9)$, one finds 
\begin{itemize}
\item $\tR(w_8/w_1,\Pi^*)_{\mu_9 2}^{3 \alpha_8}$ yields 
$\alpha_8=2$, $\mu_9=3$
\item $\tR(w_6/w_1,\Pi^* q^{2h_{\overline{\epsilon}_3}})_{\mu_8 \alpha_8}^{3 \alpha_7}$
with $\alpha_8=2$ yields  $\alpha_7=2$, $\mu_8=3$
\item ${\tR}(w_2/w_1,\Pi^*  q^{2h_{2\overline{\epsilon}_3}})
_{\mu_7 \alpha_7}^{3 \alpha_6}$ with $\alpha_7=2$ yields  $\alpha_6=2$, $\mu_7=3$.
\item  $\tR(w_1/w_1,\Pi^*q^{2h_{5 \overline{\epsilon}_2+3 \overline{\epsilon}_3}})_{\mu_1 \alpha_1}^{21}$ yields $\alpha_1=2$, $\mu_1=1$ due to 
$\tR(w_1/w_1,\Pi^*q^{2h_{5 \overline{\epsilon}_2+3 \overline{\epsilon}_3}})_{21}^{21}
=\theta(w_1/w_1)=0$
\item the factors 
$\tR(w_2/w_1,\Pi^*q^{2h_{4 \overline{\epsilon}_2+3 \overline{\epsilon}_3}})_{\mu_2 \alpha_2}^{2 \alpha_1}$,
$
\tR(w_4/w_1,\Pi^*q^{2h_{3 \overline{\epsilon}_2+3 \overline{\epsilon}_3}})
_{\mu_3 \alpha_3}^{2 \alpha_2}
$,
$\tR(w_6/w_1,\Pi^*q^{2h_{2 \overline{\epsilon}_2+3 \overline{\epsilon}_3}})_{\mu_4 \alpha_4}^{2 \alpha_3}$,\\ 
$\tR(w_7/w_1,\Pi^* q^{2h_{\overline{\epsilon}_2+3 \overline{\epsilon}_3}})
_{\mu_5 \alpha_5}^{2 \alpha_4}
$
and
$
\tR(w_8/w_1,\Pi^*q^{2h_{3 \overline{\epsilon}_3}})
_{\mu_6 \alpha_6}^{2 \alpha_5}
$ in this order with $\al_1=2$ yield $\alpha_2=\cdots=\alpha_6=\mu_2=\cdots=\mu_6=2$.
\end{itemize}
Therefore we obtain
\begin{align}
\displaystyle
\tX_{II}=&C(2^8;1, 2^5, 3^3)
Z_{(1^6,2^3),(31213231)}^{(1,2^5,3^3),(1^8)}
(w_1,w_2,w_4,w_6,w_7,w_8,w_2,w_6,w_8
;w_2,\dots,w_9; \Pi^*q^{2h_{\overline{\epsilon}_2}}).
\label{afterdecompositiontwo}
\end{align}
Here we made the identification 
\be
&&Z(2^8;1, 2^5, 3^3)=Z_{(1^6,2^3),(31213231)}^{(1,2^5,3^3),(1^8)}
(w_1,w_2,w_4,w_6,w_7,w_8,w_2,w_6,w_8
;w_2,\dots,w_9; \Pi^*q^{2h_{\overline{\epsilon}_2}}). 
\en
We can repeat this process 
and find that $\tX_{II}$ 
can be written as a product of 
the $R$-matrix elements 
uniquely determined by the configuration at the boundary of the lattice (Figure \ref{figurediagonalelementtwo}).
From the perspective of partition functions, this means that
there is only one configuration which gives nonzero contribution.

\begin{figure}[] 
\centering
\includegraphics[width=16cm]{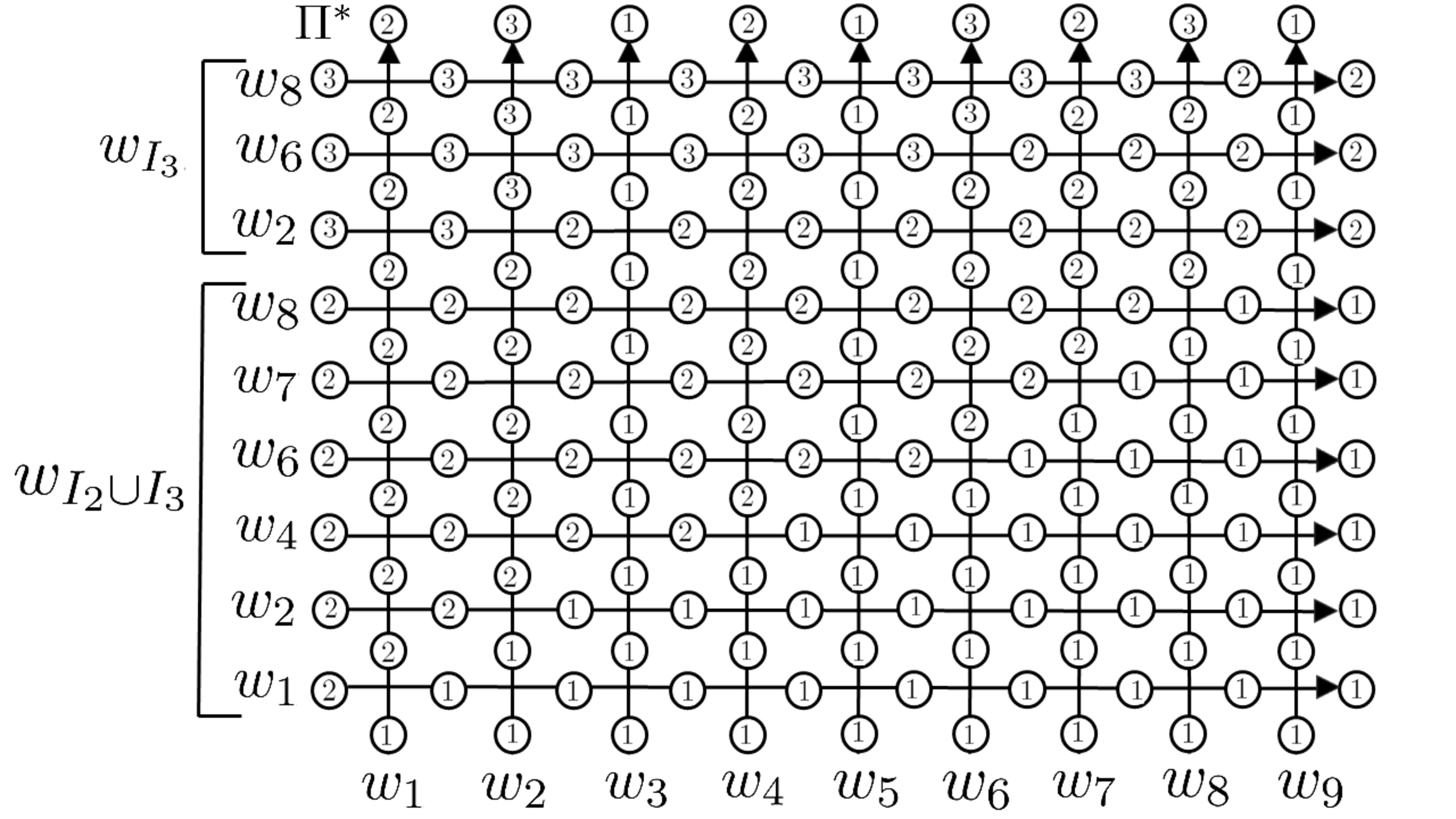}
\caption{ 
The unique configuration  corresponding to 
$\tX_{II}$ for $N=3$, $I=231213231$.
}
\label{figurediagonalelementtwo}
\end{figure}

\begin{thm}
\begin{align}
\displaystyle
\tX_{II}=&\ \theta(q^2)^{\sum_{k=2}^N (k-1) |I_k|
}
\prod_{l=1}^{N-1}
\Bigg\{
\prod_{\substack{
j \in I_{N-l+1} \cup \cdots \cup I_{N}
\\
k \in I_{1} \cup \cdots \cup I_{N-l-1}
}
} \theta(w_j/w_k)
\nonumber \\
&\times
\prod_{\substack{
j \in I_{N-l+1} \cup \cdots \cup I_{N}
\\
k \in I_{N-l} \cup \cdots \cup I_{N}
\\ j < k
}
}
\theta(q^2 w_j/w_k)
\prod_{\substack{
j \in I_{N-l+1} \cup \cdots \cup I_{N}
\\
k \in I_{N-l+1} \cup \cdots \cup I_{N} \\
j > k
}
}
\theta(q^2 w_j/w_k)
\prod_{\substack{
j \in I_{N-l+1} \cup \cdots \cup I_{N}
\\
k \in I_{N-l}
\\
j > k
}
}
\theta(w_j/w_k)
\Bigg\}. \label{coeffGZbasis}
\end{align}
\end{thm}
\begin{proof} The diagonal matrix element $\tX_{II}$ is given by \eqref{XIJ} with $I=J_{\bmu}$, $\bmu=(\mu_1,\cdots,\mu_n)$.  
We show that $\tX_{II}$ is factored into a product of the matrix elements of
the $R$-matrices  determined  uniquely by the boundary configuration.

To show this inductively in the number of columns in the lattice, 
let us introduce the following. For $k=1,\dots,n$ and $j=1,\cdots,N$, let us set  $n(I,k,j):=|\{ \mu_1,\dots,\mu_k \in I_j \}|$ 
and 
${I_j \cup \cdots \cup I_N}
=\{
i_1^{[j,N]} < 
i_2^{[j,N]} < \cdots < i_{|I_j \cup \cdots \cup I_N|}^{[j,N]
}
\}$. We define the $k$-th partition function  by
\begin{align}
&Z_k:=Z_{\bfK_k\; ( \mu_k,\dots,\mu_n )}^{\bfL (1^{n+1-k})}
(w_{i_1^{[2,N]}}, \dots, w_{i_{|I_2 \cup \cdots \cup I_N|}^{[2,N]}},
w_{i_1^{[3,N]}}, \dots, w_{i_{|I_3 \cup \cdots \cup I_N|}^{[3,N]}}
, \dots,
w_{i_1^{[N,N]}}, \dots, w_{i_{|I_N|}^{[N,N]}}; \nonumber \\
&\qquad\qquad\qquad\qquad \qquad\qquad
w_k,\dots,w_n;\Pi^*q^{\sum_{l=1}^{k-1}2h_{
\overline{\epsilon}_{\mu_l}}}
). \label{ZKdef}
\end{align}
Here, $\bfK_k=
(1^{n(I,k,2)+\cdots+n(I,k,N)},
2^{|I_2 \cup \dots \cup I_N|-n(I,k,2)},
3^{|I_3 \cup \dots \cup I_N|-n(I,k,3)}
,\dots,N^{|I_N|-n(I,k,N)}
)$, and  we fix the order of the variables 
$(w_{I_2 \cup \cdots \cup I_N}, w_{I_3 \cup \cdots \cup I_N},\dots,w_{I_N})$ as
\be
&&(w_{i_1^{[2,N]}}, \dots, w_{i_{|I_2 \cup \cdots \cup I_N|}^{[2,N]}},
w_{i_1^{[3,N]}}, \dots, w_{i_{|I_3 \cup \cdots \cup I_N|}^{[3,N]}}
, \dots, w_{i_1^{[N,N]}}, \dots, w_{i_{|I_N|}^{[N,N]}}).
\en
 We show that from the $k$-th partition function $Z_k$ one can obtain the $k+1$-th partition $Z_{k+1}$  by removing the $k$-th column as illustrated in the above example.   
The precise relation between $Z_k$ and $Z_{k+1}$ is given below. We also set $Z_{n+1}:=1$.
\begin{lemma}
For $k \in I_l$,
the relation between $Z_k$ and $Z_{k+1}$ is given by
\begin{align}
Z_k=W_k Z_{k+1},
\end{align}
where
\begin{align}
W_k&=W_k^1 W_k^2 W_k^3, \\
W_k^1&=\prod_{\substack{ j \in I_N } } \theta(w_j/w_k)
\prod_{\substack{ j \in I_{N-1} \cup I_N } } \theta(w_j/w_k)
\cdots
\prod_{\substack{ j \in I_{l+2} \cup \cdots \cup I_N} } \theta(w_j/w_k), \label{WKone} \\
W_k^2&=
\prod_{\substack{ j \in I_{l+1} \cup \cdots \cup I_N  \\ j>k} } \theta(w_j/w_k)
\prod_{\substack{ j \in I_{l+1} \cup \cdots \cup I_N  \\ j<k} } \theta(q^2 w_j/w_k)
, \label{WKtwo} \\
W_k^3&=\Bigg\{ \theta(q^2)
\prod_{\substack{ j \in I_{l} \cup \cdots \cup I_N  \\ j>k} } \theta(q^2 w_j/w_k)
\prod_{\substack{ j \in I_{l} \cup \cdots \cup I_N  \\ j<k} } \theta(q^2 w_j/w_k)
\Bigg\} \nn \\
&\times \Bigg\{ \theta(q^2)
\prod_{\substack{ j \in I_{l-1} \cup \cdots \cup I_N  \\ j>k} } \theta(q^2 w_j/w_k)
\prod_{\substack{ j \in I_{l-1} \cup \cdots \cup I_N  \\ j<k} } \theta(q^2 w_j/w_k)
\Bigg\} \nn \\
&\times \cdots \times \Bigg\{ \theta(q^2)
\prod_{\substack{ j \in I_{2} \cup \cdots \cup I_N  \\ j>k} } \theta(q^2 w_j/w_k)
\prod_{\substack{ j \in I_{2} \cup \cdots \cup I_N  \\ j<k} } \theta(q^2 w_j/w_k)
\Bigg\}. \label{WKthree}
\end{align}
\end{lemma}
\begin{proof}
We show by induction on $k$.
Figure \ref{figureonecolumnelementone} 
is a graphical description of
the $k$-th column
of the partition function corresponding to $\tilde{X}_{II}$
or equivalently the leftmost column of the partition function $Z_k$
when $k \in I_l$. 
By inductive assumption,
the left boundary condition of the column
which corresponds to the left boundary condition of $Z_k$
is given by $\bfK_k=
(1^{n(I,k,2)+\cdots+n(I,k,N)},
2^{|I_2 \cup \dots \cup I_N|-n(I,k,2)},
3^{|I_3 \cup \dots \cup I_N|-n(I,k,3)}
,\dots,N^{|I_N|-n(I,k,N)}
)$.
We find in the same way as in {\it Example} that the
empty circles in Figure \ref{figureonecolumnelementone} are uniquely filled with the numbers
as given in Figure \ref{figureonecolumnelementtwo}.
This
means that from this column we have products of
$R$-matrix elements denoted as $W_k$, which can be read out explicitly from Figure \ref{figureonecolumnelementtwo}
as $W_k=W_k^1 W_k^2 W_k^3$ where $W_k^1$, $W_k^2$, $W_k^3$ are \eqref{WKone}, \eqref{WKtwo} and \eqref{WKthree}.
Figure \ref{figureonecolumnelementtwo} also implies $Z_k=W_k Z_{k+1}$.
The right boundary of the column in Figure \ref{figureonecolumnelementtwo}
is $\bfK_{k+1}=
(1^{n(I,k+1,2)+\cdots+n(I,k+1,N)},
2^{|I_2 \cup \dots \cup I_N|-n(I,k+1,2)},
3^{|I_3 \cup \dots \cup I_N|-n(I,k+1,3)}
,\\ \dots,N^{|I_N|-n(I,k+1,N)}
)$ and this becomes the left boundary of $Z_{k+1}$.
\end{proof}

\begin{figure}[] 
\centering
\includegraphics[width=16cm]{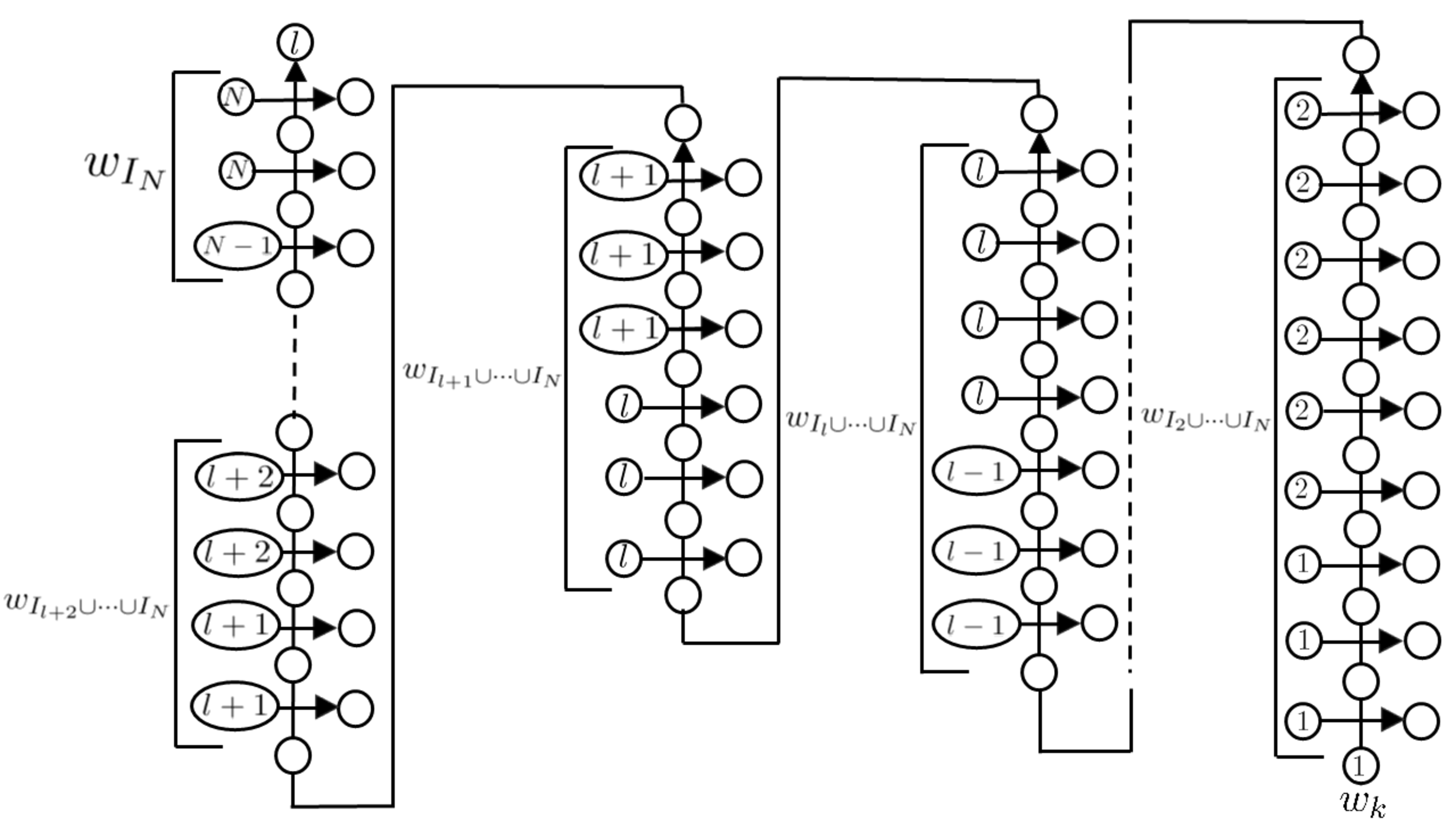}
\caption{ 
A graphical representation of the $k$-th column
of the partition function corresponding to $\tilde{X}_{II}$
or equivalently the leftmost column of the partition function $Z_k$,
when $k \in I_l$.
The long line which is bended corresponds to the space $\widehat{V}_{w_k}$.
Note there are also $R$-matrices in the dotted lines.
The left boundary condition of this column is (reading the numbers in circles from bottom to top)
$\bfK_k=
(1^{n(I,k,2)+\cdots+n(I,k,N)},
2^{|I_2 \cup \dots \cup I_N|-n(I,k,2)},
3^{|I_3 \cup \dots \cup I_N|-n(I,k,3)}
,\dots,N^{|I_N|-n(I,k,N)}
)$.
}
\label{figureonecolumnelementone}
\end{figure}

\begin{figure}[] 
\centering
\includegraphics[width=16cm]{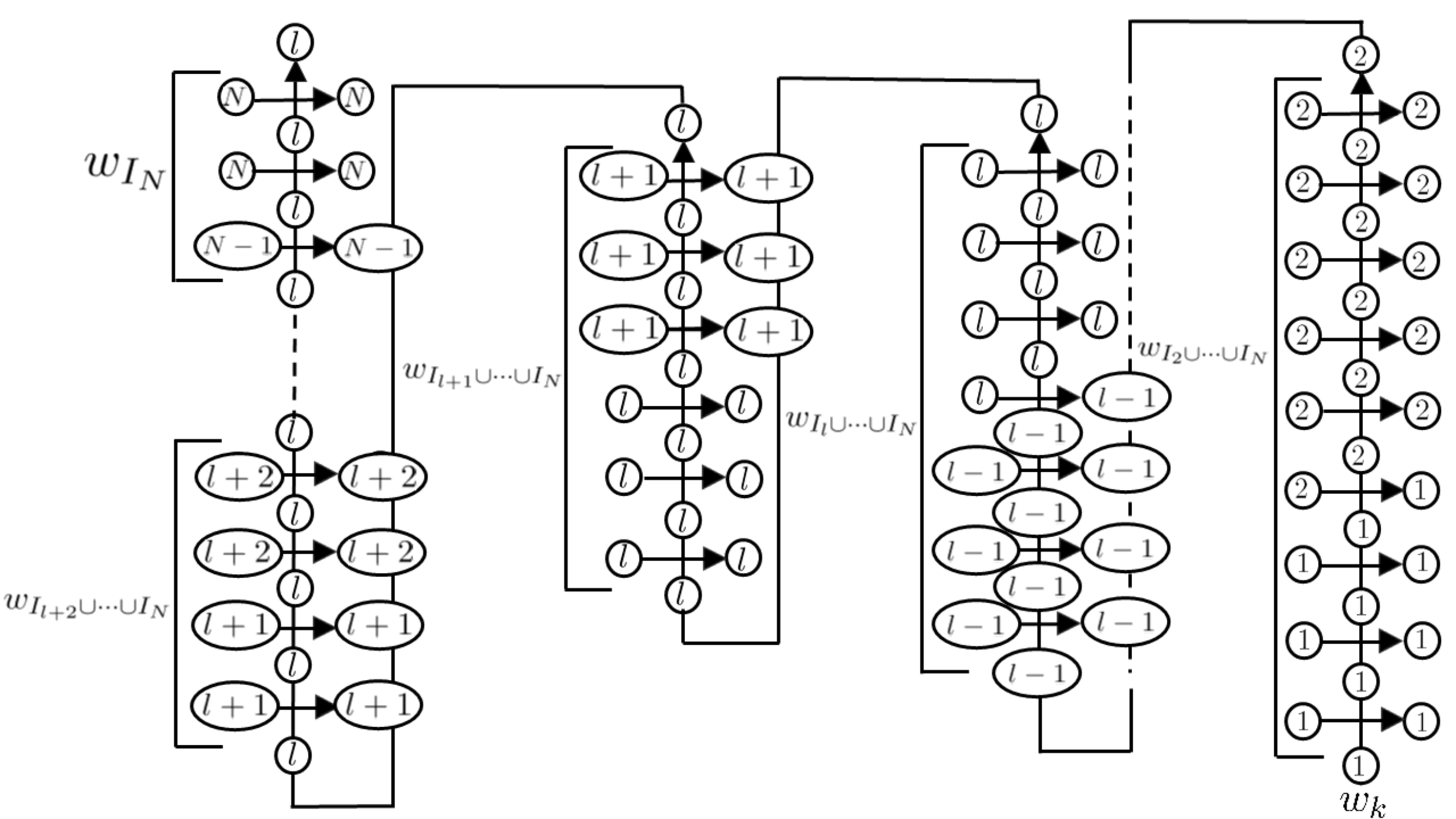}
\caption{ 
A graphical representation of the $k$-th column
of the partition function corresponding to $\tilde{X}_{II}$
or equivalently the leftmost column of the partition function $Z_k$,
when $k \in I_l$. We find in the same way as in {\it Example} that the
empty circles in Figure \ref{figureonecolumnelementone} are uniquely filled with the numbers
as given in this figure.
The product of the $R$-matrix elements for this configuration gives the factor $W_k$.
The first part from left in this figure gives $W_k^1$ \eqref{WKone}.
The second part from left gives $W_k^2$ \eqref{WKtwo}.
The remaining part gives $W_k^3$ \eqref{WKthree}.
The right boundary of this column becomes (reading the numbers in circles from bottom to top)
$\bfK_{k+1}=
(1^{n(I,k+1,2)+\cdots+n(I,k+1,N)},
2^{|I_2 \cup \dots \cup I_N|-n(I,k+1,2)},
3^{|I_3 \cup \dots \cup I_N|-n(I,k+1,3)}
,\dots,N^{|I_N|-n(I,k+1,N)}
)$.
Also note there is no $\Pi^*$-dependence on $W_k$.
}
\label{figureonecolumnelementtwo}
\end{figure}

To get the expression \eqref{coeffGZbasis},
first note
$\tX_{II}=\prod_{k=1}^n W_k
=\prod_{k=1}^n W_k^1 W_k^2 W_k^3$.
Also recall the disjoint union of $I_1,\dots,I_N$ is ${[}1, n {]}$.
Then one notes from the expressions \eqref{WKone}, \eqref{WKtwo}, \eqref{WKthree}
that $\prod_{k=1}^n W_k^1$, $\prod_{k=1}^n W_k^2$ and $\prod_{k=1}^n W_k^3$ can be expressed as
\begin{align}
\prod_{k=1}^n W_k^1&=
\displaystyle
\prod_{l=1}^{N-1}
\prod_{\substack{
j \in I_{N-l+1} \cup \cdots \cup I_{N}
\\
k \in I_{1} \cup \cdots \cup I_{N-l-1}
}
} \theta(w_j/w_k), \label{productWKone} \\
\prod_{k=1}^n W_k^2&=
\displaystyle
\prod_{l=1}^{N-1}
\Bigg\{
\prod_{\substack{
j \in I_{N-l+1} \cup \cdots \cup I_{N}
\\
k \in I_{N-l}
\\
j > k
}
}
\theta(w_j/w_k)
\prod_{\substack{
j \in I_{N-l+1} \cup \cdots \cup I_{N}
\\
k \in I_{N-l}
\\
j < k
}
}
\theta(q^2 w_j/w_k)
\Bigg\}, \label{productWKtwo}  \\
\prod_{k=1}^n W_k^3&=
\displaystyle \theta(q^2)^{\sum_{k=2}^N (k-1) |I_k|
} \prod_{l=1}^{N-1}
\Bigg\{
\prod_{\substack{
j \in I_{N-l+1} \cup \cdots \cup I_{N}
\\
k \in I_{N-l+1} \cup \cdots \cup I_{N}
\\ j < k
}
}
\theta(q^2 w_j/w_k)
\prod_{\substack{
j \in I_{N-l+1} \cup \cdots \cup I_{N}
\\
k \in I_{N-l+1} \cup \cdots \cup I_{N} \\
j > k
}
}
\theta(q^2 w_j/w_k)
\Bigg\}. \label{productWKthree} 
\end{align}
Multiplying \eqref{productWKone}, \eqref{productWKtwo} and \eqref{productWKthree}
and rearranging, one gets \eqref{coeffGZbasis}.

\end{proof}

As a corollary, we obtain $N(\bfw)=\tX_{II}/X_{II}$ as follows. 
\begin{cor}
\begin{align}
N(\bfw)
&=\theta(q^2)^{\sum_{k=2}^N (k-1) |I_k|
}
\prod_{1 \le k < l < N} \prod_{a \in I_k} \prod_{\substack{b \in I_l \\ a<b}}
\frac{\theta(q^2 w_b/w_a)}{\theta(w_b/w_a)}
\prod_{l=1}^{N-1}
\Bigg\{
\prod_{\substack{
j \in I_{N-l+1} \cup \cdots \cup I_{N}
\\
k \in I_{1} \cup \cdots \cup I_{N-l-1}
}
} \theta(w_j/w_k)
\nonumber \\
&\times
\prod_{\substack{
j \in I_{N-l+1} \cup \cdots \cup I_{N}
\\
k \in I_{N-l} \cup \cdots \cup I_{N}
\\ j < k
}
}
\theta(q^2 w_j/w_k)
\prod_{\substack{
j \in I_{N-l+1} \cup \cdots \cup I_{N}
\\
k \in I_{N-l+1} \cup \cdots \cup I_{N} \\
j > k
}
}
\theta(q^2 w_j/w_k)
\prod_{\substack{
j \in I_{N-l+1} \cup \cdots \cup I_{N}
\\
k \in I_{N-l}
\\
j > k
}
}
\theta(w_j/w_k)
\Bigg\}.
\nonumber 
\end{align}
\end{cor}

\noindent
{\it Example.}\ 
For the case $N=2$, we have
\begin{align}
N(\bfw)&
=\prod_{a \in I_1} \prod_{b \in I_2} \theta(q^2 w_b/w_a)
\prod_{a \in I_2} \prod_{b \in I_2} \theta(q^2 w_b/w_a).\nn
\end{align}

\section*{Acknowledgments}
This work was partially supported by grant-in-Aid
for Scientific Research (C) 20K03507, 21K03176, 20K03793.

\bigskip

\appendix
\setcounter{equation}{0}
\begin{appendix}

\section{Defining Relations of  $U_{q,p}(\glnhbig)$ and $U_{q,p}(\slnhbig)$}\lb{sec:UqpEqp}

\subsection{$U_{q,p}(\glnh)$}\lb{defrelUqp}
 For $g(P), g(P+h)\in 
\FF$, 
\bea
&&g({P+h})e_j(z)=e_j(z)g({P+h}),\quad g({P})e_j(z)=e_j(z)g(P-\bra Q_{\al_j},P\ket ),\lb{gegl}\\
&&g({P+h})f_j(z)=f_j(z)g(P+h-\bra {\al_j},P+h\ket  ),\quad g({P})f_j(z)=f_j(z)g(P),\lb{gfgl}\\
&&g({P})k^+_l(z)=k^+_l(z)g(P-\bra Q_{\bep_l},P\ket  ),\quad
g({P+h})k^+_l(z)=k^+_l(z)g(P+h-\bra Q_{\bep_l},P\ket  ),\nn\\
&&\lb{gkgl}
\\
&&\rho^+_+(z_2/z_1)k^+_{l}(z_1)k^+_{l}(z_2)=\rho^+_+(z_1/z_2)k^+_{l}(z_2)k^+_{l}(z_1),
\qquad  (1\leq l\leq N),\lb{kjkj}\\
&&\rho^+_+(z_2/z_1)\frac{(p^*z_2/z_1;p^*)_\infty(pq^2z_2/z_1;p)_\infty}{(p^*q^2z_2/z_1;p^*)_\infty(pz_2/z_1;p)_\infty}k^+_{j}(z_1)k^+_{l}(z_2)\nn\\
&&\qquad ={\rho}^+_+(z_1/z_2)\frac{(q^{-2}z_1/z_2;p^*)_\infty(z_1/z_z;p)_\infty}{(z_1/z_2;p^*)_\infty(q^{-2}z_1/z_2;p)_\infty} k^+_{l}(z_2)k^+_{j}(z_1)
  \quad (1\leq j<l\leq N),\lb{kjkl}
  \\
&&\frac{(p^*q^{c+2-j}z_2/z_1;p^*)_\infty}{(p^*q^{c-j}z_2/z_1;p^*)_\infty}k_j^{+}(z_1)e_j(z_2)=q^{-1}\frac{(q^{-c+j}z_1/z_2;p^*)_\infty}{(q^{-c-2+j}z_1/z_2;p^*)_\infty}e_j(z_2)k_j^{+}(z_1),\lb{kjej}
\\
&&\frac{(p^*q^{c-2-j}z_2/z_1;p^*)_\infty}{(p^*q^{c-j}z_2/z_1;p^*)_\infty}k_{j+1}^{+}(z_1)e_j(z_2)=q\frac{(q^{-c+j}z_1/z_2;p^*)_\infty}{(q^{-c+2+j}z_1/z_2;p^*)_\infty}e_j(z_2) k_{j+1}^{+}(z_1)
,
\lb{kjp1ej}\\
&&k_l^{+}(z_1)e_j(z_2)k_l^{+}(z_1)^{-1}=e_j(z_2)\qquad\qquad (l\not=j,j+1),\lb{ejkl}
\\
&&\frac{(pq^{-j}z_2/z_1;p)_\infty}{(pq^{2-j}z_2/z_1;p)_\infty}k_j^{+}(z_1)f_j(z_2)=q\frac{(q^{-2+j}z_1/z_2;p)_\infty}{(q^{j}z_1/z_2;p)_\infty}f_j(z_2)k_j^{+}(z_1),
\\
&&\frac{(pq^{-j}z_2/z_1;p)_\infty}{(pq^{-2-j}z_2/z_1;p)_\infty}k_{j+1}^{+}(z_1)f_j(z_2)=q^{-1}\frac{(q^{2+j}z_1/z_2;p)_\infty}{(q^{j}z_1/z_2;p)_\infty}f_j(z_2)k_{j+1}^{+}(z_1)
,\lb{fjkjp1}\\
&&k_l^{+}(z_1)f_j(z_2)k_l^{+}(z_1)^{-1}=f_j(z_2)\qquad\qquad (l\not=j,j+1),
\lb{fjkl}
\\
&&
z_1 \frac{(q^{2}z_2/z_1;p^*)_\infty}{(p^*q^{-2}z_2/z_1;p^*)_\infty}e_j(z_1)e_j(z_2)=
-z_2 \frac{(q^{2}z_1/z_2;p^*)_\infty}{(p^*q^{-2}z_1/z_2;p^*)_\infty}e_j(z_2)e_j(z_1),\lb{ejejgl}\\
&&z_1 \frac{(q^{-1}z_2/z_1;p^*)_\infty}{(p^*qz_2/z_1;p^*)_\infty}e_j(z_1)e_{j+1}(z_2)=
-z_2 \frac{(q^{-1}z_1/z_2;p^*)_\infty}{(p^*qz_1/z_2;p^*)_\infty}e_{j+1}(z_2)e_j(z_1),\lb{ejejp1gl}\\
&&e_j(z_1)e_l(z_2)=e_l(z_2)e_j(z_1)\qquad\qquad (|j-l|>1)\\
&&
z_1 \frac{(q^{-2}z_2/z_1;p)_\infty}{(pq^{2}z_2/z_1;p)_\infty}f_j(z_1)f_j(z_2)=
-z_2 \frac{(q^{-2}z_1/z_2;p)_\infty}{(pq^{2}z_1/z_2;p)_\infty}f_j(z_2)f_j(z_1),\lb{fjfjgl}\\
&&
z_1 \frac{(qz_2/z_1;p)_\infty}{(pq^{-1}z_2/z_1;p)_\infty}f_j(z_1)f_{j+1}(z_2)=
-z_2 \frac{(qz_1/z_2;p)_\infty}{(pq^{-1}z_1/z_2;p)_\infty}f_{j+1}(z_2)f_j(z_1),\lb{fjfjp1gl}\\
&&f_j(z_1)f_l(z_2)=f_l(z_2)f_j(z_1)\qquad\qquad (|j-l|>1)\lb{fjfl}
\\
&&[e_i(z_1),f_j(z_2)]=\frac{\delta_{i,j}\kappa}{q-q^{-1}}
\left(\delta(
q^{-c}
z_1/z_2)
k_j^-(
q^{-\frac{c}{2}}
z_1)k_{j+1}^-(
q^{-\frac{c}{2}}
z_1)^{-1}\right.\nn\\
&&\left.\qquad\qquad\qquad\qquad\qquad\qquad\qquad -
\delta(
q^c
z_1/z_2)
k_j^+(
q^{-\frac{c}{2}}
z_2)k_{j+1}^+(
q^{-\frac{c}{2}}
z_2)^{-1}
\right),\lb{eifj}
\ena
\bea
&&
\frac{(p^*q^{2}{z_{2}}/{z_{1}}; p^*)_{\infty}}
{(p^*q^{-2}{z_{2}}/{z_{1}}; p^*)_{\infty}}
\left\{\frac{(p^*q^{-1}{z_{1}}/{w}; p^*)_{\infty}}
{(p^*qz_{1}/{w}; p^*)_{\infty}}
\frac{(p^*q^{-1}{z_{2}}/w; p^*)_{\infty}}
{(p^*q{z_{2}}/w; p^*)_{\infty}}
e_j(w)e_i(z_1)e_i(z_2)\right.
\nn\\
&&\left.\qquad\qquad\qquad-[2]_q\frac{(p^*q^{-1}w/{z_{1}}; p^*)_{\infty}}
{(p^*qw/z_{1}; p^*)_{\infty}}
\frac{(p^*q^{-1}{z_{2}}/w; p^*)_{\infty}}
{(p^*q{z_{2}}/w; p^*)_{\infty}}
e_i(z_1)e_j(w)e_i(z_2)\right.
\nn\\
&&\left.
\qquad+\frac{(p^*q^{-1}w/{z_{1}}; p^*)_{\infty}}
{(p^*qw/z_{1}; p^*)_{\infty}}
\frac{(p^*q^{-1}w/{z_{2}}; p^*)_{\infty}}
{(p^*qw/{z_{2}}; p^*)_{\infty}}
e_i(z_1)e_i(z_2)e_j(w)\right\}+(z_1\leftrightarrow z_2)=0,\label{serreegl}\\
&&
\frac{(pq^{-2}{z_{2}}/{z_{1}}; p)_{\infty}}
{(pq^{2}{z_{2}}/{z_{1}}; p)_{\infty}}
\left\{\frac{(pq^{}{z_{1}}/{w}; p)_{\infty}}
{(pq^{-1}z_{1}/{w}; p)_{\infty}}
\frac{(pq^{}{z_{2}}/w; p)_{\infty}}
{(pq^{-1}{z_{2}}/w; p)_{\infty}}
f_j(w)f_i(z_1)f_i(z_2)\right.
\nn\\
&&\left.\qquad\qquad\qquad-[2]_q\frac{(pq^{}w/{z_{1}}; p)_{\infty}}
{(pq^{-1}w/z_{1}; p)_{\infty}}
\frac{(pq^{}{z_{2}}/w; p)_{\infty}}
{(pq^{-1}{z_{2}}/w; p)_{\infty}}
f_i(z_1)f_j(w)f_i(z_2)\right.
\nn\\
&&\left.
\qquad+\frac{(pq^{}w/{z_{1}}; p)_{\infty}}
{(pq^{-1}w/z_{1}; p)_{\infty}}
\frac{(pq^{}w/{z_{2}}; p)_{\infty}}
{(pq^{-1}w/{z_{2}}; p)_{\infty}}
f_i(z_1)f_i(z_2)f_j(w)\right\}+(z_1\leftrightarrow z_2)=0
\quad   |i-j|=1,\nn\\&&
\label{serrefgl}
\end{eqnarray}  
where $\delta(z)=\sum_{n\in \Z}z^n$, $\rho(z)=\rho^{+*}(z)/\rho^+(z)$, 
\bea
&&\rho^+_+(z)= \frac{\{q^2z\}^*\{q^{-2}q^{2N}z\}^*\{z\}\{q^{2N}z\}}{\{z\}^*\{q^{2N}z\}^*\{q^2z\}\{q^{-2}q^{2N}z\}},\quad \{z\}=(z;q^{2N},p)_\infty, \quad \{z\}^*=(z;q^{2N},p^*)_\infty,\nn
\ena
and
$\kappa$ is given by
\bea
&&\kappa=\frac{(p;p)_\infty(p^*q^2;p^*)_\infty}{(p^*;p^*)_\infty(pq^2;p)_\infty}.
\lb{def:kappa}
\ena   
We treat these relations 
 as formal Laurent series in $z, w$ and $z_j$'s. 
All the coefficients in $z_j$'s are well defined in the $p$-adic topology.  

\subsection{ $U_{q,p}(\slnh)$}\lb{defrelUqpslnh}
The defining relations of $U_{q,p}(\slnh)$ consists of \eqref{ellboson}, \eqref{gegl}, \eqref{gfgl}, the Serre relations \eqref{serreegl} and \eqref{serrefgl}, and the following relations. 
\bea
&&
z_1 \frac{(q^{a_{ij}}z_2/z_1;p^*)_\infty}{(p^*q^{-a_{ij}}z_2/z_1;p^*)_\infty}e_i(z_1)e_j(z_2)=
-z_2 \frac{(q^{a_{ij}}z_1/z_2;p^*)_\infty}{(p^*q^{-a_{ij}}z_1/z_2;p^*)_\infty}e_j(z_2)e_i(z_1),\lb{eesln}\\
&&
z_1 \frac{(q^{-a_{ij}
}z_2/z_1;p)_\infty}{(pq^{a_{ij}}z_2/z_1;p)_\infty}f_i(z_1)f_j(z_2)=
-z_2 \frac{(q^{-a_{ij}}z_1/z_2;p)_\infty}{(pq^{a_{ij}}z_1/z_2;p)_\infty}f_j(z_2)f_i(z_1),\lb{ffsln}\\
&&[e_i(z_1),f_j(z_2)]=\frac{\delta_{i,j}}{q-q^{-1}}
\left(\delta(
q^{-c}
z_1/z_2)
\psi_j^-(
q^{\frac{c}{2}}
z_2)-
\delta(
q^c
z_1/z_2)
\psi_j^+(
q^{-\frac{c}{2}}
z_2)
\right).\lb{eifjsln}
\ena

\end{appendix}


\end{document}